\documentclass[12pt,a4paper,oneside]{amsart}
\usepackage[utf8]{inputenc}
\usepackage{amsmath,amscd,hyperref, amsfonts, amssymb, amsthm,}
\usepackage{setspace,kantlipsum}
\usepackage{tikz-cd}
\usepackage{mathrsfs}
\usepackage{textcomp}
\usepackage[margin=1in]{geometry}
\usepackage{colonequals}
\usepackage{mathtools}
\usepackage{stackrel}
\usepackage{amssymb}

\theoremstyle{plain} 
\newtheorem{thm}{Theorem}[section]
\newtheorem{lemma}[thm]{Lemma}
\newtheorem{algorithm}[thm]{Algorithm}
\newtheorem{prop}[thm]{Proposition}
\newtheorem{corollary}[thm]{Corollary}

\theoremstyle{definition}
\newtheorem{notation}[thm]{Notation}

\newtheorem{definition}[thm]{Definition}

\newtheorem{example}[thm]{Example}
\newtheorem{remark}[thm]{Remark}
\newtheorem{problem}[thm]{Problem}

\newtheorem{question}[thm]{Question}

 \theoremstyle{plain} % just in case the style had changed
\newcommand{\thistheoremname}{}
\newtheorem{genericthm}[thm]{\thistheoremname}

\newtheorem*{genericthm*}{\thistheoremname}
\newenvironment{namedthm*}[1]
  {\renewcommand{\thistheoremname}{#1}%
   \begin{genericthm*}}
  {\end{genericthm*}}

\newcommand{\C}{\mathbb{C}}

\newcommand{\Q}{\mathbb{Q}}

\newcommand{\Z}{\mathbb{Z}}

\newcommand{\cA}{\mathcal{A}}
\newcommand{\cB}{\mathcal{B}}

\newcommand{\cH}{\mathcal{H}}

\newcommand{\cK}{\mathcal{K}}
\newcommand{\cL}{\mathcal{L}}
\newcommand{\cM}{\mathcal{M}}
\newcommand{\cN}{\mathcal{N}}
\newcommand{\cO}{\mathcal{O}}
\newcommand{\cP}{\mathcal{P}}
\newcommand{\cQ}{\mathcal{Q}}

\newcommand{\cS}{\mathcal{S}}
\newcommand{\cT}{\mathcal{T}}

\newcommand{\cV}{\mathcal{V}}

\newcommand{\cX}{\mathcal{X}}

\newcommand{\sD}{\mathscr{D}}

\renewcommand{\d}{\partial}

\newcommand{\gr}{\mathrm{gr}}
\newcommand{\IC}{\mathrm{IC}}

\renewcommand{\Im}{\mathrm{Im}}

\newcommand{\id}{\mathrm{id}}

\newcommand{\Jac}{\mathrm{Jac}}

\newcommand{\lct}{\mathrm{lct}}

\newcommand{\mef}{\tilde{\alpha}_f}
\DeclareMathOperator{\Ker}{Ker}

\DeclareMathOperator{\mult}{mult}
\newcommand{\ord}{\underset{s=-\alpha}{\textrm{ord}}}
\DeclareMathOperator{\supp}{supp}
\newcommand{\mc}[1]{{\mathcal{#1}}}
\newcommand{\mb}[1]{{\mathbb{#1}}}
\newcommand{\ms}[1]{{\mathscr{#1}}}

\newcommand{\mrm}[1]{{\mathrm{#1}}}
\renewcommand{\Re}{\mathop{\mrm{Re}}}
\renewcommand{\Im}{\mathop{\mrm{Im}}}

\title[Weight filtration and $b$-functions]{Filtrations on $\sD$-modules and multiplicities of roots of Bernstein-Sato polynomials}
\author{Andr\'as C. L\H{o}rincz and Ruijie Yang}

\begin{document}

\begin{abstract}
In this paper, we relate multiplicities of Bernstein--Sato-type polynomials with respect to a function $f$ to several singularity invariants. First, we introduce certain ``mod'' $b$-functions and show that they characterize the weight filtration on the localization of a simple algebraic holonomic $\sD$-module along $f$, and we provide an algorithm to compute them. Second, we show that they can be approximated by the multiplicities of roots of power $b$-functions (i.e., the $b$-functions with respect to powers of $f$). Furthermore, we give a sharp upper bound for the Hodge level of elements given by a certain sum of such multiplicities. Next, we give an effective asymptotic solution to the Gelfand problem by determining an explicit threshold after which every integer shift of a root of $b_f(s)$ is a pole of the Archimedean zeta function of $f$. We also show that the order of these poles is equal to the nilpotency index of the logarithmic monodromy operator, which we further express as the limit of the multiplicities of roots of power $b$-functions.
We define several filtrations, relating them to the weight and Hodge filtrations, based upon which we leave some open questions that we address in the affirmative in the case when $f$ has a homogeneous isolated singularity, or it is a hyperplane arrangement, or a semi-invariant on a spherical variety. We give several immediate applications to our results, including a solution to a question of Torrelli on whether $1/f$ lies in the intersection complex of the hypersurface of $f$ if and only if  $-1$ is a simple root of $b_f(s)$. 
\end{abstract}

\maketitle
%\tableofcontents

\section{Introduction}

Let $f$ be a non-invertible regular (algebraic) function on a smooth complex algebraic variety $X$. The \emph{Bernstein-Sato polynomial} of $f$ is the unique monic polynomial $b(s)$ of smallest degree satisfying
\[ P(s)\cdot f^{s+1} = b(s) f^s, \quad \textrm{for some $P(s)\in \sD_X[s]$}.\]
While the roots of $b(s)$ carry a lot of geometric information, their multiplicities have been considered rather mysterious until now.
In this paper, we systematically study the multiplicities of roots of various Bernstein-Sato-type polynomials and elucidate their meaning by relating them to singularity invariants of $f$. 

%\medskip

%\textbf{Note}: we work with algebraic left $\sD$-modules throughout unless otherwise specified. 

\medskip

First, let $\cM=\cS_f\neq 0$ be the localization of a simple holonomic (algebraic) $\sD$-module $\cS$ along $\{f=0\}$. We compute its weight filtration, which is induced by the mixed twistor $\sD$-module structure on $\cM$, elementwise in terms of a certain ``mod $(s+\alpha)^\ell$'' $b$-function with respect to $f^k$ for $k\gg 0$ (Corollary \ref{cor:weightb}), and we provide an algorithm to compute it. Next, we study the limit of the multiplicities of the $b$-functions of $f^k$ as $k$ goes to infinity, and we relate this limit to the $V$-filtration (Theorem \ref{thm: intro nu via higher b function}) and the weight filtration (Theorem \ref{thm: nu less than weight intro}). If $\alpha\in \Q$, we express the nilpotency index $n_{\alpha}$ of the logarithmic monodromy operator on the nearby and vanishing cycles of $\cO_X$ along $f$ in terms of the multiplicity of the $b$-function of $f^k$ for $k\gg0$ (Corollary \ref{corollary: nilpotency index for function}). This approach applies similarly to general $\sD$-modules of the form $\cM=\cS_f$ (Theorem \ref{thm:nilpotency index}). We also solve Gelfand's problem up to finitely many exceptions by determining all poles of the Archimedean zeta function $Z_f$ and their exact orders (Theorem \ref{thm: stabilization of pole order}). We also answer a question of Torrelli regarding when $1/f$ lies in the intersection complex of $\{f=0\}$ (Theorem \ref{thm: torelli}).

\subsection{Power $b$-functions and $p$-functions}
Let us begin by an elementary discussion about multiplicities of $b$-functions. Let $\cM$ be a holonomic $\sD$-module on which $f$ acts bijectively. For an element $m\in \cM$ and an integer $k\in \Z_{\geq 1}$, we define the \emph{$k$-th power $b$-function}, denoted by $b_m^{(k)}(s)$, to be the unique monic polynomial in $\C[s]$ of minimal degree for which there exists an operator $P\in \sD[s]$ satisfying 
\begin{equation}\label{eqn: function eqn for higher b function} 
P\cdot mf^{s+k}=b_m^{(k)}(s)mf^s.
\end{equation}
While it is easy to see that $b^{(k)}_{m}(ks)$ agrees with the Bernstein-Sato polynomial of $m$ with respect to $f^k$ (see Definition \ref{definition: definition of BS polynomials}), up to a constant, we intentionally choose this formulation to track the behavior of multiplicities for a fixed root. For instance, in Lemma \ref{lem:bfunpower}, we show that
\begin{equation}
 \label{eqn: bfunctionpower intro}
        \mathrm{lcm}\left(\, b_m(s), b_m(s+1), \dots, b_m(s+k-1)\, \right) \, \mid \, b_m^{(k)}(s) \mid b_m^{(k+1)}(s) \mid \prod_{i=0}^k b_m(s+i).
\end{equation}
It immediately follows that for all $k \geq 2$, the roots of $b_m^{(k)}(s)$ are completely determined by the roots of the standard $b$-function $b_m(s)=b_m^{(1)}(s)$; this is also implicit in \cite[Theorem 4.6]{Budur15}.

The multiplicities of these roots, however, turn out to be rather interesting invariants. From the divisibility relations in \eqref{eqn: bfunctionpower intro}, we observe that for a fixed $\alpha\in \C$, the sequence $\{\mult_{s=-\alpha}b_m^{(k)}(s)\}_{k\geq 1}$ is non-decreasing and bounded above a constant depending only on $b_m(s)$ and $\alpha$. Consequently, $\lim_{k\to \infty} \mult_{s=-\alpha}b_m^{(k)}(s)$ must exist and we set
\[\nu_{m,\alpha}\colonequals \lim_{k\to \infty}\mult_{s=-\alpha}b^{(k)}_m(s).\]

Our first main result is to reinterpret $\nu_{m,\alpha}$ in terms of the $V$-filtration. Fix $\alpha\in \C$ and $m\in \cM$. Let $\iota:X\to X\times \C$ be the graph embedding of $f$ sending $x$ to $(x,f(x))$. The \emph{$p$-function} $p_{m,\alpha}(s)$ is defined as the unique monic polynomial in $\C[s]$ of minimal degree satisfying
\[ p_{m,\alpha}(s)mf^s\in V^{\alpha}\iota_{+}\cM,\]
where $V^{\bullet}\iota_{+}\cM$ is the Kashiwara-Malgrange $V$-filtration on $\iota_{+}\cM$ along the coordinate function on $\C$ (see \S \ref{sec: Vfiltration}). Intuitively, because $mf^s\in V^{\beta}\iota_{+}\cM$ for some $\beta$, this polynomial measures precisely how far the element $mf^s$ is from landing inside $V^{\alpha}\iota_{+}\cM$. This notion first appeared in \cite[Definition 2.11]{LY25April} and we give a systematic treatment in \S \ref{sec: pfunction}. Among them, we show that $p_{m,\alpha}(s)$ exists (Lemma \ref{lem:pideal}) and that $p_{m,\alpha}(-\alpha)\neq 0$. In Corollary \ref{cor:explicitpfunction} and Theorem \ref{thm: numalpha via higher order pfunction}, we establish the following relationships between $\nu_{m,\alpha}$ and $p$-functions.

\begin{thm}\label{thm: intro nu via higher b function}
For $\alpha\in \C$ and $0\neq m\in \cM$, we have an equality 
\[ \nu_{m,\alpha}=\mult_{s=-\alpha}b_{p_{m,\alpha}(s)mf^s}(s),\] 
where the limit stabilizes for $k\geq \min\{ k_0\in \Z_{>0} \mid b_m(-\alpha+K)\neq 0 \textrm{ for every $K\geq k_0$}\}$. Conversely, endowing $\mb{C}$ with the lexicographic ordering, we have a finite product
    \[ p_{m,\alpha}(s)=\prod_{\beta<\alpha}(s+\beta)^{\nu_{m,\beta}}.\]
\end{thm}

From now on, let us assume $\cM=\cS_f\neq 0$, the localization of $\cS$ along $\{f=0\}$, where $\cS$ is a simple holonomic $\sD$-module. While the $p$-function plays a critical role in explicitly computing the $V$-filtration in the equivariant setup of \cite{LY25April}, we show that in general it serves as a powerful approximation for the weight level of elements in $\cM$, as well as the Hodge levels if $\cS$ underlies a pure Hodge module.

We first explain how to put a weight filtration on $\cM=\cS_f$. Recall that $(\cO_X)_f$ underlies a mixed Hodge module on $X$ by Saito's theory \cite{Saito90}. It thus carries a weight filtration, which is traditionally of great interest for computing Deligne's weight filtration on the singular cohomology of the open complement $X\setminus\{f=0\}$ \cite{Deligne}. The theory of twistor $\sD$-modules of Sabbah \cite{Sabbah05} and Mochizuki \cite{Mochizukiwild} allows us to work in the generality where $\cS$ is a simple holonomic $\sD$-module. By the foundational work \cite{Mochizukiwild}, $\cS$ underlies a polarizable, purely imaginary (wild) pure twistor $\sD$-module on $X$. For any $\alpha\in \C$, in \S \ref{sec: Malgrange transform} we show that $\cM\cdot f^{-\alpha}$ underlies a mixed twistor $\sD$-module in the sense of \cite{MochizukitwistorDmodule} and hence it is equipped with a canonical weight filtration (up to shifts) whose associated graded pieces are semi-simple $\sD$-modules, by Mochizuki's solution of Kashiwara's conjecture and Lemma \ref{lemma: grWM in terms of grV}. Note that if $\cS$ underlies a (complex) pure Hodge module, then this weight filtration is determined by the theory of (complex) mixed Hodge modules \cite{MHMproject,Saito90}.

Now, let us explain how to approximate weight levels of elements in $\cM\cdot f^{-\alpha}$ using the numbers $\nu_{m,\alpha}$. This connection emerges naturally from our attempt to compute the weight filtration on $\cM\cdot f^{-\alpha}$ via the $V$-filtration. More precisely, evaluation at $s=-\alpha$ induces a surjective map
\[ ev_{s=-\alpha} \colon \iota_{+}\cM \xrightarrow{\sim} \cM[s]f^s \xrightarrow{s=-\alpha}\cM\cdot f^{-\alpha}.\]
The first $\sD_X[s]$-isomorphism was observed by Malgrange \cite{Malgrange} (see also Musta\c{t}\u{a}--Popa \cite{MPVfiltration}). In Theorem \ref{thm: ses of Valpha and Mfalpha}, we show that this map restricts to a surjective map
\begin{equation}\label{eqn: from V to twist} ev_{s=-\alpha} \colon V^{\alpha}\iota_{+}\cM \to \cM\cdot f^{-\alpha},\end{equation}
which admits a set-theoretic section given by the $p$-function: $mf^{-\alpha}\mapsto \frac{p_{m,\alpha}(s)}{p_{m,\alpha}(-\alpha)}mf^s$.

For any element $w\in V^{\alpha}\iota_{+}\cM$, Sabbah's Theorem \cite{Sabbah} ensures that all roots of $b_w(s)$ are bounded above by $-\alpha$. This property naturally gives rise to a filtration of coherent $\sD_X\langle s,t\rangle$-modules on $V^{\alpha}\iota_{+}\cM$ (Proposition \ref{prop: Znalpha is full Z}):
\[ Z_{\ell}V^{\alpha}\iota_{+}\cM\colonequals\{w\in V^{\alpha}\iota_{+}\cM \mid \mult_{s=-\alpha}b_w(s)\leq \ell\}, \quad \textrm{if $\ell\in \Z_{\geq 0}$},\]
see \cite{dakin} for a related discussion. Using this, we prove the following result.

\begin{prop}[{Theorem \ref{thm:Zfiltexact}}]\label{thm: Z and W intro}For any $\ell \in \Z_{\geq 0}$, we have an exact sequence
\[ 0 \to Z_{\ell+1}  V^{\alpha} \iota_+ \cM \xrightarrow{s+\alpha} Z_\ell V^\alpha \iota_+ \cM \xrightarrow{ev_{s=-\alpha}} W_{q+ \ell}(\cM \cdot f^{-\alpha}) \to 0, \]
where $W_{\bullet}$ is the weight filtration  induced by the mixed twistor $\sD$-module structure on $\cM\cdot f^{-\alpha}$ and $\cS$ is pure of weight $q$.
\end{prop}
In particular, denote by $\omega_{m,\alpha}$ the \emph{weight level} of $mf^{-\alpha}$ in $\cM\cdot f^{-\alpha}$, i.e., the smallest integer $\ell$ such that $mf^{-\alpha}\in W_{q+\ell}(\cM  f^{-\alpha})\setminus W_{q+\ell-1} (\cM f^{-\alpha})$. Combined with Theorem \ref{thm: intro nu via higher b function}, we have the following approximation.

\begin{thm}\label{thm: nu less than weight intro}
For any $m\in \cM$, one has $\omega_{m,\alpha}\leq \nu_{m,\alpha}$. 
\end{thm}

Although in many cases we find that $\omega_{m,\alpha}= \nu_{m,\alpha}$ (see \S \ref{sec: examples} and (\ref{eq:nunilp})), they do not agree in general. One subtle reason for this discrepancy is that, while the set $\{m f^{-\alpha}\mid \nu_{m,\alpha}\leq \ell\}$ always forms an $\cO_X$-submodule of the $\sD$-module $W_{q+\ell}(\cM\cdot f^{-\alpha})$ (Proposition \ref{prop:restrnu}), Lemma \ref{lem:counterex} provides examples where it fails to be a $\sD$-submodule. Nevertheless, we can ask whether this equality holds \lq\lq generically", after resolving this discrepancy. 

\begin{question}\label{que: weight filtration on Mf-alpha in terms of nu}
Do we always have $W_{q+\ell}(\cM\cdot f^{-\alpha})=\sD_X\cdot \{mf^{-\alpha}\mid \nu_{m,\alpha}\leq \ell\}$?
\end{question}

We answer this in the affirmative for certain values of $\ell$ (e.g. see Proposition \ref{prop:lastweight}) and for some specific cases, such as when $f$ has an isolated quasi-homogeneous singularity, or it is a hyperplane arrangement (and $\cS=\cO_X$), or a semi-invariant on a spherical variety (see \S \ref{sec: examples}). In these examples we illustrate how the mechanism of our invariants $\nu$ can be used effectively to calculate weight levels of various elements, such as powers of $f$.

\medskip

Suppose $\cS$ underlies a (complex) pure Hodge module \cite{MHMproject}, then $\cM \cdot f^{-\alpha}$ underlies a (complex) mixed Hodge module for $\alpha\in \mathbb{R}$ (see \cite[\S 5]{DY26} for example) and hence it carries a canonical Hodge filtration $F_{\bullet}$ starting at degree $0$. In the case $\cS=\cO_X$, we find that the Hodge level of elements in $\cM\cdot f^{-\alpha}$ can be approximated using the \emph{degree} of the $p$-function.

\begin{thm}
  Assume $\cS=\cO_X$, then for any $m\in \cM$ and $\alpha\in \mathbb{R}$, we have 
  \[ m f^{-\alpha} \in F_{\deg p_{m,\alpha}(s)}(\cM\cdot f^{-\alpha}).\]
  In other words, the Hodge level of $m f^{-\alpha}$ is at most $\deg p_{m,\alpha}(s)=\sum_{\beta < \alpha} \nu_{m, \beta}$.
\end{thm}

In Section \ref{sec: examples}, we demonstrate that this containment is sharp in a number of cases, i.e.  the Hodge level of $m f^{-\alpha}$ is equal to $\deg p_{m,\alpha}(s)$.

\subsection{Weight filtration and (mod) $b$-functions}
 Let $\cS$ be a simple holonomic $\sD$-module and $\cM=\cS_f$. By the discussion above, $\cS$ underlies a pure twistor $\sD$-module, which we assume has weight $q$ and hence $\cM\cdot f^{-\alpha}$ is equipped with a weight filtration. Via Beilinson's glueing formula, Saito \cite{Saito90} (in the Hodge case) and Mochizuki \cite{MochizukitwistorDmodule} showed that this weight filtration is ultimately controlled by the monodromy weight filtration on the nearby and vanishing cycles of $\cM$. 
 
The next result of our work captures the weight level of a given element in $\cM \cdot f^{-\alpha}$ via a relaxed functional equation utilizing $b$-function-type invariants, which is inspired by Sabbah's characterization of the $V$-filtration \cite{Sabbah} via $b$-functions of elements.

Since $f$ acts bijectively on $\cM$, there is a natural $\sD_X[s]$-module structure on $\cM[s]f^s$ \cite{Malgrange} (see also \cite{Kas76}). For any element $m\in \cM$ and each integer $k > 0$, we define $\ell_{m,k}$ as the minimal non-negative integer $\ell$ such that there exists some $P\in \sD_X[s]$ satisfying the congruence
\begin{equation}\label{eqn:functioneqn mods+alpha} 
P\cdot mf^{s+k}\equiv (s+\alpha)^{\ell}mf^s \pmod{(s+\alpha)^{\ell+1}\cdot \cM[s]f^s}.
\end{equation}
The sequence $\ell_{m,k}$ stabilizes (Proposition \ref{prop: omega malpha exists}), and considering the limit we obtain the following description of the weight level (recall that $\omega_{m, \alpha}$ is the weight level of $mf^{-\alpha}$)
\begin{equation}
\lim_{k\to \infty} \ell_{m,k} = \omega_{m, \alpha}.    
\end{equation}
More precisely, we have the following (see Theorem \ref{thm:Bnice} and Remark \ref{remark: suffices to choose m'}).
\begin{thm}\label{thm: remark5.6 intro}
     Take any $0\neq \tilde{m} \in \cM$ such that $b_{\tilde{m}}(s)$ has no roots in $-\alpha+\Z_{\geq 0}$. Then $\omega_{m, \alpha}$ is the minimal integer $\ell \geq 0$ satisfying
    \[P\cdot \tilde{m}f^{s} = (s+\alpha)^\ell mf^s \pmod{(s+\alpha)^{\ell+1}\cdot \cM[s]f^s}\]
    for some $P\in \sD_X[s]$. In particular,
    \begin{equation}
    \omega_{m,\alpha}=\ell_{m, k}, \,\, \mbox{ for any } k> \max\bigl\{j\in\Z_{\geq0}\mid b_m(-\alpha+j)=0\bigr\}.
     \end{equation}
\end{thm}
Based on this, we can deduce the following characterization of the weight filtration.
\begin{corollary}\label{cor:weightb}
For any $\ell \in \Z$ and $\alpha \in \C$ we have a set-theoretic equality
\[ W_{q+\ell}(\cM\cdot f^{-\alpha})=\{ mf^{-\alpha} \mid \ell_{m, k} \leq \ell, \mbox{ for  all } k\in \Z_{> 0}\}.\] 
\end{corollary}
 Because our characterization is entirely algebraic, we provide an algorithm (\S \ref{sec: algorithm}) based on Theorem \ref{thm: remark5.6 intro} to determine $\omega_{m,\alpha}$, which has been implemented in the computer algebra system Macaulay2 \cite{M2}, see \cite{jsag}. 

As a first illustration of the general theory, we readily reveal some deeper consequences starting from the following simple identity:
 \[\partial_x \cdot f^s = s \cdot \frac{1}{f} \frac{\partial f}{\partial x} \cdot f^{s}.\]
By Theorem \ref{thm: remark5.6 intro}, this identity immediately yields $\omega_{m, 0} \leq 1$ for the element $m\colonequals \frac{1}{f} \frac{\partial f}{\partial x}$. Furthermore, if $m\not\in \cO_X$, then $\omega_{m, 0}=1$. This is already powerful: it immediately recovers Saito's result $\mathrm{IC}_{\{f=0\}}=W_{\dim X+1}(\cO_X)_f/\cO_X$ \cite[(4.5.9)]{Saito90}, and provides a simple argument for the classical Barlet--Kashiwara theorem \cite[Theorem 1.1]{BK86} in the hypersurface case, which states that the class $[m]$ lies in $\mathrm{IC}_{\{f=0\}}$.

We finish this section by a comparison with Theorem \ref{thm: nu less than weight intro}: even though the element $\frac{p_{m,\alpha}(s)}{p_{m,\alpha}(-\alpha)}mf^s$ lifts $m f^{-\alpha}$ via the map in \eqref{eqn: from V to twist}, in general it may not serve as the correct lift of $m f^{-\alpha}\in W_{q+\omega_{m,\alpha}}(\cM\cdot f^{-\alpha})$ under the sequence in Theorem \ref{thm: Z and W intro}. The key insight motivating the definition of $\omega_{m,\alpha}$ is that we do not actually need the exact functional equation \eqref{eqn: function eqn for higher b function}; rather, it suffices to consider its ``modulo $(s+\alpha)^{\ell}$'' variant \eqref{eqn:functioneqn mods+alpha}. It is precisely this relaxed condition that captures the weight level of $mf^{-\alpha}$.

Though we have a precise algorithm determining the weight level $\omega$, the mod-equations yielding them are conceptually more mysterious than the classical Bernstein--Sato equations yielding $\nu$, and the algorithmic determination of $\omega$ often takes longer than the invariants $\nu$. Moreover, the invariants $\nu$ are more accessible and rigid in the sense that they are bound to some inequalities that make them easier to control (e.g. see Proposition \ref{prop:nurecursivebound}, and examples in Section \ref{sec: examples}). Thus, we believe Question \ref{que: weight filtration on Mf-alpha in terms of nu} to be particularly relevant.

\subsection{Applications to singularity invariants}
In the remaining sections, we deduce several applications to classical invariants of singularities. Let $f:X\to \C$ be a non-invertible regular function on a smooth algebraic variety $X$.

\subsubsection{Nilpotency index}
Let $\cS$ be a simple holonomic $\sD$-module such that $\cM\colonequals \cS_f\neq 0$, and fix $\alpha\in \C$. A fundamental invariant of singularities is the \emph{nilpotency index} $n_{\alpha}$, defined as the smallest integer $n$ such that $(s+\alpha)^n\cdot \gr^{\alpha}_V\iota_{+}\cM=0$. While this invariant has been extensively studied in the literature (see e.g. \cite{Dimca}), a major computational difficulty remains: even in the simplest geometric setting where $\cS=\cO_X$ and $f$ has isolated singularities, there is no general, explicit formula for $n_{\alpha}$.

We solve this problem by expressing $n_{\alpha}$ via multiplicities of $b$-functions. Recall from Theorem \ref{thm: intro nu via higher b function} that the invariant $\nu_{m,\alpha}$ arises as the stabilization of these multiplicities. We also have the monotonicity relation $\nu_{m,\alpha}\leq \nu_{m,\alpha+1}$ (Lemma \ref{lemma: basic property of pfunction}). On the other hand, \cite[Corollary 2.7]{DLY} shows that for any $w\in V^{\alpha}\iota_{+}\cM$, the multiplicity is given by
\[ \mult_{s=-\alpha}b_w(s)=\min\{\ell \mid (s+\alpha)^{\ell}[w]=0\in \gr^{\alpha}_V\iota_{+}\cM\}.\]
Applying this to $w=p_{m,\alpha}(s)mf^s$ and using the isomorphism $\gr^{\alpha+1}_V\iota_{+}\cM\cong \gr^{\alpha}_V\iota_{+}\cM$, the sequence $\{\nu_{m,\alpha+k}\}_{k\in \Z_{\geq 0}}$ is uniformly bounded from above by $n_{\alpha}$. In particular, the limit $\lim_{k\to \infty}\nu_{m,\alpha+k}$ must exist. It turns out this limit recovers the nilpotency index.

\begin{thm}[{Theorem \ref{thm: stabilization of nu and omega}}]\label{thm:nilpotency index}
For any $0\neq m\in \cM$, we have
\[ n_{\alpha}=\lim_{k\to \infty} \nu_{m,\alpha+k},\]
where the limit stabilizes for $k\geq \max\{ k_0\in \Z\mid b_m(-\alpha-k_0)=0\}$.
\end{thm}
As a corollary, when $\cS=\cO_X$, we can compute $n_{\alpha}$ explicitly in terms of \emph{two} $b$-functions. Note that in this case,  we can choose $\alpha\in \Q$, and $n_{\alpha}$ coincides with the nilpotency index of the log monodromy operator on the nearby cycles $\psi_{f,e^{-2\pi i\alpha}}\Q$.  Below, let $b_g(s)$ denote the standard Bernstein-Sato polynomial for a function $g$.

\begin{corollary}\label{corollary: nilpotency index for function}
Let $-\beta$ be the smallest and $-\gamma$ the largest root of $b_f(s)$ in $-\alpha+\Z$. Then 
\begin{equation}\label{eqn: nalpha for f via higher b} n_{\alpha}=\underset{s=-\beta/k}{\mult}b_{f^k}(s), \quad \forall k\geq \beta-\gamma+1.\end{equation}
\end{corollary}

If $f$ has isolated singularities, \eqref{eqn: nalpha for f via higher b} computes the maximal size among all Jordan blocks for the $e^{-2\pi i\alpha}$-eigenspaces of the local monodromy on the cohomology of the Milnor fiber. Furthermore, Corollary \ref{corollary: nilpotency index for function} immediately gives the following bounds for $n_\alpha$ in terms of $b_f(s)$ (see Corollary \ref{cor: bound the nilpotency index by multiplicity}), based on the sharper bounds in Proposition \ref{prop:nurecursivebound}: 
\begin{equation}\label{eqn: upper and lower bound of nalpha}\mult_{s=-\alpha}b_f(s)\leq n_\alpha \leq \sum_{\alpha'\in \, \alpha+\Z} \mult_{s=-\alpha'}b_f(s),\end{equation}
which might be known to some experts; for example, it can be also deduced from Saito's $G$-filtration \cite{Saito94}. The upper bound in \eqref{eqn: upper and lower bound of nalpha} is sharp and can be achieved—for example, if $f=\mathrm{det}$ on the space $X$ of square matrices and $\alpha=0$. However, we can establish a \emph{strict} lower bound whenever $\alpha\in \Z_{\geq 2}$. 

\begin{prop}[{Proposition \ref{prop:strictnumu}}]\label{prop: strict lower bound}
    If $\alpha\in \Z_{\geq 2}$, then $\nu_{1,\alpha}>\mult_{s=-\alpha}b_f(s)$. 
\end{prop}
In particular, since $n_{\alpha}\geq \nu_{1,\alpha}$ by Theorem \ref{thm:nilpotency index}, we have $n_\alpha>\mult_{s=-\alpha}b_f(s)$.

To wrap up the discussion, we note a similar stabilization phenomenon for the weight levels. By Proposition \ref{prop:lastweight}, the length of the weight filtration on $\cM\cdot f^{-\alpha}$ coincides with the nilpotency index $n_{\alpha}$, thus we have $\omega_{m,\alpha}\leq n_{\alpha}$. The weight level of  $mf^{-\alpha-k}$ inside $\cM\cdot f^{-\alpha-k}$ also stabilizes, i.e.,
\[ \lim_{k\to \infty} \omega_{m,\alpha+k}=n_{\alpha},\]
with the limit achieved at the same $k\geq \max\{ k_0\in \Z\mid b_m(-\alpha-k_0)=0\}$ (Theorem \ref{thm: stabilization of nu and omega}). Further, in the maximal case we do have equality between $\nu$ and $\omega$ (see Proposition \ref{prop:lastweight}):
\begin{equation}\label{eq:nunilp}
\omega_{m, \alpha} = n_\alpha  \,\,  \mbox{ if and only if } \,\, \nu_{m, \alpha} = n_\alpha.
\end{equation}
\subsubsection{Archimedean zeta function}
Suppose $X\subseteq \C^n$ is an open domain, and $f$ holomorphic on $X$. Recall from \cite{DLY} the Archimedean zeta function $Z_f$ is defined as the distribution that maps a compactly supported smooth function $\varphi$ on $\C^n$ (with support contained in $X$) to $Z_f(\varphi;s)$, to the meromorphic extension of 
\[ s\mapsto \int_X |f|^{2s}\varphi \,\, d\mu(x), \quad \Re(s)>0,\]
where $d\mu(x)$ is the Lebesgue measure. A value $s_0$ is a \emph{pole} of $Z_f$ if it is a pole of $Z_{f}(\varphi;s)$ for some test function $\varphi$. For any $\alpha\in \Q$, the pole order of $Z_f$ at $s=-\alpha$ is defined as $\ord Z_f \colonequals \underset{\varphi}{\max}  \ord Z_f(\varphi;s)$. A classical result of Bernstein \cite{bernstein} states that for every pole $-\alpha$ of $Z_f$ with $\alpha\in \Q_{>0}$, one has
\begin{equation}\label{eqn: Bernstein pole order bound} \ord Z_f\leq \underset{s=-\alpha}{\mult}\prod_{i=0}^rb_f(s+i),\quad \textrm{for $r\gg 0$}.\end{equation}
Conversely, Barlet \cite{Barlet84} proved that for any root $-\alpha$ of $b_f(s)$, there exists some integer $k\in \Z_{\geq 0}$ such that $-\alpha-k$ is a pole of $Z_f$.

One of the remaining difficulties is to explicitly determine the shift $k$ and the corresponding pole orders, which pertain to Gelfand's original problem \cite{gelfand}. In recent joint work with Davis \cite[Theorem 1.5]{DLY}, the authors established an effective bound for $k$ using the theory of mixed Hodge modules. However, since this bound depends on the Hodge filtration on $\gr^{\alpha}_V\iota_{+}\cO_X$, it is not so easy to compute in practice. Furthermore, outside the case of the minimal exponent, the precise pole orders remained undetermined. 

We overcome this difficulty and determine the precise shifts as well as pole orders, depending only on explicit data from $b_f(s)$. This effectively solves Gelfand's problem up to finitely many (explicit) exceptions. First, in Lemma \ref{lemma: pole order bound by nu}, we improve the bound \eqref{eqn: Bernstein pole order bound} to the following, in view of \eqref{eqn: bfunctionpower intro},
\begin{equation}\label{eqn: a priori bound}\ord Z_f\leq \nu_{1,\alpha}\leq n_{\alpha}.\end{equation}
Furthermore, if $-\alpha$ is a pole of $Z_f$, then $-\alpha-1$ is also a pole of $Z_f$ with $\underset{s=-\alpha}{\mathrm{ord }}Z_f \leq \underset{s=-\alpha-1}{\mathrm{ord }}Z_f$. It follows that the sequence $\{\mathrm{ord}_{s=-\alpha-k}Z_f\}_{k\in \Z_{\geq 0}}$ has a limit. We determine when this limit is achieved and compute its value.

\begin{thm}\label{thm: stabilization of pole order}
    Let $\alpha \in \Q$ such that $b_f(-\alpha)=0$. If $k\geq \max\{ k_0\in \Z_{\geq 0}\mid b_f(-\alpha-k_0)= 0\}$, then $-\alpha-k$ is a pole of $Z_f$ and 
    \[ \underset{s=-\alpha-k}{\mathrm{ord }}Z_f=\nu_{1,\alpha+k}=n_{\alpha}.\]
    In particular, we have $\lim_{k\to \infty} \, \underset{s=-\alpha-k}{\mathrm{ord }}Z_f=n_{\alpha}$.
\end{thm}
If $f$ has only isolated singularities, a similar relation between pole orders and nilpotency index was previously obtained by Barlet–Maire asymptotically \cite{BM00}.  
Furthermore, for arbitrary $\alpha$ we improve \eqref{eqn: a priori bound} in terms of the weight level $\omega_{1,\alpha}$, see Proposition \ref{prop:w=v}. 

To wrap up, in the case of quasi-homogeneous polynomial of isolated singularities, pole orders of $Z_f$ are studied by Loeser \cite{Loeser}. Building on his work and using our mechanism of $\nu$, we can completely determine the pole orders in this case. 
\begin{prop}[Corollary \ref{cor: weight level}, Corollary \ref{cor:quasihomogeneous-zeta-poles}]
If $f$ is a quasi-homogeneous polynomial with isolated singularities, then for every $\alpha\in\mathbb Q_{>0}$ we have 
\[ \operatorname{ord}_{s=-\alpha} Z_1=\nu_{1,\alpha}=\omega_{1,\alpha}.\]
Furthermore, we have
\[\omega_{1,\alpha} = \begin{cases}
        0, \mbox{ if } b_f(s) \mbox{ has no roots in } -\alpha+\Z_{\geq 0}; \\
        1,  \mbox{ if either } \alpha\in\Z_{\geq 1} \mbox{ and } \tilde{b}_f(s) \mbox{ has no roots in } -\alpha+\Z_{\geq 0}, \\
        \quad \mbox{ or } \alpha \notin \Z \mbox{ and } b_f(s) \mbox{ has roots in } -\alpha+\Z_{\geq 0}; \\
        2,  \mbox{ if } \alpha\in \Z_{\geq 1} \mbox{ and } \tilde{b}_f(s) \mbox{ has a root in } -\alpha+\Z_{\geq 0}. \\
    \end{cases}\]
\end{prop}
We want to emphasize that the bound given by $\nu$ in (\ref{eqn: a priori bound}) significantly improves Bernstein's bound \eqref{eqn: Bernstein pole order bound} on order of poles. In the case of quasi-homogeneous isolated singularities, when $\mult_{s=-\alpha} b_f(s)=1$ for some $\alpha\in \Z_{\geq 2}$ and $\mult_{s=-1}b_f(s)=2$, then \eqref{eqn: Bernstein pole order bound} only gives $\ord Z_f\leq 3$, but by above $\ord Z_f=\nu_{1,\alpha}=2$ and also $n_0=2$ by Proposition \ref{prop: isoquasi}. 
\subsubsection{Further applications}
Let $f$ be an algebraic function on a smooth complex algebraic variety $X$ defining an irreducible hypersurface $D\colonequals \mathrm{div}(f)$. We denote by $b_f(s)$ the Bernstein-Sato polynomial of $f$.

\medskip
A classical problem is to understand the explicit algebraic structure of the intersection $\sD$-module $\mathrm{IC}_D$. While abstract $\sD$-module characterizations of $\mathrm{IC}_D$ exist, determining precisely which  elements—such as the class of $f^{-1}$—actually belong to $\mathrm{IC}_D$ remains a  nontrivial problem. 

We apply our $b$-function invariants to overcome this difficulty. First, let $x$ be a coordinate direction on $X$ such that the element $m=\frac{\partial f}{\partial x}\cdot f^{-1}$ does not lie in $\cO_X$. We show (Proposition \ref{prop: IC is weight 1}) that $\nu_{m,0}=1$, which implies, in particular, again that $[m] \in \mathrm{IC}_D$.  It is natural to push the method further to determine exactly when the class of $f^{-1}$ itself lies in $\mathrm{IC}_D$. Torrelli \cite[Remark 4.4]{Torrelli} asked whether one has an equivalence 
\begin{equation}\label{eqn: Torrelli equivalence} \mult_{s=-1} b_f(s)=1 \Longleftrightarrow [f^{-1}] \in \mathrm{IC}_D.\end{equation}
In \textit{loc. cit.}, Torrelli proved \eqref{eqn: Torrelli equivalence} under the assumption that $\sD_X\cdot f^{-1}=(\cO_X)_f$. We resolve his question by giving criterion for when the equivalence holds and does not hold. Furthermore, we significantly generalize his result (see Remark \ref{remark: recovering Torrelli}).
\begin{thm}\label{thm: torelli}
With the setup above, the following hold.
\begin{enumerate}
\item[(a)] The $``\Rightarrow"$ part of \eqref{eqn: Torrelli equivalence} always holds: if $\mult_{s=-1} b_f(s)=1$, then $[f^{-1}] \in \mathrm{IC}_D$.

\item[(b)] Assume either that $\mathrm{ord}_{s=-1} Z_f = \mult_{s=-1} b_f(s)$ or that the nilpotency index satisfies $n_1=\mult_{s=-1}b_f(s)$, then the $``\Leftarrow"$ part of \eqref{eqn: Torrelli equivalence} holds, i.e.,
\begin{equation*}\label{eqn: 1/f implies mult of -1}
\mult_{s=-1} b_f(s)>1 \Longrightarrow [f^{-1}] \notin \mathrm{IC}_D.
\end{equation*}
\item[(c)] There exists a polynomial $h$ such that the $``\Leftarrow"$ part of \eqref{eqn: Torrelli equivalence} does not hold. 

\noindent More precisely, let $X=\C^n$ and $g(y)$ be a polynomial on $\C^m$ with variables disjoint from $f(x)$. Suppose $h(x,y)=f(x)+g(y)$ is an irreducible polynomial on $\C^{n+m}$, and  that the following is satisfied (e.g. if $g$ is quasi-homogeneous):
\begin{equation*} \mrm{Sing}(h^{-1}(0))=\mrm{Sing}(f^{-1}(0))\times \mrm{Sing}(g^{-1}(0)).\end{equation*}
and assume that there exists a unique pair of roots of $b_f(s)$ and $b_g(s)$, respectively, whose sum equals $-1$, say $(-c,-(1-c))$ with $c\in (0,1)$, and that both roots are simple. Then $\mult_{s=-1}b_h(s)=2$ and
\[ f^{-c}\in \sD_X\cdot f^{1-c}\Longrightarrow [h^{-1}]\in \mathrm{IC}_{h=0}.\]
 \end{enumerate}
\end{thm}
In particular, using (c) we produce a counterexample to the equivalence \eqref{eqn: Torrelli equivalence}, see Example \ref{ex: coex Torrelli}. Furthermore, Theorem \ref{thm: torelli}(b) generalizes some known results in the literature. Assume that $-1$ is the largest root of $b_f(s)$. Then $\lct(f)=1$, and it is known in this case that $\mathrm{ord}_{s=-1} Z_f = \mult_{s=-1} b_f(s)$ (see \cite[Theorem 1.1]{DLY} for example, and the discussion below it). Hence, it follows from Theorem \ref{thm: torelli}(b) that the $``\Leftarrow"$ part of \eqref{eqn: Torrelli equivalence} holds. The $\lct(f)=1$ case can be seen alternatively  by combining results in \cite{Saitorational}, \cite{MP16}, \cite{Olano}.

\smallskip

In order to prove part (c) above, in \S\ref{sec:thomseb} we develop several Thom--Sebastiani-type results related to $\nu$ and $\omega$. We highlight the following result, based on Euler-homogeneity, a condition analogous to the one in \cite[Theorem 0.8]{Saito94}. 

Suppose $\cS$ is a simple $\sD$-module. For any $m\in \cS$ and $\gamma \in \mathbb{R}$, set 
$\tilde{\nu}_{m, \gamma}=\max \{ \nu_{m, \gamma} - 1, \, 0 \}$, if $\gamma \in \Z_{\geq 1}$, and $\tilde{\nu}_{m, \gamma}=\nu_{m,\gamma}$ otherwise.

\begin{thm}
    Let $f,g$ be functions on smooth algebraic varieties $X,Y$, respectively, and consider $h=f+g$ on $X\times Y$. Let $\cS_1$ and $\cS_2$ be simple holonomic $\sD_X$- and $\sD_Y$-modules underlying (complex) Hodge modules, and $\cS=\cS_1 \boxtimes \cS_2$, such that $(\cS_1)_f, (\cS_2)_g, \cS_h$ are non-zero. Assume there exists a vector field $\chi$ on $Y$ such that $\chi \cdot g = g$. Let $m_i \in \cS_i$ such that $\chi \cdot m_2 = c \cdot m_2$, for some $c\in \C$, and set $m=m_1 \boxtimes m_2$. Then for any $\gamma \in \mathbb{R}$ we have
    \[\tilde{\nu}_{m, \gamma} = \! \max_{\substack{\alpha+\beta = \gamma \\ \tilde{\nu}_{m_1, \alpha}\neq 0 \neq \tilde{\nu}_{m_2, \beta} }} \!\!\!\!\! \{\tilde{\nu}_{m_1, \alpha} + \tilde{\nu}_{m_2, \beta} -1\}.\]
\end{thm}

We note that the inequality $\leq$ above always holds (see Corollary \ref{cor:thomsebineq}), in contrast with the multiplicities of roots of (reduced) Bernstein--Sato polynomials, see \cite[Example 4.8]{Saito94}. In this sense, the invariants $\nu$ behave better than the multiplicities of roots: for example, the failure of \cite[Theorem 0.8]{Saito94} beyond the Euler-homogeneous case as in \cite[Example 4.8]{Saito94} is rectified by working with $\nu$ instead, so that in the same example we have equality above for all $\gamma \in \mathbb{R}$ (see Example \ref{ex:tpqr} with $m_i=1$).

\medskip

Next, we discuss a question of Walther. Let $\mathrm{Jac}(f)\subseteq \cO_X$ be the ideal generated by $f$ and its partial derivatives. In \cite{ulisurvey}, the following ideal is considered for any polynomial $q(s) \in \C[s]$: 
\[\mathfrak{a}_{f, q(s)} = \{g \in \cO_X \mid q(s) \cdot gf^s \in \sD_X[s]f^{s+1} \}.\]
One can show that
\[\mathfrak{a}_{f, s+1} = \cO_X \cap \left(\mathrm{ann}_{\sD_X[s]}(f^s) + \sD_X[s]\cdot \mathrm{Jac}(f)\right).\]
Clearly, $\mathrm{Jac}(f)\subseteq \mathfrak{a}_{f, s+1}$. Walther asked whether the two ideals are always equal \cite[Question 1.4]{ulisurvey}. We prove the following:

\begin{thm}\label{thm: afs+1 lies in adjoint}
    Assume $D=\mathrm{div}(f)$ is reduced. Then we have 
    \[\mathfrak{a}_{f, s+1} \subseteq \mathrm{adj}(X,D) \cap \sqrt{\mathrm{Jac}(f)}.\]
\end{thm}

\medskip

To finish this subsection, we discuss the question of Budur--Walther (see \cite{saitopower}), which asks when $\sD_X\cdot f^{-\alpha}\neq \sD_X\cdot f^{-\alpha+1}$ under the assumption $b_f(-\alpha)=0$. While the question has a negative answer in general \cite{saitopower} (which Theorem \ref{thm: torelli}(c) relies on), the following provides a result in the affirmative, see Corollary \ref{cor:positive Budur-Walther}.
\begin{corollary}
    If $\nu_{1,\alpha-1}<\nu_{1,\alpha}$ and $\underset{s=-\alpha}{\mathrm{ord}} Z_f=\nu_{1,\alpha}$, then $\sD_X\cdot f^{-\alpha}\neq \sD_X\cdot f^{-\alpha+1}$.
\end{corollary}

Note that the condition $\nu_{1,\alpha-1}<\nu_{1,\alpha}$ is satisfied whenever $b_f(-\alpha)=0$ and $b_f(-\alpha+k)\neq 0$ for any $k\in \Z_{> 0}$, as in such case $\nu_{1,\alpha-1}=0$ and $\nu_{1,\alpha}= \mult_{s=-\alpha} b_f(s)$. 

\subsection*{Acknowledgements}
The authors would like to thank Dougal Davis, Bradley Dirks and Sheng Tan for helpful discussions. 
\subsection*{AI Disclosure}
The main mathematical content of this manuscript was completed around April 2025, although its submission to arXiv was delayed. AI tools were used during the final revision to improve the language and audit the the manuscript.

\section{$V$-filtrations and mixed twistor $\sD$-modules}
In this section, let $X$ be a smooth complex algebraic variety. We review some basic facts about the $V$-filtrations and weight filtrations of mixed twistor $\sD$-modules.
\subsection{Recollection of $V$-filtrations}\label{sec: Vfiltration}
Let $\mc{M}$ be a coherent left $\ms{D}_X$-module. Suppose $D \subseteq X$ is a smooth divisor. We fix a local equation $t = 0$ for $D$ and a vector field $\partial_t$ such that $[\partial_t, t] = 1$. The \emph{$V$-filtration of $\ms{D}_X$} is the decreasing $\mb{Z}$-indexed filtration given by
\[ V^n \ms{D}_X = \{ P \in \ms{D}_X \mid P t^m \in (t^{m + n}) \text{ for all }m \geq 0\}.\]
In this paper, unless otherwise specified, we work with \emph{$\mb{C}$-indexed} $V$-filtrations. We endow $\mb{C}$ with the lexicographic ordering: $\alpha<\beta$ if and only if $\Re(\alpha)<\Re(\beta)$ or, if $\Re(\alpha)=\Re(\beta)$, then $\Im(\alpha)<\Im(\beta)$. 

\begin{definition} \label{defn:CV-filtration}
A \emph{$V$-filtration on $\mc{M}$} along $D$ is a $\mb{C}$-indexed filtration $\{V^{\alpha}\mc{M}\}_{\alpha \in \mb{C}}$ satisfying the following properties:
\begin{enumerate}
\item \label{itm:CV-filtration 1} $V^\bullet \mc{M}$ is decreasing, $V^{\alpha - \epsilon}\mc{M} = V^{\alpha}\mc{M}$ for $\epsilon\in \C$ with $\epsilon\geq 0$ and $|\epsilon|\ll 1$ (i.e., left continuous), and, locally on $X$, there exists a finite set $A\subseteq \C$ such that $V^{\alpha+\epsilon}\cM=V^{\alpha}\cM$ for $\alpha\not\in A+\Z$. 
\item \label{itm:CV-filtration 2} The filtration $V^{\bullet}\cM$ is good over $V^{\bullet}\sD_X$, i.e., it is exhaustive, satisfies $V^n \ms{D}_X \cdot V^{\alpha} \mc{M} \subseteq V^{\alpha + n}\mc{M}$ for all $\alpha\in \mathbb{C},n\in \mathbb{Z}$, each $V^{\alpha}\mc{M}$ is a coherent $V^0\ms{D}_X$-module, and there exists, locally on $X$, a finite set of indices $\{\alpha_i\in \C\}_{i\in I}$ such that
\[ V^{\alpha}\cM=\sum_{\substack{n\in \Z,i\in I\\ n+\alpha_i\geq \alpha}}V^n\sD_X\cdot V^{\alpha_i}\cM, \quad \textrm{for all $\alpha\in \C$}.\]
In particular, the operator $t : V^{\alpha}\mc{M} \to V^{\alpha + 1}\mc{M}$ is an isomorphism for $\textrm{Re}(\alpha) \gg 0$.
\item \label{itm:CV-filtration 3} For each $\alpha \in \mb{C}$, the operator $(-\d_tt+\alpha)$ is nilpotent on $\gr_V^{\alpha}\mc{M}\colonequals V^{\alpha}\cM/V^{>\alpha}\cM$, where $V^{>\alpha}\cM\colonequals \cup_{\alpha'>\alpha}V^{\alpha'}\cM$. 
\end{enumerate}
\end{definition}
It is known \cite{Kas83} that if $\cM$ is holonomic, then the $V$-filtration on $\cM$ along $D$ exists and is unique; see also \cite[Proposition 2.3.2]{Sabbah}. In connection to mixed twistor $\sD$-modules, we also need to work with \emph{$\mb{R}$-indexed $V$-filtrations}.
\begin{definition} \label{defn:V-filtration}
An \emph{$\mb{R}$-indexed $V$-filtration} on $\mc{M}$ along $D$ is a filtration $\{V^\beta\mc{M}\}_{\beta \in \mb{R}}$ satisfying the following properties:
\begin{enumerate}
\item \label{itm:V-filtration 1} $V^\bullet \mc{M}$ is decreasing and left-continuous (i.e., $V^{\beta - \epsilon}\mc{M} = V^\beta\mc{M}$ for $0<\epsilon \ll 1$), and the set $\{\beta \in \mb{R} \mid V^\beta \mc{M} \neq V^{>\beta}\mc{M}\}$ is discrete, where $V^{>\beta}\mc{M}\colonequals \cup_{\beta'>\beta}V^{\beta'}\cM$.
\item \label{itm:V-filtration 2} The filtration $V^{\bullet}\cM$ is good over $V^{\bullet}\sD_X$; that is, it is exhaustive, $V^n \ms{D}_X \cdot V^\beta \mc{M} \subseteq V^{\beta + n}\mc{M}$ for all $\beta\in \mathbb{R}$ and $n\in \mathbb{Z}$, each $V^\beta\mc{M}$ is a coherent $V^0\ms{D}_X$-module, and there exists, locally on $X$, a finite set of indices $\{\beta_i\in \mb{R}\}_{i\in I}$ such that
\[ V^{\beta}\cM=\sum_{\substack{n\in \Z,i\in I\\ n+\beta_i\geq \beta}}V^n\sD_X\cdot V^{\beta_i}\cM, \quad \textrm{for all $\beta\in \mb{R}$}.\]
In particular, the operator $t : V^\beta\mc{M} \to V^{\beta + 1}\mc{M}$ is an isomorphism for $\beta \gg 0$.
\item \label{itm:V-filtration 3} For each $\beta \in \mb{R}$, there exist, locally on $X$, finitely many indices $\alpha_1,\ldots,\alpha_{k}\in \C$ with $\Re(\alpha_i)=\beta$ such that the operator $\prod_{i=1}^k (-\partial_tt+\alpha_i)$ acts by zero on $\gr_V^\beta\mc{M}=V^{\beta}\cM/V^{>\beta}\cM$.
\end{enumerate}
\end{definition}

\begin{lemma}\label{lemma: complex and real Vfiltrations}
A coherent $\sD$-module $\cM$ admits a $V$-filtration along $D$ if and only if it admits an $\mb{R}$-indexed $V$-filtration along $D$.
\end{lemma}

\begin{proof}
Suppose first that $\{V^{\alpha}\cM\}_{\alpha\in\C}$ is a
$\C$-indexed $V$-filtration. For $\beta\in\mathbb{R}$, define
\[
U^{\beta}\cM\colonequals\bigcup_{\substack{\alpha\in\C,\, \Re(\alpha)=\beta}}V^{\alpha}\cM.
\]
The discreteness of the jumps of $V^\bullet\cM$ implies that $U^{\beta}\cM$ is locally equal to some $V^\alpha\cM$. In particular, $U^\beta\cM$ is coherent
over $V^0\sD_X$. The remaining of properties in \eqref{itm:V-filtration 1} and \eqref{itm:V-filtration 2} follow immediately from the corresponding properties of $V^\bullet\cM$, after taking real parts of the indices.  The property \eqref{itm:V-filtration 3} follows from the fact that for a fixed $\beta\in \mb{R}$, the module $\gr_U^\beta\cM$ has a finite
filtration whose successive quotients are
$\gr_V^\alpha\cM$ with $\Re(\alpha)=\beta$. Thus $U^\bullet\cM$ is an $\mb{R}$-indexed $V$-filtration.

Conversely, suppose that $\{U^\beta\cM\}_{\beta\in\mb{R}}$ is an
$\mb{R}$-indexed $V$-filtration. By assumption, the action of $\partial_t t$ on $\gr_U^\beta\cM$ is
annihilated by a polynomial whose roots have real part $\beta$.
Consequently, there is a finite generalized-eigenspace
decomposition
\[
\gr_U^\beta\cM=\bigoplus_{\substack{\lambda\in\C,\, \Re(\lambda)=\beta}}
G_{\beta,\lambda},
\qquad
G_{\beta,\lambda}
\colonequals
\ker(\partial_t t-\lambda)^N, \quad \textrm{for $N\gg 0$}.
\]
Since the
class of $\partial_t t$ is central in $\gr_V^0\sD_X$, each
$G_{\beta,\lambda}$ is a coherent $\gr_V^0\sD_X$-submodule of
$\gr_U^\beta\cM$. For $\alpha\in\C$, with $\beta=\Re(\alpha)$, define
\[
V^\alpha\cM\colonequals
q_\beta^{-1}
\left(
\bigoplus_{\substack{\lambda\geq\alpha,\,
                     \Re(\lambda)=\beta}}
G_{\beta,\lambda}
\right),
\]
where $q_\beta:U^\beta\cM\longrightarrow \gr_U^\beta\cM$ is the natural quotient map. 

The filtration $V^\bullet\cM$ is clearly decreasing, exhaustive, and
left-continuous. Moreover, the commutator relation $[\partial_t t,P]=nP\in \gr_V^n\sD_X$ shows that an element of $\gr_V^n\sD_X$ maps
$G_{\beta,\lambda}$ into $G_{\beta+n,\lambda+n}$. Hence $V^n\sD_X\cdot V^\alpha\cM\subseteq V^{\alpha+n}\cM$.
Together with the goodness of the original $\mb{R}$-indexed $V$--filtration, this also
shows that $V^\bullet\cM$ is good and that its jumps are locally
contained in a finite union of translates of $\Z$. Finally, note that $\gr_V^\alpha\cM\simeq G_{\Re(\alpha),\alpha}$, so $-\partial_t t+\alpha$ acts nilpotently on
$\gr_V^\alpha\cM$. Therefore $V^\bullet\cM$ is a $\C$-indexed
$V$-filtration.
\end{proof}

\subsection{Malgrange transform and weight filtrations}\label{sec: Malgrange transform}
Let $f$ be an algebraic function on $X$ and let $\iota:X\to X\times \mathbb{C}_t$ be the graph embedding of $f$. Let $\cM$ be a holonomic $\sD_X$-module such that $f$ acts bijectively on $\cM$; equivalently, $\cM=\cM_f$, the localization of $\cM$ along $\{f=0\}$. We endow $\cM[s]f^s$ with a $\sD_X\langle s,t\rangle$-module structure via:
\begin{align*}
    \xi\cdot(ms^{\ell}f^s)&=\left(\xi(m)s^{\ell}+ms^{\ell+1}\frac{\xi(f)}{f}\right)f^s, \quad \forall \xi\in T_X,\\
    s\cdot (ms^{\ell}f^s)&=ms^{\ell+1}f^s,\quad t\cdot (ms^{\ell}f^s)=\left(m(s+1)^{\ell} f\right) f^s,
\end{align*}
and $st=t(s-1)$. We identify $\iota_{+}\cM \cong \sum_{\ell\geq 0} \cM\otimes \partial_t^{\ell}$ and regard it as a coherent $\sD_X\langle s,t\rangle$-module where $s$ acts by $-\partial_tt$.

\begin{prop}[\cite{Malgrange,MPVfiltration}]\label{prop:malgrange}There is an isomorphism of $\sD_X \langle s,t\rangle$-modules
\begin{equation}\label{eqn: Malgrange isomorphism}
\cM[s]f^s \xrightarrow{\sim} \iota_+ \cM, \quad ms^{\ell}f^s \mapsto m\otimes (-\partial_tt)^{\ell},
\end{equation}
and the inverse isomorphism is given by $m\otimes \partial_t^{\ell} \mapsto \frac{m}{f^{\ell}}\prod_{j=0}^{\ell-1}(-s+j)f^s$.
\end{prop}

From now on, fix $\alpha \in \C$. We have a holonomic $\sD$-module $\cM\cdot f^{-\alpha}\colonequals  \cM[s]f^s/(s+\alpha)$.
\begin{lemma}\label{lemma: shift of V filtration by talpha}
There is a natural isomorphism of $\sD_X\langle s,t\rangle$-modules
\[ \Phi: \iota_{+}\left(\cM\cdot f^{-\alpha}\right) \to \iota_{+}\cM,\]
where the $\sD_X\langle s,t\rangle$-action on the right-hand side is twisted by the automorphism of $\sD_X\langle s,t\rangle$ that maps $s \mapsto s+\alpha$ and acts as the identity on both $\sD_X$ and $t$. It induces a $\sD_X\langle s,t\rangle$-isomorphism
    \begin{equation*}\label{eqn: talpha shift the Vfiltration} \Phi: V^{\alpha'}\iota_{+}(\cM  \cdot f^{-\alpha})\xrightarrow{\sim} V^{\alpha'+\alpha}\iota_{+}\cM.\end{equation*}
\end{lemma}
\begin{proof}
   The proof of \cite[Proposition 2.6]{MPVfiltration} works verbatim for our $\C$-indexed $V$-filtration.
\end{proof}

\begin{definition}\label{definition: weight filtration and hard lefschetz} 
Denote $N\colonequals s+\alpha$. Let $W(N)_{\bullet}\gr^{\alpha}_V\iota_{+}\cM$ be the associated monodromy weight filtration centered at $0$, which is uniquely characterized by the following properties:
    \begin{itemize}
        \item $N:W(N)_{\ell}\gr^{\alpha}_V\iota_{+}\cM\to W(N)_{\ell-2}\gr^{\alpha}_V\iota_{+}\cM$ for any $\ell$, and
        \item it induces isomorphisms
        \[ N^{\ell}: \gr^{W(N)}_{\ell}\gr^{\alpha}_V\iota_{+}\cM\xrightarrow{\sim}\gr^{W(N)}_{-\ell}\gr^{\alpha}_V\iota_{+}\cM \quad \textrm{for all $\ell>0$}. \]
    \end{itemize}
\end{definition}
    It is also given by the following convolution formula:
    \begin{equation}\label{lemma: alternative convolution formula}
    W(N)_\ell = \sum_{i+j=\ell} \ker N^{i+1}\cap \mathrm{Im} N^{-j}=\sum_{j\geq 0} N^{j}(\ker N^{\ell+2j+1}).
    \end{equation}
    The first equality can be found in \cite[\S 4.5]{DM05}. The second follows from the observation that $\ker N^{i+1} \cap \mathrm{Im} N^{j} = N^j(\ker N^{i+j+1})$ for any $i,j$. 

\begin{definition}\label{definition:nilpotence degree}
    Let $n_\alpha$ be the \emph{nilpotency index} of $s+\alpha$ on $\gr_V^\alpha \iota_+ \cM$; that is, the minimal integer $\ell$ such that $(s+\alpha)^{\ell} \cdot \gr_V^\alpha \iota_+ \cM = 0$. 
\end{definition}

From now on, assume $\cS$ is a simple holonomic $\sD$-module and $\cM=\cS_f\neq 0$ the localization of $\cS$ along $D\colonequals\textrm{div}(f)$. Set
\begin{equation}\label{eqn: definition of Salpha} \cS^{-\alpha}\colonequals \textrm{Im}\left((\cM\cdot f^{-\alpha})(!D)\to \cM\cdot f^{-\alpha}\right).\end{equation}
Clearly, $\cM\cdot f^{-\alpha}$ is the localization of $\cS^{-\alpha}$ along $D$. Let $i:D\to X$ be the closed embedding.

\begin{lemma}\label{lemma: kernel and cokernel of s+alpha}
 For any  $\alpha\in \C$, there is an exact sequence of $\sD$-modules 
\begin{equation}\label{eqn: coker of grValpha is MquotientbyS} 0\to \cH^{-1}(i_{\ast}i^{\ast}\cS^{-\alpha})\to \gr^{\alpha}_V\iota_{+}\cM \xrightarrow{s+\alpha} \gr^{\alpha}_V\iota_{+}\cM\to \cM\cdot f^{-\alpha}/\cS^{-\alpha}\to 0.\end{equation}
\end{lemma}

\begin{remark}\label{remark: evaluation }
If $\cS$ underlies a pure Hodge module, this can be found in \cite[(2.24.2)]{Saito90}. In Proposition \ref{prop: diagram of j*j!} we will show that the last map in \eqref{eqn: coker of grValpha is MquotientbyS} is induced by the evaluation map $ev_{s=-\alpha}:\cM[s]f^s\to \cM\cdot f^{-\alpha}$, given by $s\mapsto -\alpha$.
\end{remark}
\begin{proof}
We can reduce to the case $\alpha=0$ using the following consequence of Lemma \ref{lemma: shift of V filtration by talpha}:
\[ \textrm{Cone}\left(\gr_V^0\iota_{+}(\cM\cdot f^{-\alpha})\xrightarrow{s} \gr_V^0\iota_{+}(\cM\cdot f^{-\alpha})\right)\cong \textrm{Cone}
\left(\gr^{\alpha}_V\iota_{+}\cM\xrightarrow{s+\alpha} \gr^{\alpha}_V\iota_{+}\cM\right).\]
So we assume $\alpha=0$. Since $f$ acts bijectively on $\cM$, it is a classical fact (see, e.g., \cite[Page 45, line 3]{Beilinson} and \cite[Corollary 4.1.12]{BBDG}) that there is an isomorphism
\[ \textrm{Cone}\left(\cM(!D)\to \cM\right)\xrightarrow{\sim} \textrm{Cone}\left(\gr^1_V\iota_{+}\cM\xrightarrow{s+1}\gr^1_V\iota_{+}\cM\right).\]
Since $D$ is a hypersurface, we obtain an exact sequence:
\[ 0\to \cH^{-1}(i_{\ast}i^{\ast}\cS)\to \cM(!D)\to \cM\to \cM/\cS\to 0.\]
On the other hand, the isomorphism $t:\gr^0_V\iota_{+}\cM\xrightarrow{\sim} \gr^1_V\iota_{+}\cM$ immediately yields
\[ \textrm{Cone}\left(\gr^1_V\iota_{+}\cM\xrightarrow{s+1}\gr^1_V\iota_{+}\cM\right)\cong \textrm{Cone}\left(\gr^0_V\iota_{+}\cM\xrightarrow{s}\gr^0_V\iota_{+}\cM\right),\]
which finishes the proof. 
\end{proof}

Next, we use the simplicity of $\cS$ to put a weight filtration on $\cM\cdot f^{-\alpha}$. Let $\cX=X\times \C_\lambda$ and consider the sheaf of $\cO_{\cX}$-modules $R_{\cX}$, which is locally $\cO_{\cX}\langle \lambda\partial_1,\ldots,\lambda\partial_{\dim X}\rangle$. In particular, $R_{\cX}|_{\lambda=1}=\sD_X$. Since $\cS$ is simple and holonomic, it underlies a polarizable purely imaginary (wild) pure twistor $\sD$-module $\cT_0$ on $X$ \cite[Theorem 19.4.1]{Mochizukiwild}. After Tate twist, we fix $\cT_0$ of weight $q$. Let $U:=X\setminus D$ and let $j:U\hookrightarrow X$ be the open embedding. Set $\cV_\alpha:=j^*\cT_0\otimes\cL_{-\alpha}$, where $\cL_{-\alpha}$ is the weight 0 rank one variation of pure twistor structures associated with $f^{-\alpha}$, and define its minimal extension
\[\cT_\alpha\colonequals \Im\bigl((\cV_\alpha)_!\longrightarrow(\cV_\alpha)_*\bigr).\]
By \cite[\S13.3.5]{MochizukitwistorDmodule}, $\cT_\alpha$ is a polarizable purely imaginary pure twistor $\sD$-module of weight $q$. Write $\cT_\alpha=(\cM'_\alpha,\cM''_\alpha,C)$. Here $\cM'_{\alpha}$ and $\cM''_{\alpha}$ are coherent $R_{\cX}$-modules, $C:\cM'_{\alpha}|_{X\times \mathbf{S}}\times \cM''_{\alpha}|_{X\times \mathbf{S}}\to \mathrm{Db}_{X\times \mathbf{S}/\mathbf{S}}$
is a sesquilinear pairing, and $\mathbf{S}=\{ \lambda\in \C\mid |\lambda|=1\}$; see \cite[\S 1.3.1]{MochizukitwistorDmodule} for details. After specializing at $\lambda=1$, the first underlying $\sD_X$-module is the minimal extension of $j^*\cS\cdot f^{-\alpha}$. Therefore, by \eqref{eqn: definition of Salpha},
\[\cM'_\alpha\big|_{\lambda=1}=\Im\bigl((\cM\cdot f^{-\alpha})(!D)\longrightarrow \cM\cdot f^{-\alpha}\bigr)=\cS^{-\alpha}.\]
Because $\cT_{\alpha}$ is purely imaginary, it is an $(\mathbb{R}\times \sqrt{-1}\mathbb{R})$-pure twistor $\sD$-module 
\cite[\S 7.1.8]{MochizukitwistorDmodule}. Consequently, \cite[Example 11.2.8]{MochizukitwistorDmodule} implies that $\lambda=1$ is generic with respect to the KMS spectrum $\textrm{KMS}(\cT_{\alpha},f)$ in the sense of \cite[\S 2.1.2.4]{MochizukitwistorDmodule}. This genericity makes $\cM\cdot f^{-\alpha}$ underlying a mixed twistor $\sD$-module. Indeed, consider the localization mixed twistor $\sD$-module $\cT_{\alpha}[\ast D]$ \cite[\S 7.1.6.2]{MochizukitwistorDmodule}. By construction, its weight filtration starts at level $q$. By \cite[Lemma 3.3.4]{MochizukitwistorDmodule},
\[ \cM'_{\alpha}(\ast D)|_{\lambda=1}=\left(\cM'_{\alpha}|_{\lambda=1}\right)(\ast D)=\cS^{-\alpha}(\ast D)=\cM\cdot f^{-\alpha}.\]
This induces a canonical weight filtration on $\cM\cdot f^{-\alpha}$ (up to shifts):
\[ W_{\ell}\left(\cM\cdot f^{-\alpha}\right)\colonequals W_{\ell}\left(\cM'_{\alpha}(\ast D)\right)|_{\lambda=1}.\]

We explain that this weight filtration can be computed by the monodromy weight filtration on $\gr^{\alpha}_V\iota_{+}\cM$ given in Definition \ref{definition: weight filtration and hard lefschetz} (this is well-known when $\cM$ underlies a $\Q$-mixed Hodge module \cite[(2.11.10)]{Saito90}).
\begin{lemma}\label{lemma: grWM in terms of grV}
We have $W_{q}(\cM\cdot f^{-\alpha})=\cS^{-\alpha}$, and 
\begin{equation}\label{eqn: WmoduloS is W on grV}\frac{W_{\ell}(\cM\cdot f^{-\alpha})}{\cS^{-\alpha}} \cong \frac{W(N)_{\ell-q-1}\gr^{\alpha}_V\iota_{+}\cM}{N\cdot W(N)_{\ell-q+1}\gr^{\alpha}_V\iota_{+}\cM}, \quad \textrm{whenever $\ell>q$},\end{equation}
and this isomorphism is induced by $ev_{s=-\alpha}$.
\end{lemma}
\begin{proof}
Since $\cS$ is simple, one can see directly that $\cS^{-\alpha}$ is the unique simple submodule of $\cM\cdot f^{-\alpha}$, with multiplicity $1$. Furthermore, because $\cM\neq 0$ by assumption, the support of $\cS^{-\alpha}$ is not contained in $\{f=0\}$. Hence there is a short exact sequence of $\sD$-modules
\begin{equation}\label{eqn: ses of S-alpha and Mf-alpha} 0 \to \cS^{-\alpha}\to \cM\cdot f^{-\alpha} \to \cM\cdot f^{-\alpha}/\cS^{-\alpha}\to 0.\end{equation}
Because $\cS^{-\alpha}$ underlies the pure imaginary $\cT_{\alpha}$, by the discussion above, $\lambda=1$ is generic with respect to the KMS spectrum $\textrm{KMS}(\cT_{\alpha},f)$. In particular, the sequence \eqref{eqn: ses of S-alpha and Mf-alpha} underlies the exact sequence of mixed twistor $\sD$-modules, i.e. by setting $\lambda=1$,
\[ 0\to \cT_{\alpha} \to \cT_{\alpha}[\ast D]\to\cT_{\alpha}[\ast D]/\cT_{\alpha} \to 0.\]
Note that $\cT_{\alpha}$ is a pure twistor $\sD$-module of weight $q$, and the weight filtration of $\cT_{\alpha}[\ast D]$ starts at $q$. By the strictness of weight filtrations, the semisimplicity of $W_q(\cM\cdot f^{-\alpha})$ and the fact that $\cS^{-\alpha}$ is the unique simple submodule of $\cM\cdot f^{-\alpha}$, we have $W_{q}(\cM\cdot f^{-\alpha})=\cS^{-\alpha}$.

For \eqref{eqn: WmoduloS is W on grV}, denote by $U^{\bullet}\iota_{+}\cM$ the $\mb{R}$-indexed $V$-filtration along $\{t=0\}$. Lemma \ref{lemma: complex and real Vfiltrations} implies that $\gr^{\alpha}_V\iota_{+}\cM$ is the generalized $\alpha$-eigenspace of $\gr^{\textrm{Re}(\alpha)}_U\iota_{+}\cM$ for the action of $\partial_tt$. Because $\lambda=1$ is generic, by \cite[Lemma 4.1.11 and Lemma 10.2.9]{MochizukitwistorDmodule}, for $\alpha\neq 0$, $\gr^{\alpha}_V\iota_{+}\cM$ underlies the mixed twistor $\sD$-module (see \cite[Page 22]{MochizukitwistorDmodule} for the definition)
\[ \widetilde{\psi}_{f,u_{\alpha}}\cT_{0}, \quad u_{\alpha}\colonequals
\left(-\Re(\alpha),\frac{\sqrt{-1}}{2}\Im(\alpha)\right).\]
Indeed, by \cite[\S 2.1.2.1,(v)]{MochizukitwistorDmodule}, after evaluating at $\lambda=1$, the operator $-\d_tt+\mathfrak{e}(1,u_{\alpha})$ acts nilpotently on the generalized $\alpha$-eigenspace of $\gr^{\Re(\alpha)}_V\iota_{+}\cM$, and so $\mathfrak{e}(1,u_{\alpha})=\alpha$. Since $\cT_{\alpha}$ is purely imaginary, we can set $u_{\alpha}=(a,\sqrt{-1}b)\in \mb{R}\times\sqrt{-1}\mb{R}$, then $\mathfrak{e}(1,u_{\alpha})=-a+2\sqrt{-1}b$, and this determines $u_{\alpha}$.

Furthermore, $\gr^{0}_V\iota_{+}\cM$ underlies the mixed twistor $\sD$-module $\phi_f(\cT_{0}[\ast D])$, and the monodromy weight filtrations on $\gr^{\alpha}_V\iota_{+}\cM$ underly the corresponding monodromy weight filtrations. By Beilinson's construction of nearby and vanishing cycles for $R_{\cX}$-modules \cite[\S 4.2.2]{MochizukitwistorDmodule} the maps 
 \[ \mathrm{can}=\partial_t:\gr^1_{V}\iota_{+}\cM\to \gr^0_{V}\iota_{+}\cM,\quad \mathrm{var}=t:\gr^0_{V}\iota_{+}\cM\to \gr^1_{V}\iota_{+}\cM,\]
underly morphisms between mixed twistor $\sD$-modules.  We conclude that the same is true for all maps in \eqref{eqn: coker of grValpha is MquotientbyS}, because $\left(\cM\cdot f^{-\alpha}\right)(!D)=\cS^{-\alpha}(!D)$ underlies the mixed twistor $\sD$-module $\cT_{\alpha}(!D)$ by \cite[Lemma 3.3.4 and \S 7.1.6.2]{MochizukitwistorDmodule}. 

Now we are ready to prove \eqref{eqn: WmoduloS is W on grV}. First assume $\alpha\neq 0$. Because $\cT_0$ has weight $q$, by \cite[\S 7.1.1.1]{MochizukitwistorDmodule}, the monodromy weight filtration on $\gr^{\alpha}_V\iota_{+}\cM$ is centered at $q$. By \cite[(4.18) and \S 7.1.1]{MochizukitwistorDmodule}, the operator $N$ induces a morphism of mixed twistor $\sD$-modules
\[ N:\widetilde{\psi}_{f,u_{\alpha}}\cT_{0}\otimes U(-1,0)\to \widetilde{\psi}_{f,u_{\alpha}}\cT_{0}\otimes U(0,-1),\]
where $U(a,b)$ are defined in \cite[\S 2.1.8.1]{MochizukitwistorDmodule}. Furthermore, at $\lambda=1$ the two sides have weight filtrations $W(N)_{\bullet-q+1}\gr^{\alpha}_V\iota_{+}\cM$ and $W(N)_{\bullet-q-1}\gr^{\alpha}_V\iota_{+}\cM$. So by \eqref{eqn: coker of grValpha is MquotientbyS} and the strictness of weight filtrations, we have
\begin{align*}
\frac{W_\ell(\cM\cdot f^{-\alpha})}{\cS^{-\alpha}}&\cong \frac{W(N)_{\ell-q-1}\gr^{\alpha}_V\iota_{+}\cM}{N\cdot W(N)_{\ell-q+1}\gr^{\alpha}_V\iota_{+}\cM}, \quad \ell>q.
\end{align*}

If $\alpha=0$, since $\cM=\cM(\ast D)$, we use the isomorphism $\gr^{0}_V\iota_{+}\cM\cong \gr^1_{V}\iota_{+}\cM$ which identifies the monodromy weight filtration. This proves \eqref{eqn: WmoduloS is W on grV}. 

The compatibility with $ev_{s=-\alpha}$ follows from Proposition \ref{prop: diagram of j*j!} below.
\end{proof}

\section{Power $b$-functions and $p$-functions}
We introduce power $b$-functions and $p$-functions, and establish some fundamental properties. Let $f$ be an algebraic function on a smooth complex algebraic variety $X$ (or holomorphic function on a complex manifold), $\iota:X\to X\times \C_t$ the graph embedding, and let $\cM$ be a holonomic $\sD_X$-module such that $f$ acts bijectively.

\subsection{Basic facts of $b$-functions}
As $\cM=\cM_f$, we have a $\sD\langle s,t\rangle$-module $\cM[s]f^s$ \S \ref{sec: Malgrange transform}.
\begin{definition}\label{definition: definition of BS polynomials}
For a local section $w\in \cM[s]f^s$, the \emph{Bernstein-Sato polynomial} $b_w(s)$ is the unique monic polynomial $b(s)\in\C[s]$ of minimal degree satisfying
\begin{align}\label{eqn: definition of b-function}
     P(s,t)\cdot (tw)=b(s)\cdot w,
\end{align}
for some $P(s,t)\in \sD_X\langle s,t\rangle$. If $w=p(s)mf^s$, set $b_{p(s)m}(s)\colonequals b_{w}(s)$ and we can choose $P(s)\in \sD_X[s]$ in \eqref{eqn: definition of b-function}.\end{definition}

\begin{thm}{\cite[Proposition 2.3.2]{Sabbah}}\label{thm: Sabbah}
Let $V^{\bullet}\iota_{+}\cM$ be the $V$-filtration, then
\[V^{\alpha}\iota_{+}\cM=\{w \in \iota_{+}\cM \mid \textrm{every root of $b_w(s)$ $\leq -\alpha$}\}, \quad \forall \alpha\in \mb{C}.\]
\end{thm}

The following is clear.
\begin{lemma}\label{lem: b function of tw}
For $w \in \cM[s]f^s$, we have  $b_{t w}(s)=t \cdot b_{w}(s)=b_w(s+1)$.
\end{lemma}

Given a complex number $a\in \C$, set
\begin{equation}\label{eqn: [s+a]}
[s + a]_k\colonequals \begin{cases}\prod_{i=0}^{k-1} (s + a + i), \quad &\text{if } k \geq 1,\\
1, \quad &\text{otherwise}.
\end{cases} \end{equation}

\begin{lemma}\label{lem:easy2}
Let $w \in \cM[s]f^s$ and $\alpha \in \C$. Then $b_w(s)$ divides $(s+\alpha) \cdot b_{(s+\alpha)w}(s)$ and there exists $k\in \Z_{>0}$ such that $ b_{(s+\alpha)w}(s)$ divides $[s+\alpha+1]_k \cdot b_w(s)$. If $b_w(-\alpha)=0$, then there exist distinct $k_1,\dots, k_j \in \Z_{>0} $ such that
\begin{equation}\label{bfunction of s+alphaw}b_{(s+\alpha)w}(s) = \frac{b_w(s)}{(s+\alpha)} \cdot \prod_{i=1}^j (s+\alpha+k_i).\end{equation}
\end{lemma}
\begin{proof}
    It follows from the proof of \cite[Lemma 2.5 and Lemma 2.6]{DLY}.
\end{proof}

\begin{corollary}\label{cor:tozfilt}
    Let $w\in V^\alpha \iota_+ \cM$ and $q(s) \in \C[s]$. Then we have
    \[\mu_{q(s)\cdot w, \alpha} \, = \, \max \{ \mu_{w,\alpha}- \mult_{s=-\alpha} q(s), \, 0 \}.\]
\end{corollary}

\begin{proof}
     If $w \in V^{>\alpha} \iota_+ \cM$, the claim is clear. Thus, we assume $w \notin V^{>\alpha} \iota_+ \cM$. We may further assume $q(s) = s+r$. If $r = \alpha$, the claim follows from (\ref{bfunction of s+alphaw}).

    Now consider the case $r \neq \alpha$. Lemma \ref{lem:easy2} implies that $ \mu:=\mu_{w, \alpha} \leq \mu_{(s+r)w, \alpha}$. By (\ref{bfunction of s+alphaw}) and Theorem \ref{thm: Sabbah}, we get $(s+\alpha)^\mu w \in V^{>\alpha} \iota_+ \cM$. Thus, $(s+r)(s+\alpha)^\mu w \in V^{>\alpha} \iota_+ \cM$. Then Lemma \ref{lem:easy2} implies $\mu\geq  \mu_{(s+r)w, \alpha}$, and so equality must hold.
\end{proof}

\begin{prop}\label{prop:bgcd}
For $m\in \cM$ and $0\neq p(s) \in \C[s]$. Write $b_m(s)=\prod_{i=1}^{d} (s+r_i)$ and choose a maximal $(k_1,\dots, k_d) \in \Z_{\geq 0}^d$ such that $[s+r_1]_{k_1} \cdots [s+r_d]_{k_d}$ divides $p(s)$, then
\begin{equation}\label{eqn: bpsm divides bm}
b_{p(s) m}(s) \, \mid \,  \prod_{i=1}^{d}(s+r_i + k_i).
\end{equation}
In particular, $\deg b_{p(s)m}(s) \leq \deg b_m(s)$.
\end{prop}

\begin{proof}
It follows from \cite[Proposition 2.9 and Lemma 2.10]{LY25April}.
\end{proof}

\begin{notation}\label{notation: muwalpha}
    Let $\alpha \in \C$ and $p(s) \in \C[s]$. For $w\in \cM[s]f^s$, we set $\mu_{w,\alpha} \colonequals \textrm{mult}_{s=-\alpha}b_w(s)$. For $w=p(s)mf^s$, we simply write $\mu_{p(s)m, \alpha} \colonequals \mu_{w,\alpha}$.
\end{notation}

\begin{lemma}\label{lem:localized}
    Let $0\neq p(s)\in \C[s]$, $m\in \cM$, and $\alpha\in \C$. Let $k_0\in \Z_{>0}$ be the minimal integer such that $p(k_0- \alpha)\neq 0$, and choose $k\geq k_0$. Then $\mu_{p(s)m, \alpha}$ is characterized as the smallest integer $\mu \in \Z_{\geq 0}$ such that 
    \[ (s+\alpha)^{\mu+ \mult_{s=-\alpha}p(s)} \cdot m f^s \in \sum_{i=1}^{k} \sD_X[s]_{(s+\alpha)} \cdot (s+\alpha)^{\mult_{s=-(\alpha-i)}p(s)} \cdot m f^{s+i}. \]
\end{lemma}
\begin{proof}
    Let $w=p(s)mf^s$, and recall that $b_{w}(s)$ is the non-zero polynomial of minimal degree such that there exists $Q \in \sD_X\langle s,t\rangle$ with
    \[b_w(s) \cdot p(s)  m f^s = Q \cdot p(s+1) m f^{s+1}, \]
 or equivalently, after applying the powers of $t$ that appear in $Q$,
     \[ b_w(s) \cdot p(s) m f^s  \in \sum_{i=1}^{\infty} \sD_X[s] \cdot p(s+i) m f^{s+i}.\]
Localizing at $(s+\alpha)$, we get that $\mu_{w, \alpha} \in \Z_{\geq 0}$ is the smallest such that
\[(s+\alpha)^{\mu_{w, \alpha} +\mult_{s=-\alpha}p(s)} \cdot m f^s \in \sum_{i=1}^{\infty} \sD_X[s]_{(s+\alpha)} \cdot (s+\alpha)^{\mult_{s=-(\alpha-i)}p(s)} \cdot m f^{s+i}. \]
Note that $(s+\alpha)^{\mult_{s=-(\alpha-k_0)}p(s)}=1$, and thus for $i\geq k_0$, we have $(s+\alpha)^{\mult_{s=-(\alpha-i)}p(s)} \cdot m f^{s+i} \in \sD_X[s]_{(s+\alpha)} \cdot  m f^{s+k_0}$ so that the sum stabilizes at $i=k_0$.
\end{proof}

\subsection{Power $b$-functions}
Recall that $\cM$ is a holonomic $\sD_X$-module such that $\cM=\cM_f$.
\begin{definition}\label{definition: higher order bfunction}
Let $w\in \cM[s]f^s$ and $k \in \Z_{> 0}$. We define the \emph{$k^{\textrm{th}}$ power $b$-function} of $w$, denoted by $b^{(k)}_w(s)$, as the unique monic polynomial $b(s)\in \C[s]$ of minimal degree satisfying 
\begin{equation}\label{eqn: power b-function}
    P(s,t)\cdot (t^kw) = b(s)\cdot w,
\end{equation}
for some $P(s,t)\in \sD_X\langle s,t\rangle$. For $w = p(s)mf^s$, we set $b^{(k)}_{p(s)m}(s)\colonequals b^{(k)}_w(s)$.
\end{definition}

\begin{lemma}\label{lem:bfunpower}
Let $k\in \Z_{>0}$ and $w\in \cM[s]f^s$. Then $b^{(k)}_{w}(s)$ divides $b^{(k+1)}_{w}(s)$ and there exists an $n\in \Z_{\geq 0}$ such that
    \begin{equation}\label{eqn: division of higher b function of w} 
        b^{(k)}_{w}(s) \, \mid \, b_w(s) b_w(s+1) \cdot \prod_{i=2}^n b_w(s+i)^{n}. 
    \end{equation}
    If $m\in \cM$, then $b_m^{(k)}(s) \, \mid \, b_m^{(i)}(s+j)\cdot b_m^{(j)}(s)$ for any $i,j\in \Z_{> 0}$ with $i+j=k$, and
    \begin{equation}\label{eqn: division of higher b function of m}
        \mathrm{lcm}\left(\, b_m(s), b_m(s+1), \dots, b_m(s+k-1)\, \right) \, \mid \, b^{(k)}_m(s) \, \mid \, \prod_{i=0}^{k-1} b_m(s+i).
    \end{equation}
\end{lemma}
\begin{proof}
    By definition, there exists an operator $P(s,t) \in \sD_X\langle s,t\rangle$ such that
    \begin{equation}\label{eq:startb}
        P(s,t) \cdot (t^k w) = b^{(k)}_w(s) w.
    \end{equation}
    We can rewrite this as $P(s,t)t \cdot (t^{k-1} w) = b^{(k)}_w(s) \cdot w$, proving $b^{(k-1)}_{w}(s) \mid b^{(k)}_w(s)$, if $k>1$.
    
To prove \eqref{eqn: division of higher b function of w}, let us define polynomials $B_i(s) \in \C[s]$ recursively by $B_0(s)= 1$ and $B_{i+1}(s) = b_{B_i(s)w}(s) \cdot B_i(s)$ for $i\geq 0$. Let $P_i \in \sD_X\langle s,t\rangle$ be the operators yielding the respective Bernstein-Sato polynomials:
    \[ P_i \cdot t^{k-i}(B_i(s)w) = b_{t^{k-i-1}B_i(s)w}(s) \cdot (t^{k-i-1}B_i(s) w)  =  t^{k-i-1} \cdot ( B_{i+1}(s)  w), \]
    where the last equality follows from $b_{t^{k-i-1}B_i(s)w}(s) \cdot t^{k-i-1} = t^{k-i-1} \cdot b_{B_i(s)w}(s)$ by Lemma \ref{lem: b function of tw}. Thus, applying the operator $P' = P_{k-1} \cdots P_0$ to $t^k w$ yields the equation
     \[ P'(s,t) \cdot (t^k w) = B_k(s) \cdot w, \]
    which implies $b^{(k)}_{w}(s) \mid B_k(s)$. On the other hand, note that for any element $z \in \iota_+ \cM$, repeatedly applying Lemma \ref{lem:easy2} yields
    \begin{equation}\label{eq:bb}
        b_{b_z(s) \cdot z}(s) \, \mid \,\, b_z(s+1) \cdots b_z(s+n'),
    \end{equation}
    for some $n' \in \Z_{>0}$. Applying \eqref{eq:bb} to $z=B_i(s) \cdot w$ for each $i$, we find by induction that there exists an $n\in \Z_{\geq 0}$ such that
    \[  B_k(s) \, \mid \,\, b_w(s)b_w(s+1) \cdot \prod_{i=2}^n b_w(s+i)^{n}. \]
This proves \eqref{eqn: division of higher b function of w}.

    Now assume $w = m f^s$ for $m \in \cM$. In \eqref{eq:startb}, we may assume $P \in \sD_X[s]$. For any $i=0,\dots, k-1$, we can rewrite the equation as
    \[ (f^i P f^{k-i-1}) \cdot (f^i m) f^s = b^{(k)}_m(s) \cdot (f^i m) f^s. \]
    Combined with Lemma \ref{lem: b function of tw}, this implies $b_m(s+i)=b_{f^i m}(s) \mid b^{(k)}_m(s)$, proving the left hand side of \eqref{eqn: division of higher b function of m}. Let $i+j=k$. By definition, there exists $Q(s) \in \sD_X[s]$ such that
    \[ Q(s) \cdot m f^{s+i} = b_m^{(i)}(s) \cdot m f^s. \]
    Applying $Q(s+j)$ to $m f^{s+k}$ we obtain an equation
    \[ Q(s+j) \cdot m f^{s+k} = b^{(i)}_m(s+j)\cdot m f^{s+j}. \]
    Applying an operator that yields $b_m^{(j)}(s)$ to this equation gives $b_m^{(k)}(s) \mid b_m^{(i)}(s+j)\cdot b_m^{(j)}(s)$. From this, it follows by induction that $b^{(k)}_m(s) \mid  b_m(s)b_m(s+1)\cdots b_m(s+k-1).$
\end{proof}

For the last three results of this subsection, we assume $\cM=\cS_f$, where $\cS$ is a \emph{simple} holonomic $\sD_X$-module.

\begin{lemma}\label{lemma: relating two elements in iota+M}
    For any $w,w'\in \cM[s]f^s$, there exists an operator $P\in \sD_X\langle s,t\rangle$ such that
    \begin{equation}\label{eqn: higher bfunction of w' via w} 
        P\cdot w' = b_{w}^{(k)}(s)\cdot w,
    \end{equation}
    for some $k\in \Z_{\geq 0}$. Furthermore, if $w'=mf^s$, then we can choose such $P \in \sD_X[s]$. 
\end{lemma}

 \begin{proof}
    Since $\iota_+ (\cM/\cS)$ supports on $\{t=0\}$, there exists an $a \in \Z_{\geq 0}$ such that $t^{a} \cdot w, t^{a} \cdot w' \in \iota_+ \cS$. Because $\iota_+ \cS$ is simple by Kashiwara's Theorem, there exists an operator $P_1\in \sD_{X\times \C}$ such that $P_1(t, \d_t) \cdot (t^a w') = t^a w$. Multiplying by a sufficiently high power of $t$ eliminates $\partial_t$ from $P_1$, yielding the equation
    \begin{equation}\label{eq:appi}
        P_2(s, t) \cdot  w' = t^{k} w,
    \end{equation}
    for some $P_2 \in \sD_X\langle s,t\rangle$ and $k\in \Z_{> 0}$. Let $P_3\in \sD_X\langle s,t\rangle$ be the operator yielding the $k$-th power $b$-function of $w$: $P_3(s,t) \cdot (t^k w) = b^{(k)}_w(s) w$. Applying $P_3$ to \eqref{eq:appi} immediately gives \eqref{eqn: higher bfunction of w' via w}.
 \end{proof}

From Lemma \ref{lemma: relating two elements in iota+M}, we record two useful corollaries.
\begin{corollary}[Kashiwara]\label{cor: generation criterion in terms of roots}
 For $m\in \cM$ and $\alpha\in \C$, if $b_{m}(s)$ has no roots in $-\alpha-\Z_{> 0}$, then $\sD_X\cdot (mf^{-\alpha})=\cM\cdot f^{-\alpha}$. 
\end{corollary}

\begin{lemma}\label{lem:prop2}
    For any $w,w_1, w_2 \in \cM[s]f^s$ and $Q \in \sD_X[s]$, there exists $k \in \Z_{\geq 0}$ such that
    \[ b_{Q\cdot w}(s) \, \mid \,\, b^{(k)}_w(s), \quad b_{w_1+w_2}(s) \, \mid \,\, \mathrm{lcm}\left(b^{(k)}_{w_1}(s), b^{(k)}_{w_2}(s)\right). \]
\end{lemma}

\begin{proof}
    Applying Lemma \ref{lemma: relating two elements in iota+M} to $w'=t Q w$ and $w$, and multiplying \eqref{eqn: higher bfunction of w' via w} by $Q$, we obtain
    \[ QP \cdot (t Q w) = b_{w}^{(k)}(s)(Q w). \]
    It follows that $b_{Qw}(s) \mid b^{(k)}_w(s)$. Similarly, applying Lemma \ref{lemma: relating two elements in iota+M} to $w'=t(w_1+w_2)$ and $w=w_i$ for $i=1,2$ provides the equations
    \[ P_i \cdot t (w_1+w_2) = b^{(k)}_{w_i}(s) \cdot w_i, \quad \text{for } i=1,2. \]
    Multiplying these equations by $\mathrm{lcm}\left(b^{(k)}_{w_1}(s), b^{(k)}_{w_2}(s)\right)/b^{(k)}_{w_i}(s)$ respectively and summing them yields the second desired divisibility.
\end{proof}

\subsection{$p$-functions}\label{sec: pfunction} 
Recall that $\cM$ is a holonomic $\sD_X$-module such that $f$ acts bijectively. The notion of the \emph{$p$-function} of an element $m \in \cM$ is introduced in \cite[Definition 2.11]{LY25April}. In this section, we provide a systematic treatment of its properties. Below, we will use the isomorphism $\cM[s]f^s\cong \iota_{+}\cM$ from \eqref{eqn: Malgrange isomorphism} throughout. 

\begin{lemma}\label{lem:pideal}
  For a non-zero $w \in \iota_+ \cM$ and $\alpha \in \C$, the set
  \[ I_{w, \alpha} \colonequals \{p(s) \in \C[s] \mid \textrm{$b_{p(s)w}$ has no roots $>-\alpha$}\} \]
forms a non-zero ideal in $\C[s]$.
\end{lemma}

\begin{proof}
By Theorem \ref{thm: Sabbah}, $I_{m, \alpha}$ can be defined via the intersection:
    \begin{equation}\label{eqn: pfunction via intersection of Valpha with mfs}
        \C[s] w \, \cap \, V^\alpha\iota_{+}\cM = I_{w,\alpha}\cdot w.
    \end{equation}   
From this it is clear that $I_{w, \alpha}$ is an ideal. To see that it is non-zero, let $\beta \in \C$ such that $w \in V^\beta \iota_+ \cM$. We construct a non-zero element $p(s) \in I_{w, \alpha}$ inductively on $\beta$. If $\beta \geq \alpha$, then we can pick $p(s)=1$. Otherwise, Let $\nu:= \mult_{s=-\beta} b_w(s)$. By (\ref{bfunction of s+alphaw}), we have $w':=(s+\beta)^\nu w \in V^{>\beta}\iota_+\cM = V^{\beta'} \iota_+\cM$, for some $\beta'>\beta$. By the induction hypothesis, there exists a non-zero $p'(s) \in \C[s]$ such that $p'(s) w' \in V^\alpha \iota_+ \cM$. Then $p(s):=(s+\beta)^\nu \cdot p'(s)$ is the desired polynomial.
\end{proof}

\begin{definition}\label{definition: pfunction of m}
For a non-zero $w \in \iota_+ \cM$, we define the \emph{$p$-function} of $w$ with respect to $\alpha$, denoted by $p_{w,\alpha}(s)$, as the unique monic polynomial $p(s)\in \C[s]$ of minimal degree such that $p(s) w\in V^\alpha\iota_{+}\cM$; i.e., $(p_{w,\alpha}(s))=I_{w,\alpha}$, by Theorem \ref{thm: Sabbah}. When $w=m \cdot f^s$ with $m \in \cM$, we will put $p_{m,\alpha}(s):=p_{w,\alpha}(s)$.
\end{definition}

\begin{remark}\label{rem:shift}
    The proof of Lemma \ref{lem:pideal} together with (\ref{bfunction of s+alphaw}) shows the roots of $p_{w,\alpha}(s)$ are $\Z_{\leq 0}$-shifts of roots of $b_w(s)$.
\end{remark}

The following is clear since the $V$-filtration is decreasing.

\begin{lemma}\label{lem:smalldiv}
    Let $w \in \iota_+ \cM$. For $\alpha\leq \beta$, we have $p_{w, \alpha} (s) \, \mid \, p_{w,\beta}(s)$.
\end{lemma}

\begin{prop}\label{prop:bdivp}
    For $\textrm{Re}(\alpha) \gg 0$, we have $b_w(s) \, | \, p_{w,\alpha}(s)$.
\end{prop}

\begin{proof}
    Let $\cN:= \sD_{X\times \C} \cdot w \, \subset \iota_+ \cM$, which is a holonomic submodule. We pick $\textrm{Re}(\beta) \gg 0$ such that $t^k: V^{\beta}\cN \to V^{\beta+k} \cN$ is an isomorphism, for any $k \in \Z_{\geq 0}$. Since $V^{\beta}\cN$ is coherent over $\sD_X\langle s,t\rangle$, we can choose finitely many generators $w_i = P_i(t, \partial t) \cdot w\in V^{\beta}\cN$ where $P_i(t, \partial t) \in \sD_{X\times \C}$.  Then there exist an integer $k_0\geq 1$ such that $t^{k_0} \cdot P_i(t, \partial t) \in V^0\sD_{X\times \C}=\sD_X\langle s,t\rangle$, for all $i$. Consequently, we have 
    \[ V^{\beta+k_0} \cN=\sum_i \sD_X\langle s,t\rangle\cdot t^{k_0}P_i(t, \partial t)w\subseteq \sD_X\langle s,t\rangle \cdot w.\]
    Setting $\alpha_0 := \beta+k_0 + 1$, we have
    \[p_{w,\alpha_0}(s) \cdot w \in \cN \cap V^{\alpha_0}\iota_+ \cM  \, = V^{\alpha_0} \cN = t \cdot V^{\beta+ k_0} \cN \subset t \cdot  \sD_X\langle s,t\rangle w = \sD_X\langle s,t\rangle \cdot (tw).\]
    By definition, this implies that $b_w(s) \, \mid \, p_{w,\alpha_0}(s)$. Hence, $b_w(s) \, \mid \, p_{w,\alpha}(s)$ for any $\alpha \geq \alpha_0$ by Lemma \ref{lem:smalldiv}.
\end{proof}

\begin{corollary}\label{cor:multpfun}
    Let $w\in V^\beta \iota_+ \cM$ and set $\nu := \mult_{s=-\beta} b_w(s)$ and $w' = (s+\beta)^\nu \cdot w$. Then $w' \in V^{>\beta}\iota_+ \cM$, and for any $\alpha > \beta$ we have $p_{w,\alpha}(s) = (s+\beta)^\nu \cdot p_{w', \alpha}(s)$.
\end{corollary}

\begin{proof}
    It follows from the proof of Lemma \ref{lem:pideal} that we have $p_{w,\alpha}(s) \, \mid (s+\beta)^\nu \cdot p_{w',\alpha}(s)$ and $w' \in V^{>\beta}\iota_+ \cM$. Further, since all roots of $p(s)$ are $\leq - \beta$, by (\ref{bfunction of s+alphaw}) and Theorem \ref{thm: Sabbah} we must have $(s+\beta)^\nu \, \mid p_{w, \alpha}(s)$. Writing $p_{w, \alpha}(s) = (s+\beta)^\nu \cdot p(s)$, we get that $p(s) \cdot w' = p_{w,\alpha}(s) \cdot w \in V^\alpha \iota_+ \cM$, thus $p_{w',\alpha}(s) \, \mid p(s)$.
\end{proof}

\begin{lemma}\label{lem:pfunsum} We have the following:
\begin{itemize}
    \item[(i)] For $w_1, w_2\in \iota_+\cM$, we have $p_{w_1+w_2, \alpha}(s) \, \mid \, \mathrm{lcm}(p_{w_1, \alpha}(s), p_{w_2, \alpha}(s))$.
    \item[(ii)] For $w \in \iota_+ \cM$ and $Q(s) \in \sD_X[s]$, we have $p_{Q(s)\cdot w, \alpha}(s) \, \mid \, p_{w,\alpha}(s)$.
\end{itemize}

\end{lemma}

\begin{proof}
   Since $p_{w_i, \alpha}(s) \, \mid \, \mathrm{lcm}(p_{w_1, \alpha}(s), p_{w_2, \alpha}(s))$, we have $\mathrm{lcm}(p_{w_1, \alpha}(s), p_{w_2, \alpha}(s)) \cdot w_i \in V^\alpha \iota_+ \cM$. But then $\mathrm{lcm}(p_{w_1, \alpha}(s), p_{w_2, \alpha}(s)) \cdot (w_1+w_2) \in V^\alpha \iota_+ \cM$, which proves part (i).   For part (ii), since $p_{w,\alpha}(s) w \in V^{\alpha}\iota_+ \cM$, multiplying this by $Q(s)$ yields the conclusion.
\end{proof}

\begin{corollary}\label{cor:bfunsum}We have the following:
\begin{itemize}
    \item[(i)] For $w_1, w_2\in \iota_+\cM$, any root of $b_{w_1+w_2}(s)$ is a $\Z_{\leq 0}$-shift of a root of $b_{w_1}(s)$ or $b_{w_2}(s)$. 
    \item[(ii)] For $w \in \iota_+ \cM$ and $Q(s) \in \sD_X[s]$, any root of $b_{Q(s)\cdot w}$ is a $\Z_{\leq 0}$-shift of a root of $b_{w}(s)$.
    \end{itemize}
\end{corollary}

\begin{proof}
Pick $\alpha$ with $\textrm{Re}(\alpha) \gg 0$. By Remark \ref{rem:shift}, roots of $p_{w_i, \alpha}(s)$ are $\Z_{\leq 0}$-shifts of roots of $b_{w_i}(s)$. It follows by Lemma \ref{lem:pfunsum} (i) that a root of $p_{w_1+w_2, \alpha}(s)$ is a $\Z_{\leq 0}$-shift of a root of $b_{w_1}(s)$ or $b_{w_2}(s)$. But by Proposition \ref{prop:bdivp}, the same is true for $b_{w_1+w_2}(s)$, proving part (i). Part (ii) follows analogously using Lemma \ref{lem:pfunsum} (ii) instead.
\end{proof}

\begin{corollary}\label{cor:pideal}
  For a non-zero $w \in \iota_+ \cM$ and $\alpha \in \C$, the set
\[ \tilde{I}_{w, \alpha} \colonequals \{\tilde{p}(s) \in \C[s] \mid \textrm{$b_{\tilde{p}(s)w}(s)$ has no roots in $-\alpha + \Z_{>0}$} \} \]
is a non-zero ideal, and $I_{w, \alpha} \subseteq \tilde{I}_{w,\alpha}$.  
\end{corollary}

\begin{proof}
Clearly,  $0\neq I_{w, \alpha} \subseteq \tilde{I}_{w,\alpha}$. By Corollary \ref{cor:bfunsum} (i), $\tilde{I}_{w,\alpha}$ is closed under addition, and by Corollary \ref{cor:bfunsum} (ii), it is closed under multiplication by elements in $\C[s]$.
\end{proof}

\begin{definition}\label{def: reducedpfunction}
    For $w \in \iota_+ \cM$, the \emph{reduced $p$-function} $\tilde{p}_{w,\alpha}(s)$ is defined as the unique monic polynomial $\tilde{p}(s) \in \C[s]$ of minimal degree such that $b_{\tilde{p}(s)w}(s)$ has no roots in $-\alpha + \Z_{>0}$; that is, $(\tilde{p}_{w,\alpha}(s)) = \tilde{I}_{w,\alpha}$. When $w=m \cdot f^s$ with $m \in \cM$, put $\tilde{p}_{m,\alpha}(s):=\tilde{p}_{w,\alpha}(s)$. 
    
    Define $\nu_{m, \alpha} \colonequals \mu_{\tilde{p}_{m,\alpha}(s) \cdot m, \alpha}$, i.e., the multiplicity of $-\alpha$ as a root of $b_{\tilde{p}_{m,\alpha}(s)\cdot m}(s)$. 
\end{definition}

 Clearly, $\tilde{p}_{w,\alpha}(s)  \, \mid \, p_{w,\alpha}(s)$. The proposition below, a generalization of Corollary \ref{cor:tozfilt}, demonstrates that $\nu_{m,\alpha}$ also equals the multiplicity of $-\alpha$ as a root of $b_{p_{m,\alpha}(s)\cdot m}(s)$.

\begin{prop}\label{prop:tozfilt}
 Let $w \in \iota_+ \cM$ such that $b_w(s)$ has no roots in $-\alpha + \Z_{>0}$. Then $b_{q(s)w}(s)$ satisfies the same property for any $q(s)\in \C[s]$, and
  \[\mu_{q(s)\cdot w, \alpha} \, = \, \max \{ \mu_{w,\alpha}- \mult_{s=-\alpha} q(s), \, 0 \}.\]
  Furthermore, $p_{w,\alpha}(s) $ has no roots in $-\alpha + \Z_{\geq 0}$.
\end{prop}

\begin{proof}
    The recursive property in Corollary \ref{cor:multpfun}, together with (\ref{bfunction of s+alphaw}), assures that $p_{w,\alpha}(s)$ has no roots in $-\alpha + \Z_{\geq 0}$. By Lemma \ref{lem:easy2}, we have $\mu_{w, \alpha}= \mu_{p_{w,\alpha}(s) w, \alpha}$, and  $\mu_{q(s)w, \alpha}= \mu_{p_{w,\alpha}(s) q(s) w, \alpha}$. Since $p_{w,\alpha}(s) w \in V^{\alpha}\iota_+ \cM$, the claim follows from Corollary \ref{cor:tozfilt}.
\end{proof}

Here and throughout, for $\gamma\in\C$, we write
\[
\lceil \gamma\rceil:=\min\{k\in\Z\mid \gamma\le k\},
\]

\begin{lemma}\label{lemma: basic property of pfunction} 
    Let $m \in \cM$ and $\alpha \in \C$. Write $b_m(s) =\prod_{i} (s+r_i)$. Then we have:
    \begin{enumerate}        
        \item[(i)] We have 
        \begin{equation*}\label{eqn: factors of pfunction}
            p_{m, \alpha}(s) \,\, \mid \,\, \prod_{r_i<\alpha} [s+r_i]_{\lceil \alpha-r_i\rceil}, \quad \tilde{p}_{m, \alpha}(s) \,\, \mid \,\, \prod_{r_i\in\alpha-\Z_{>0}} [s+r_i]_{ \alpha-r_i}.
        \end{equation*}   
        In particular, $(s+\alpha)$ divides neither $p_{m, \alpha}(s)$ nor $\tilde{p}_{m, \alpha}(s)$.
        \item[(ii)] $\nu_{m,\alpha} \leq \#\{ i \mid \alpha-r_i \in \Z_{\geq 0}\}$, with equality only if $\tilde{p}_{m, \alpha}(s)=\prod_{r_i \in \alpha - \Z_{>0}} [s+r_i]_{\alpha-r_i}$.
        \item[(iii)] If $q(s)$ is a multiple $\tilde{p}_{m,\alpha}(s)$, then 
        \[ \mu_{q(s)\cdot m, \alpha} = \max\{\nu_{m,\alpha}-\underset{s=-\alpha}{\mathrm{mult}}\, q(s), \, 0 \}. \]
        In particular, $\nu_{m,\alpha}=\mu_{p_{m,\alpha}(s)m, \alpha}$.
        \item[(iv)] $p_{f \cdot m,\alpha+1}(s)=p_{m,\alpha}(s+1)$ and $\nu_{f \cdot m, \alpha+1}=\nu_{m,\alpha}$.
        \item[(v)] $p_{m, \alpha}(s+1) \mid p_{m, \alpha+1}(s)$ and $\nu_{m,\alpha} \leq \nu_{m,\alpha+1}$.
    \end{enumerate}
\end{lemma}

\begin{proof}
For (i), the divisibility for $p_{m,\alpha}(s)$ follows because $\prod_{i=1}^d [s+r_i]_{\lceil \alpha-r_i\rceil} \in I_{m, \alpha}$ by \eqref{eqn: bpsm divides bm}. The argument for $\tilde{p}_{m,\alpha}(s)$ is analogous. Property (ii) is a direct consequence of (i) and Proposition \ref{prop:bgcd}.

Property (iii) follows directly from Proposition \ref{prop:tozfilt}. Property (iv) follows from Lemma \ref{lem: b function of tw} and the fact that multiplication by $t$ induces an isomorphism $V^\alpha\iota_{+}\cM \xrightarrow{\sim} V^{\alpha+1}\iota_{+}\cM$ (e.g., see Lemma \ref{lemma: shift of V filtration by talpha}).

For (v), part (iv) and Lemma \ref{lem:pfunsum} immediately imply 
\[ p_{m, \alpha}(s+1) = p_{fm, \alpha+1}(s) \mid p_{m, \alpha+1}(s). \]
Thus, we can write $p_{m,\alpha}(s)=q(s) p_{fm,\alpha}(s)$ for some $q(s)\in \C[s]$.  Let $w=p_{m,\alpha}(s)m f^s$. Since $w\in V^\alpha\iota_{+}\cM$ by definition and Theorem \ref{thm: Sabbah}, Corollary \ref{cor:multpfun} and Lemma \ref{lem:pfunsum} dictate that $\nu_{m,\alpha}=\mu_{w,\alpha}\geq \mu_{fw,\alpha}$. By part (i), $(s+\alpha) \nmid q(s)$, so Corollary \ref{cor:tozfilt} and part (iv) imply that 
\[ \mu_{fw,\alpha} = \mu_{q(s)p_{fm,\alpha}(s)fm,\alpha} = \mu_{p_{fm,\alpha}(s)fm,\alpha} = \nu_{fm,\alpha} = \nu_{m,\alpha-1}. \]
Therefore, shifting the index gives $\nu_{m,\alpha+1} \geq \nu_{m,\alpha}$.
\end{proof}

\begin{prop}\label{prop:pfunvaryalpha}
Let $m \in \cM$ and $\alpha \in \C$. Then\begin{enumerate}
    \item Let $\beta\in \C$ be the minimal complex number such that $V^{>\alpha}\iota_+ \cM = V^\beta \iota_+ \cM$. Then
    \[ p_{m, \beta}(s) = (s+\alpha)^{\nu_{m,\alpha}} \cdot p_{m, \alpha}(s). \]
    \item $\tilde{p}_{m, \alpha+1}(s) = (s+\alpha)^{\nu_{m, \alpha}} \cdot \tilde{p}_{m,\alpha}(s)$.
\end{enumerate}
\end{prop}

\begin{proof}
    We give the proof for $p_{m,\alpha}(s)$; the argument for $\tilde{p}_{m, \alpha}(s)$ is analogous. By Lemma \ref{lem:smalldiv} we have $p_{m,\alpha}(s) \mid p_{m, \beta}(s)$. Conversely, because $p_{m, \alpha}(s)mf^s \in V^{\alpha}\iota_+ \cM$, Lemma \ref{lemma: basic property of pfunction} (iii) and Theorem \ref{thm: Sabbah} dictate that 
    \[ (s+\alpha)^{\nu_{m,\alpha}} \cdot p_{m, \alpha}(s)mf^s \in V^{>\alpha}\iota_+ \cM = V^\beta \iota_+ \cM.\]
    This shows that $p_{m, \beta}(s) \mid (s+\alpha)^{\nu_{m,\alpha}} \cdot p_{m, \alpha}(s)$, so we can write
    \[ p_{m, \beta}(s) = (s+\alpha)^{k} \cdot p_{m, \alpha}(s) \]
    for some integer $0\leq k \leq \nu_{m,\alpha}$. Because $\beta>\alpha$ and $p_{m,\beta}(s)mf^s \in V^\beta \iota_+ \cM$, by Theorem \ref{thm: Sabbah} $-\alpha$ is not a root of the $b$-function of $p_{m,\beta}(s)mf^s$. By Lemma \ref{lemma: basic property of pfunction}(iii), we have $0 =\mu_{p_{m, \beta}(s) \cdot m, \alpha} = \nu_{m,\alpha} - k$, forcing $k = \nu_{m,\alpha}$ and completing the proof.
\end{proof}

\begin{corollary}\label{cor:explicitpfunction}
    For any $m\in \cM$, we have 
    \begin{equation}\label{eqn: explicit formula of p function via bfunction}
        p_{m,\alpha}(s) = \prod_{\beta< \alpha} (s+\beta)^{\nu_{m,\beta}}.
    \end{equation}
    Furthermore, let $-\tilde{\gamma}$ be the largest root of $b_m(s)$ in $-\alpha+\Z_{>0}$ (if no such root exists, set $\tilde{p}_{m, \alpha}(s) = 1$). Then
    \begin{equation}\label{eqn: explicit reduced rearranged}
        \tilde{p}_{m,\alpha}(s) = \prod_{i=0}^{\alpha - \tilde{\gamma}-1} [s+\tilde{\gamma}+i]_{\alpha - \tilde{\gamma}-i}^{\nu_{m,\tilde{\gamma}+i}-\nu_{m,\tilde{\gamma}+i-1}} = \prod_{i=0}^{\alpha - \tilde{\gamma}-1} (s+\tilde{\gamma}+i)^{\nu_{m,\tilde{\gamma}+i}}.
    \end{equation}
\end{corollary}

\begin{proof}
    We first consider $p_{m,\alpha}(s)$. Let $\gamma\in \C$ be the smallest complex number such that $mf^s\in V^{\gamma}\iota_{+}\cM \setminus V^{>\gamma}\iota_{+}\cM$. By Theorem \ref{thm: Sabbah}, $-\gamma$ is the largest root of $b_m(s)$. If $\alpha\leq \gamma$, then $m\in V^{\gamma}\iota_{+}\cM\subseteq V^{\alpha}\iota_{+}\cM$, which implies $p_{m,\alpha}(s)=1$ and $\nu_{m,\alpha}=\mu_{m,\alpha}$. If $\alpha>\gamma$, let the jumping numbers of $V^{\bullet}\iota_{+}\cM$ between $\gamma$ and $\alpha$ be 
    \[ \gamma < \gamma_1 < \cdots < \gamma_k = \alpha. \]
    Repeatedly applying Proposition \ref{prop:pfunvaryalpha}(1) yields $p_{m,\alpha}(s)=\prod_{i=1}^{k-1} (s+\gamma_i)^{\nu_{m,\gamma_i}}$, proving \eqref{eqn: explicit formula of p function via bfunction}. 

    For $\tilde{p}_{m,\alpha}(s)$, repeated application of Proposition \ref{prop:pfunvaryalpha}(2) yields
    \begin{equation}\label{eqn: explicit formula of reduced}
        \tilde{p}_{m,\alpha}(s) = \prod_{i=0}^{\alpha - \tilde{\gamma}-1} (s+\tilde{\gamma}+i)^{\nu_{m,\tilde{\gamma}+i}}.
    \end{equation}
 Note that Lemma \ref{lemma: basic property of pfunction}(v) implies the sequence of exponents $\nu_{m, \tilde{\gamma}+i}$ is non-decreasing with respect to $i$. Hence using \eqref{eqn: [s+a]} we can rewrite \eqref{eqn: explicit formula of reduced} into the form of \eqref{eqn: explicit reduced rearranged}.
\end{proof}

As a consequence, we obtain the following bounds for $\nu_{m, \alpha}$. In particular, the upper bound provides a recursive strengthening of Lemma \ref{lemma: basic property of pfunction}(ii). 

\begin{prop}\label{prop:nurecursivebound}
    We have 
    \[ \max\{\nu_{m, \alpha-1} , \mu_{m, \alpha}\} \leq \nu_{m, \alpha} \leq \nu_{m, \alpha-1}+ \mu_{m, \alpha}. \]
\end{prop}

\begin{proof}
    Lemma \ref{lemma: basic property of pfunction}(v) states that $\nu_{m, \alpha-1} \leq  \nu_{m, \alpha}$. The inequality $\mu_{m,\alpha} \leq \nu_{m, \alpha}$ follows from Lemma \ref{lem:easy2} and the fact that $(s+\alpha)\nmid \tilde{p}_{m,\alpha}(s)$. This establishes the lower bound.
    
    For the upper bound, we proceed by induction in $\alpha+\Z$. For sufficiently small $\alpha \ll 0$, all terms vanish and the statement trivially holds. Assume the inequality holds for $\alpha - i$ for all $i\in \Z_{>0}$. Then, setting $-\tilde{\gamma}$ to be the largest root of $b_m(s)$ in $-\alpha+\Z_{>0}$, using \eqref{eqn: explicit reduced rearranged} and \eqref{eqn: bpsm divides bm}, we find that
    \[ \nu_{m, \alpha} \leq \mu_{m, \alpha} + \nu_{m,\tilde{\gamma}} + (\nu_{m,\tilde{\gamma}+1} - \nu_{m,\tilde{\gamma}}) + \dots + (\nu_{m,\alpha-1} - \nu_{m,\alpha -2 }) = \mu_{m, \alpha} + \nu_{m, \alpha-1}. \qedhere\]
\end{proof}

\begin{thm}\label{thm: numalpha via higher order pfunction}
    For $\alpha \in \C$ and $m\in \cM$, let $k_0\in \Z_{>0}$ be the minimal integer such that $b_m(k- \alpha)\neq 0$ for all $k\geq k_0$. Then for any $k \geq k_0$,
    \[ \nu_{m, \alpha} = \mult_{s=-\alpha} b^{(k)}_m(s). \]
\end{thm}
    
\begin{proof}
    Fix any $k\geq k_0$. By Lemma \ref{lemma: basic property of pfunction}(iii), $\nu_{m, \alpha} = \mu_{p_{m,\alpha}(s)m,\alpha}$. Moreover, Corollary \ref{cor:explicitpfunction} dictates that for $\ell>0$
    \[  \mult_{s=-\alpha}p_{m,\alpha}(s)=0 \quad \text{and} \quad  \mult_{s=-(\alpha-\ell)}p_{m,\alpha}(s)= \nu_{m,\alpha-\ell}. \]
    Applying Lemma \ref{lem:localized} to the polynomial $p_{m,\alpha}(s)$, we can characterize $\nu_{m,\alpha}$ as the smallest integer $\mu \in \Z_{\geq 0}$ such that 
    \begin{equation}\label{eq:charloc}
        (s+\alpha)^{\mu} \cdot m f^s \in \sum_{\ell=1}^{k} \sD_X[s]_{(s+\alpha)} \cdot (s+\alpha)^{\nu_{m, \alpha-\ell}} \cdot m f^{s+\ell}.
    \end{equation}
    We claim that the sum on the right-hand side simplifies significantly:
    \begin{equation}\label{eq:claimequal of sum from k=1 to k' with k'}
        \sum_{\ell=1}^{k} \sD_X[s]_{(s+\alpha)} \cdot (s+\alpha)^{\nu_{m, \alpha-\ell}} \cdot m f^{s+\ell} = \sD_X[s]_{(s+\alpha)} \cdot (s+\alpha)^{\nu_{m, \alpha-k}} \cdot m f^{s+k}.
    \end{equation}

    Clearly, the right-hand module is contained in the left, so it suffices to prove the reverse inclusion. For each $1\leq i\leq k-1$, our choice of $k_0$ ensures $b_m((k-i)-(\alpha-i))\neq 0$. Thus, applying \eqref{eq:charloc} shifted by $\alpha \mapsto \alpha-i$ yields
    \[ \sD_X[s]_{(s+\alpha-i)} \cdot(s+\alpha-i)^{\nu_{m,\alpha-i}} \cdot m f^s \subseteq \sum_{\ell=1}^{k-i} \sD_X[s]_{(s+\alpha-i)} \cdot (s+\alpha-i)^{\nu_{m, \alpha-i-\ell}} \cdot m f^{s+\ell}. \]
    Applying the shift operator $t^i$ to both sides, we obtain
    \[ \sD_X[s]_{(s+\alpha)} \cdot(s+\alpha)^{\nu_{m,\alpha-i}} \cdot m f^{s+i} \subseteq \sum_{\ell=i+1}^{k} \sD_X[s]_{(s+\alpha)} \cdot (s+\alpha)^{\nu_{m, \alpha-\ell}} \cdot m f^{s+\ell}. \]
    Consequently, we obtain a chain of inclusions:
    \begin{align*}
        \sum_{\ell=1}^{k} \sD_X[s]_{(s+\alpha)} \cdot (s+\alpha)^{\nu_{m, \alpha-\ell}} \cdot m f^{s+\ell} &\subseteq \sum_{\ell=2}^{k} \sD_X[s]_{(s+\alpha)} \cdot (s+\alpha)^{\nu_{m, \alpha-\ell}} \cdot m f^{s+\ell} \subseteq \dots \\
        \dots &\subseteq \sD_X[s]_{(s+\alpha)} \cdot (s+\alpha)^{\nu_{m, \alpha-k}} \cdot m f^{s+k}.
    \end{align*}
    This establishes \eqref{eq:claimequal of sum from k=1 to k' with k'}.  By the assumption on $k$, we have $\nu_{m, \alpha-k}=0$. Combining this with our earlier characterization, we see that $\nu_{m,\alpha}$ is the smallest $\mu\in \Z_{\geq 0}$ such that
    \[ (s+\alpha)^\mu \cdot m f^s  \in \sD_X[s]_{(s+\alpha)} \cdot m f^{s+k}. \]
    By Definition \ref{definition: higher order bfunction}, this minimal $\mu$ is precisely $\mult_{s=-\alpha}b^{(k)}_m(s)$.
\end{proof}

\begin{remark}\label{remark:nu11is mu}
    Consider $1\in \cM=(\cO_X)_f$. By Kashiwara's theorem, the $b$-function $b_1(s)$ (which is the standard Bernstein-Sato polynomial of $f$) has only strictly negative roots. Thus, Theorem \ref{thm: numalpha via higher order pfunction} and Lemma \ref{lemma: basic property of pfunction}(iv) imply that
    \begin{equation}\label{eqn: nu11 is mult}
        \nu_{f^{-1},0} = \nu_{1,1} = \mult_{s=-1}b_f(s).
    \end{equation}
\end{remark}

\section{$V$-filtration and localization of $\cM$} 
In this section, we establish some new relations between the $V$-filtrations with the twists of a holonomic $\sD$-module. The main goal is to relate the last map in \eqref{eqn: coker of grValpha is MquotientbyS} with the evaluation map $ev_{s=-\alpha}:\cM[s]f^s\to \cM\cdot f^{-\alpha}$. We also put several additional filtrations on the $V$-filtrations. 

Let $f$ be an algebraic function on a smooth complex algebraic variety $X$ and set $D=\mathrm{div}(f)$. Let $\cM$ be a holonomic $\sD$-module acted bijectively by $f$.

\subsection{Twists of $\cM$ via $V$-filtration}\label{sec: Mf-alpha as the image}
Let $\alpha \in \C$. Recall the $\sD_X[s]$-module $\cM[s] f^s$, which fits into the following short exact sequence of $\sD_X[s]$-modules:
\begin{equation}\label{eq:basic}
    0 \to \cM[s] f^s \xrightarrow{s+\alpha} \cM[s]f^s \xrightarrow{s=-\alpha} \cM \cdot f^{-\alpha} \to 0.
\end{equation}
Denote by $ev_{s=-\alpha}:\iota_{+}\cM\xrightarrow{\sim}\cM[s]f^s\xrightarrow{s=-\alpha}  \cM \cdot f^{-\alpha}$ the evaluation map, where the first isomorphism is from \eqref{eqn: Malgrange isomorphism}. It turns out that in \eqref{eq:basic} we can replace $\cM[s]f^s$ with any $V^{\beta}\iota_{+}\cM$ as long as $\beta\leq \alpha$. 

\begin{thm}\label{thm: ses of Valpha and Mfalpha}
Let $\beta\leq \alpha$, then there is a short exact sequence of $\sD_X[s]$-modules
  \begin{equation}
  0\to V^{\beta}\iota_{+}\cM \xrightarrow{s+\alpha} V^{\beta}\iota_{+}\cM \xrightarrow{ev_{s=-\alpha}} \cM\cdot f^{-\alpha}\to 0.
  \end{equation}
\end{thm}
 
\begin{proof}[Proof of Theorem \ref{thm: ses of Valpha and Mfalpha}]
   First by \eqref{eq:basic}, we have 
   \[\ker ev_{s=-\alpha}|_{V^\beta \iota_+ \cM}= (s+\alpha)\cdot \iota_+ \cM \cap V^\beta \iota_+ \cM\supseteq (s+\alpha) \cdot V^\beta \iota_+ \cM.\]
  For the other inclusion, consider any element $w \in \ker ev_{s=-\alpha}|_{V^\beta \iota_+ \cM}$. We can write $w=(s+\alpha)w'$ for some $w' \in \iota_+ \cM$ by \eqref{eq:basic}. Then by Lemma \ref{lem:easy2}, $b_{w'}(s) \mid (s+\alpha)b_{w}(s)$. Since $w \in V^\beta \iota_+ \cM$, Theorem \ref{thm: Sabbah} says that all roots of $b_{w}(s)$ are $\leq -\beta$ and so the same hold for the roots of $b_{w'}(s)$ as $\beta\leq \alpha$. It follows that $w' \in V^\beta \iota_+ \cM$ and thus $\ker ev_{s=-\alpha}|_{V^\beta \iota_+ \cM} = (s+\alpha)\cdot V^\beta \iota_+ \cM.$

Now we show the surjectivity of $ev_{s=-\alpha}$. Let $m \in \cM$, using Definition \ref{definition: pfunction of m} and Theorem \ref{thm: Sabbah}, we see that $p_{m,\alpha}(s)m \cdot f^s \in V^\alpha \iota_+ \cM\subseteq V^{\beta}\iota_{+}\cM$. Moreover, Lemma \ref{lemma: basic property of pfunction}(i) implies that $p_{m,\alpha}(-\alpha)\neq 0$. This induces a set-theoretic map
  \begin{align}\label{eqn: section from Mfalpha to Valpha}
   \sigma: \cM\cdot f^{-\alpha}&\to V^{\beta}\iota_{+}\cM,\\ \nonumber
   m f^{-\alpha}&\mapsto \frac{p_{m,\alpha}(s)}{p_{m,\alpha}(-\alpha)}m f^s.
  \end{align}
 Clearly one has $ev_{s=-\alpha}\circ \sigma=\mathrm{id}_{\cM\cdot f^{-\alpha}}$, proving the surjectivity of $ev_{s=-\alpha}$. 
\end{proof}

\begin{corollary}\label{cor: decomposition of elements in grV in terms of pfunction}
    Any element $[w]\in \gr_V^\alpha \iota_+ \cM $ has an expansion of the form 
    \[ [w]= \sum_{i=0}^{n_\alpha-1} (s+\alpha)^i \cdot [p_{m_i, \alpha}(s) m_i f^s],\]
    with $m_i \in \cM$; here $n_{\alpha}$ is defined in Definition \ref{definition:nilpotence degree}. In particular, then
    \[\gr_V^\alpha \iota_+ \cM = \mathrm{span}_{\C} \left\{ \frac{\C[s]}{(s+\alpha)^{\nu_{m,\alpha}}} \cdot [p_{m,\alpha}(s)  m f^s] \, \mid \, m \in \cM 
    \right\}.\]  
\end{corollary}

\begin{proof}
Using the map $\sigma$ in \eqref{eqn: section from Mfalpha to Valpha}, we define the elements $w_i \in V^\alpha \iota_+ \cM$ recursively by putting $w_0 = w$ and 
\[(s+\alpha)w_{i+1}\colonequals w_i-\sigma\circ ev_{s=-\alpha}(w_i),\quad \textrm{for $0\leq i\leq n_\alpha-1$}.\]
This is well-defined because $ev_{s=-\alpha}\left(w_i-\sigma\circ ev_{s=-\alpha}(w_i)\right)=0$, and hence Theorem \ref{thm: ses of Valpha and Mfalpha} implies that $w_i-\sigma\circ ev_{s=-\alpha}(w_i)\in \mathrm{Im}(s+\alpha)$. Set $n_i\in\cM$ such that $ev_{s=-\alpha}(w_i)=n_i\cdot f^{-\alpha}$ and $m_i=\frac{n_i}{p_{n_i,\alpha}(-\alpha)}\in \cM$. Consequently,
\begin{align*}
w&=\sum_{i=0}^{n_\alpha-1} (s+\alpha)^i\cdot \sigma\circ ev_{s=-\alpha}(w_i)+(s+\alpha)^{n_\alpha}\cdot w_{n_\alpha}\\
&=\sum_{i=0}^{n_\alpha-1} (s+\alpha)^i\cdot \frac{p_{n_i,\alpha}(s)}{p_{n_i,\alpha}(-\alpha)}n_i f^s+(s+\alpha)^{n_\alpha}\cdot w_{n_\alpha},\\
&=\sum_{i=0}^{n_\alpha-1} (s+\alpha)^i\cdot p_{m_i,\alpha}(s) m_i f^s+(s+\alpha)^{n_\alpha}\cdot w_{n_\alpha}.
\end{align*}
The last equality uses $p_{\lambda m,\alpha}(s)=p_{m,\alpha}(s)$ for any $\lambda\neq 0$ by Definition \ref{definition: pfunction of m}. Since $(s+\alpha)^{n_\alpha}\cdot w_{n_\alpha} \in V^{>\alpha}\iota_+ \cM$, this proves the first statement. For the last statement, we use Lemma \ref{lem:easy2} and Theorem \ref{thm: Sabbah} to show that $(s+\alpha)^{\mu_{p_{m,\alpha}(s)m,\alpha}}p_{m,\alpha}(s)mf^s\in V^{>\alpha}\iota_{+}\cM$.
\end{proof}

\begin{corollary}\label{cor:nilpdegmax}
We have $n_\alpha = \max_{m\in \cM} \nu_{m,\alpha}$.
\end{corollary}
\begin{proof}
It is an immediate consequence of Corollary \ref{cor: decomposition of elements in grV in terms of pfunction} and Lemma \ref{lemma: basic property of pfunction}(iii).
\end{proof}

In the remaining part of this section, we relate the last map in \eqref{eqn: coker of grValpha is MquotientbyS} with the evaluation map $ev_{s=-\alpha}$.

The following statement can be found in \cite[Proposition 2.49 and Remark 5.21]{DY26} (for $r=1$); see also \cite[Propositions 11.3.3 and 11.4.2]{MHMproject}.
\begin{lemma}\label{lemma: j*andj! in terms of Vfiltration}
    Let $Y$ be a smooth complex algebraic variety, $g:Y\to \C$ is an algebraic function such that $D=g^{-1}(0)$ is smooth and  $\cN$ is a holonomic $\sD_Y$-module.  Then \[ \cN(\ast D)=\sD_Y\otimes_{V^0\sD_Y}V^0\cN(\ast D), \quad \cN(!D)=\sD_Y\otimes_{V^0\sD_Y}V^{>0}\cN(\ast D).\]
\end{lemma}

\begin{lemma}\label{lemma: commutative diagram of Vj*andj!}
    There exists a commutative diagram
    \[\begin{tikzcd}
           V^{>\alpha}\iota_{+}\cM \arrow[r,"\widetilde{ev}_{s=-\alpha}"]\arrow[d,hook]& \cM\cdot f^{-\alpha}(!D) \arrow[d,"IC"]\\
       V^{\alpha}\iota_{+}\cM\arrow[r,"ev_{s=-\alpha}"] & \cM\cdot f^{-\alpha},
    \end{tikzcd}
    \]
where the vertical map on the right is induced by the natural map $(-)(!D)\xrightarrow{IC} (-)(\ast D)$.   
\end{lemma}

\begin{proof}
First we claim that there is a $\sD_{X\times \C}$-module isomorphism
    \begin{equation*}
    \iota_{+}\left(\cM\cdot f^{-\alpha}(!D)\right)\xrightarrow{\sim}\sD_{X\times \C}\otimes_{V^0\sD_{X\times \C}}V^{>\alpha}\iota_{+}\cM,\end{equation*}
which induces a morphism of $\sD_X$-modules 
\begin{equation}\label{eqn: from V>alpha to j!} \widetilde{ev}_{s=-\alpha}: V^{>\alpha}\iota_{+}\cM \to \cM\cdot f^{-\alpha}(!D).\end{equation}
Consider the following commutative diagram
\[ \begin{tikzcd}
    U\colonequals f^{-1}(\C^{\ast}) \arrow[r,"j"] \arrow[d,"\iota|_U"] & X \arrow[d,"\iota"]\\
    U'\colonequals X\times \C^{\ast} \arrow[r,"j'"] &X \times \C=:X'    
\end{tikzcd}
\]
Since the complement $D'\colonequals X'\setminus U'$ is a smooth divisor, Lemma \ref{lemma: j*andj! in terms of Vfiltration} and \cite[Remark 11.4.10]{MHMproject} imply that
\begin{align*}
    \iota_{+}\left(\cM\cdot f^{-\alpha}(!D)\right)&\xrightarrow{\sim}\left(\iota_{+}(\cM\cdot f^{-\alpha})\right)(!D')\\
    &=\sD_{X\times \C}\otimes_{V^0\sD_{X\times \C}}V^{>0}\iota_{+}(\cM\cdot f^{-\alpha})\\
    &\cong \sD_{X\times \C}\otimes_{V^0\sD_{X\times \C}}V^{>\alpha}\iota_{+}\cM
\end{align*}
Lemma \ref{lemma: shift of V filtration by talpha} gives the last isomorphism. 
Therefore we can construct a $\sD_X$-module map
\[ \widetilde{ev}_{s=-\alpha}: V^{>\alpha}\iota_{+}\cM \to \sD_{X\times \C}\otimes_{V^0\sD_{X\times \C}}V^{>\alpha}\iota_{+}\cM\xrightarrow{\sim}\iota_{+}\left(\cM\cdot f^{-\alpha}(!D)\right)\to \cM\cdot f^{-\alpha}(!D). \]
The last map is the projection map $\iota_{+}\cN\cong \cN\otimes_{\C}\C[\d_t]\to \cN$ for any $\sD$-module $\cN$.

Similarly, we can construct a map
    \[ V^{\alpha}\iota_{+}\cM \hookrightarrow \sD_{X\times \C}\otimes_{V^0\sD_{X\times \C}}V^{\alpha}\iota_{+}\cM\xrightarrow{\sim}\iota_{+}(\cM\cdot f^{-\alpha})\to  \cM\cdot f^{-\alpha}.\]
We claim that it coincides with the map $ev_{s=-\alpha}:V^{\alpha}\iota_{+}\cM\to \cM\cdot f^{-\alpha}$ from Theorem \ref{thm: ses of Valpha and Mfalpha}. Indeed, via Lemma \ref{lemma: shift of V filtration by talpha}, this map can be identified with
    \[V^{\alpha}\iota_{+}\cM\xrightarrow{\sim}V^0\iota_{+}(\cM\cdot f^{-\alpha})\hookrightarrow \iota_{+}(\cM\cdot f^{-\alpha})\cong \cM\cdot f^{-\alpha}[s]f^s\xrightarrow{s=0} \cM\cdot f^{-\alpha}.\]
Now the composition of the last two morphisms coincides with the natural projection map $\iota_{+}(\cM\cdot f^{-\alpha})\to \cM\cdot f^{-\alpha}$. Therefore, it suffices to prove there is a commutative diagram 
    \[\begin{tikzcd}
        V^{0}\iota_{+}(\cM\cdot f^{-\alpha}) \arrow[r,hook]\arrow[d,"="]& \sD_{X\times \C}\otimes_{V^0\sD_{X\times \C}}V^{0}\iota_{+}(\cM\cdot f^{-\alpha}) \arrow[d,"\cong"]\\
       V^0\iota_{+}(\cM\cdot f^{-\alpha}) \arrow[r,hook] &\iota_{+}(\cM\cdot f^{-\alpha}).
    \end{tikzcd}\]
But this is straightforward, since the top map sends $w$ to $1\otimes w$, and the isomorphism on the right is given by multiplication $P\otimes w\mapsto P\cdot w$. 

Finally, the commutativity of the diagram is direct to check. This finishes the proof.
\end{proof}

From now on, assume $\cS$ is a simple holonomic $\sD_X$-module and $\cM=\cS_f\neq 0$.
\begin{lemma}\label{lem:evsimple}
     For any $0\neq m \in \cM$ such that $\nu_{m,\alpha}=0$, we have $(V^{>\alpha}\iota_+ \cM)_{(s+\alpha)} = \sD_X[s]_{(s+\alpha)} \cdot mf^s$. 
\end{lemma}

\begin{proof}
Let $w\in V^{>\alpha}\iota_{+}\cM$. It follows from Lemma \ref{lemma: relating two elements in iota+M} that there exists some $k\in \Z_{>0}$so that $b_{w}^{(k)}(s)w=P\cdot mf^s$ for some $P\in \sD_X[s]$. Since $w\in V^{>\alpha}\iota_{+}\cM$,  Theorem \ref{thm: Sabbah} and Lemma \ref{lem:bfunpower} imply that $b_w^{(k)}(-\alpha)\neq 0$, so $(V^{>\alpha}\iota_+ \cM)_{(s+\alpha)} \subseteq \sD_X[s]_{(s+\alpha)} \cdot mf^s$. For the converse containment, note that the assumption $\nu_{m,\alpha}=0$ implies that $p_{m,\alpha}(s)mf^s\in V^{>\alpha}\iota_{+}\cM$. Under the localization, as $p_{m,\alpha}(-\alpha)\neq 0$, one has $\sD_X[s]_{(s+\alpha)} \cdot mf^s\subseteq (V^{>\alpha}\iota_{+}\cM)_{(s+\alpha)}$.
\end{proof}

\begin{prop}\label{prop: diagram of j*j!}
There is a commutative diagram of exact sequences of $\sD_X$-modules:
    \[ \begin{tikzcd}
    {} & {} & {} &0\arrow[d] & {}\\
    {}          & 0 \arrow[d] & 0 \arrow[d] & \Ker IC \arrow[d]& {}\\
    0 \arrow[r] & V^{>\alpha}\iota_{+}\cM \arrow[r,"s+\alpha"] \arrow[d]&  V^{>\alpha}\iota_{+}\cM \arrow[r,"\widetilde{ev}_{s=-\alpha}"]\arrow[d]&  \cM\cdot f^{-\alpha}(!D) \arrow[r]\arrow[d,"IC"]& 0\\
    0 \arrow[r] & V^{\alpha}\iota_{+}\cM \arrow[r,"s+\alpha"] \arrow[d]& V^{\alpha}\iota_{+}\cM \arrow[r,"ev_{s=-\alpha}"] \arrow[d]& \cM\cdot f^{-\alpha} \arrow[r]\arrow[d,"\pi"] & 0 \\
    {} & \gr^{\alpha}_V\iota_{+}\cM \arrow[r,"s+\alpha"] \arrow[d]& \gr^{\alpha}_V\iota_{+}\cM \arrow[r,"ev_{s=-\alpha}"]\arrow[d] & \cM\cdot f^{-\alpha}/\cS^{-\alpha} \arrow[d] \ar[r]& 0 \\
    {} & 0 & 0 & 0 &
\end{tikzcd}
\]
Moreover, the exact sequence of $\sD_X$-modules obtained from the snake lemma
\begin{equation}\label{eqn: M/S in terms of cokernel of s on grV}
0\to \cH^{-1}(i_{\ast}i^{\ast}\cS^{-\alpha})\to \gr^{\alpha}_V\iota_{+}\cM \xrightarrow{s+\alpha} \gr^{\alpha}_V\iota_{+}\cM \xrightarrow{ev_{s=-\alpha}}\cM\cdot f^{-\alpha}/\cS^{-\alpha}\to 0\end{equation}
coincides with \eqref{eqn: coker of grValpha is MquotientbyS}, where $i:D\to X$ is the closed embedding.

\end{prop}
\begin{proof}
 Put $V^{>\alpha}=V^{>\alpha}\iota_{+}\cM$, $V^{\alpha}=V^{\alpha}\iota_{+}\cM$, $\gr^{\alpha}_V=\gr_V^{\alpha}\iota_{+}\cM$ and $N=s+\alpha$. Applying the snake lemma to the short exact sequence $0\to V^{>\alpha}\to V^{\alpha}\to \gr^{\alpha}_V \to0$ with vertical maps given by $N$, and using Theorem \ref{thm: ses of Valpha and Mfalpha}, gives
\[0\to\ker(N:\gr_V^{\alpha}\to \gr_V^{\alpha})\to V^{>\alpha}/NV^{>\alpha} \to\cM\cdot f^{-\alpha}\to \gr^{\alpha}_V/N \gr^{\alpha}_V\to0.\]
Since $\cS^{-\alpha}$ is simple with multiplicity one, Lemmas \ref{lemma: commutative diagram of Vj*andj!} and \ref{lem:evsimple} show that the image of $V^{>\alpha}/NV^{>\alpha}\to\cM\cdot f^{-\alpha}$ is $\cS^{-\alpha}$. Hence the preceding sequence restricts to
\begin{equation}\label{eqn: GA}0\to\ker(N:\gr_V^{\alpha}\to \gr_V^{\alpha})\to V^{>\alpha}/NV^{>\alpha}\to\cS^{-\alpha}\to0.\end{equation}
On the other hand, the proof of Lemma \ref{lemma: grWM in terms of grV} gives the exact sequence
\begin{equation}\label{eqn: !D IC}0\to\cH^{-1}(i_{\ast}i^{\ast}\cS^{-\alpha})\to\cM\cdot f^{-\alpha}(!D)\xrightarrow{IC}\cS^{-\alpha}\to0.\end{equation}
The morphism $\widetilde{ev}_{s=-\alpha}$ from Lemma \ref{lemma: commutative diagram of Vj*andj!} induces a morphism between the two  short exact sequences \eqref{eqn: GA} and \eqref{eqn: !D IC}. By the construction in Lemma \ref{lemma: grWM in terms of grV}, the induced map on the left is the natural isomorphism $\ker(N:\gr_V^{\alpha}\to \gr_V^{\alpha})\simeq\cH^{-1}(i_{\ast}i^{\ast}\cS^{-\alpha})$, and the map on the right is the identity, so the five lemma shows that the middle map $V^{>\alpha}/NV^{>\alpha}\to\cM\cdot f^{-\alpha}(!D)$ is an isomorphism. This proves the exactness of the first row. 

The commutativity follows from Lemma \ref{lemma: commutative diagram of Vj*andj!}. Finally, applying the snake lemma to the resulting diagram gives \eqref{eqn: M/S in terms of cokernel of s on grV}, and the construction above identifies it with \eqref{eqn: coker of grValpha is MquotientbyS}. \end{proof}

\subsection{Weight and $Z$-filtrations on $V$-filtrations}
Recall that $f$ acts bijectively on the holonomic $\sD$-module $\cM$. We introduce two filtrations on $V^{\alpha}\iota_{+}\cM$ for $\alpha\in \C$.
\begin{definition}\label{def: Z filtration}
Define an increasing, exhaustive filtration on $V^\alpha \iota_+ \cM$ by
\[ Z_\ell V^\alpha \iota_+ \cM \colonequals \{\, w \in V^\alpha \iota_+ \cM \mid \mult_{s=-\alpha}b_w(s) \leq \ell \,\}, \quad \ell\in \Z_{\geq 0}, \]
where we also set $\mu_{w,\alpha}\colonequals \mult_{s=-\alpha}b_w(s)$. By convention, we set $\mu_{0, \alpha} = -\infty$.
\end{definition}

Alternatively, the $Z$-filtration can be characterized using the nilpotency index $n_{\alpha}$ of $N=s+\alpha$ on $\gr^\alpha_V \iota_+ \cM$ (see Definition \ref{definition:nilpotence degree}).

\begin{lemma}\label{lem:znilp}
    For all $\ell \in \Z_{\geq 0}$, we have $Z_\ell V^\alpha \iota_+ \cM / V^{>\alpha} \iota_{+} \cM = \ker (s+\alpha)^{\ell}$. In particular,
    \[ (s+\alpha)\cdot Z_{\ell}V^\alpha \iota_+ \cM \subseteq Z_{\ell-1}V^\alpha \iota_+ \cM. \]
\end{lemma}

\begin{proof}
    This follows immediately from Corollary \ref{cor:tozfilt}; see also \cite[Corollary 2.7]{DLY}.
\end{proof}

\begin{prop}\label{prop: Znalpha is full Z}
 For all $\ell \in \Z_{\geq 0}$, $Z_\ell V^\alpha \iota_+ \cM$ is a coherent $\sD_X\langle s, t \rangle$-module, and
 \[ V^{>\alpha} \iota_+ \cM = Z_0 V^\alpha \iota_+ \cM \subsetneq Z_1 V^\alpha \iota_+ \cM \subsetneq \dots \subsetneq Z_{n_\alpha - 1} V^\alpha \iota_+ \cM \subsetneq Z_{n_\alpha} V^\alpha \iota_+ \cM = V^\alpha \iota_+ \cM. \]
\end{prop}

\begin{proof}
 Theorem \ref{thm: Sabbah} implies that $Z_0 V^\alpha \iota_+ \cM  = V^{>\alpha} \iota_+ \cM$. Corollaries \ref{cor:tozfilt}, \ref{cor:multpfun} and Lemma \ref{lem:pfunsum} imply that $Z_\ell V^\alpha \iota_+ \cM$ is a $\sD_X\langle s,t\rangle$-submodule of $V^\alpha \iota_+ \cM$. Because $\sD_X\langle s, t \rangle$ is Noetherian (see e.g. \cite[p.~56]{Sabbah} or \cite[Lemma 2.12]{DY26}) and $V^\alpha \iota_+ \cM$ is a coherent $\sD_X\langle s, t \rangle$-module by Definition \ref{defn:CV-filtration}, the submodule $Z_\ell V^\alpha \iota_+ \cM$ is also coherent. The sequence of strict inclusions follows from Lemma \ref{lem:znilp}.
\end{proof}

\begin{prop}\label{prop: bfunction characterization of monodromy weight filtration}
We have
\begin{align*}
    W(N)_\ell \gr_V^\alpha \iota_+ \cM  
     & = \mathrm{span}_{\C[s]} \left\{\, (s+\alpha)^{\max\left\{\lceil \frac{\mu_{w, \alpha}-\ell - 1}{2} \rceil, \, 0 \right\}} \cdot [w]  \mid [w] \in \gr_V^\alpha \iota_+ \cM \,\right\}\\
    \gr^{W(N)}_\ell \gr_V^\alpha \iota_+ \cM &= \mathrm{span}_{\C} \left\{\, [(s+\alpha)^{\frac{\mu_{w,\alpha}-\ell-1}{2}}[w]]  \mid [w] \in \gr_V^\alpha \iota_+ \cM, \, \mu_{w,\alpha}>\ell, \, \mu_{w,\alpha}-\ell \textrm{ is odd} \,\right\}. 
\end{align*}
\end{prop}

\begin{proof}
By  \eqref{lemma: alternative convolution formula}, for any $w\in V^\alpha \iota_+ \cM$, we have 
\[ (s+\alpha)^j \cdot [w] \in N^j \cdot \ker N^{\ell + 2j+1} \iff \max\{\lceil(\mu_{w,\alpha}-\ell-1)/2\rceil, 0\} \leq j, \]
which yields the first equality. Since $W(N)_\ell \gr_V^\alpha \iota_+ \cM$ is stable under the action of $s$, the second equality follows immediately. 
\end{proof}

\begin{definition}
    For $\ell\in \Z$, we define a \emph{weight filtration} on $V$-filtrations by
    \begin{equation}\label{eqn: weight filtration} W_{\ell-1} V^\alpha \iota_+ \cM \colonequals \sum_{i=0}^\infty (s+\alpha)^i \cdot Z_{\ell+2i}V^{\alpha}\iota_{+}\cM. \end{equation}
    We also define an auxiliary filtration by $K_\ell  V^{\alpha} \iota_+ \cM \colonequals (s+\alpha) \cdot W_{\ell +1}  V^{\alpha} \iota_+ \cM$.
\end{definition}
By construction, we have 
    \begin{align*}
        W_{\ell-1} V^\alpha \iota_+ \cM &= \sum_{i=0}^{\lceil (n_\alpha - \ell )/2 \rceil}(s+\alpha)^i \cdot Z_{\ell+2i}V^{\alpha}\iota_{+}\cM,\\
        W_{\ell - 1}  V^{\alpha} \iota_+ \cM &= Z_\ell  V^{\alpha} \iota_+ \cM + K_\ell  V^{\alpha} \iota_+ \cM.
    \end{align*}

\begin{lemma}\label{lem:key} 
    For any $\ell \in \Z_{\geq 0}$, we have
    \[ Z_\ell V^\alpha \iota_+ \cM \cap (s+\alpha)V^\alpha\iota_+ \cM = Z_\ell  V^{\alpha} \iota_+ \cM \cap K_\ell  V^{\alpha} \iota_+ \cM = (s+\alpha) \cdot Z_{\ell+1}  V^{\alpha} \iota_+ \cM. \]
\end{lemma}

\begin{proof}
    Because $(s+\alpha)\cdot Z_{\ell+1}  V^{\alpha} \iota_+ \cM \subseteq Z_\ell V^\alpha \iota_+ \cM$, we immediately obtain the inclusions
    \[ (s+\alpha) \cdot Z_{\ell+1}  V^{\alpha} \iota_+ \cM \subseteq Z_\ell  V^{\alpha} \iota_+ \cM \cap K_\ell  V^{\alpha} \iota_+ \cM \subseteq Z_\ell V^\alpha \iota_+ \cM \cap (s+\alpha)V^\alpha\iota_+ \cM. \]
    Conversely, let $w \in  Z_\ell  V^{\alpha} \iota_+ \cM \cap (s+\alpha)V^{\alpha} \iota_+ \cM$. Then $w=(s+\alpha) w'$ for some $w' \in V^\alpha \iota_+ \cM$. By Lemma \ref{lem:easy2}, we must have $w' \in Z_{\ell+1} V^{\alpha} \iota_+ \cM$, verifying the reverse inclusion.
\end{proof}

From now on, assume $\cM=\cS_f\neq 0$ with $\cS$ a simple holonomic $\sD$-module, where $\cS$ automatically underlies a pure twistor $\sD$-module on $X$, say of weight $q$. 
\begin{prop}\label{prop:strict}
There are short exact sequences: 
\begin{align*}
    0 \to V^{>\alpha}\iota_+ \cM \to &W_\ell V^\alpha \iota_+ \cM \to W(N)_\ell \gr_V^\alpha \iota_+ \cM \to 0, \quad \forall \ell\geq -1,\\
 0 \to W_{\ell+1} V^\alpha \iota_+ \cM  \xrightarrow{s+\alpha}  &W_{\ell-1} V^\alpha \iota_+ \cM  \xrightarrow{ev_{s=-\alpha}} W_{q + \ell}(\cM \cdot f^{-\alpha}) \to 0, \quad \forall \ell\geq 0,
 \end{align*}
where the second one is induced by Theorem \ref{thm: ses of Valpha and Mfalpha}.
\end{prop}

\begin{proof}
    The first exact sequence follows directly from Lemma \ref{lem:znilp} and Proposition \ref{prop: bfunction characterization of monodromy weight filtration}. 
    
    For the second, fix $\ell \in \Z_{\geq 0}$ and set
    \[ ev \colonequals ev_{s=-\alpha}|_{W_{\ell-1}V^\alpha \iota_+ \cM}.\]
    Let us describe $\ker(ev)$. By Theorem \ref{thm: ses of Valpha and Mfalpha}, we have $\ker(ev) = W_{\ell-1} V^\alpha \iota_+ \cM \cap (s+\alpha) \cdot V^\alpha \iota_+ \cM$, which clearly contains $(s+\alpha)\cdot W_{\ell+1}V^\alpha \iota_+ \cM$. For the reverse inclusion, take 
    \[w \in W_{\ell-1} V^\alpha \iota_+ \cM \cap (s+\alpha) V^\alpha \iota_+ \cM.\] Since $W_{\ell-1}V^\alpha \iota_+\cM = Z_\ell V^\alpha \iota_+ \cM + K_\ell V^\alpha \iota_+ \cM$, we can decompose $w=z+k$ where $z \in Z_\ell V^\alpha \iota_+ \cM$ and $k\in K_\ell V^\alpha \iota_+ \cM$. Because $w$ and $k$ both belong to $(s+\alpha) V^\alpha \iota_+ \cM$, their difference $z$ must as well. Lemma \ref{lem:key} then implies $z\in K_\ell V^\alpha \iota_+ \cM$, meaning $w \in K_\ell V^\alpha \iota_+ \cM$. Thus, $\ker(ev) = K_\ell V^\alpha \iota_+ \cM=(s+\alpha)\cdot W_{\ell+1}V^{\alpha}\iota_{+}\cM$. 
    
    Combining this with the first short exact sequence and Lemma \ref{lemma: grWM in terms of grV}, we obtain the following commutative diagram with exact rows and columns for $\ell \geq 0$:
    \[ \begin{tikzcd}[column sep=scriptsize]
    {}          & 0 \arrow[d] & 0 \arrow[d] & & {}\\
    0 \arrow[r] & V^{>\alpha}\iota_{+}\cM \arrow[r,"s+\alpha"] \arrow[d]&  V^{>\alpha}\iota_{+}\cM \arrow[r]\arrow[d]&  V^{>\alpha}\iota_{+}\cM/(s+\alpha)V^{>\alpha}\iota_{+}\cM \arrow[r]\arrow[d,"\pi_\ell"]& 0\\
    0 \arrow[r] & W_{\ell+1} V^{\alpha}\iota_{+}\cM \arrow[r,"s+\alpha"] \arrow[d]& W_{\ell-1} V^{\alpha}\iota_{+}\cM \arrow[r,"ev"] \arrow[d]& Q_\ell \arrow[r]\arrow[d] & 0 \\
    {} & W(N)_{\ell+1} \gr^{\alpha}_V\iota_{+}\cM \arrow[r,"s+\alpha"] \arrow[d]& W(N)_{\ell-1} \gr^{\alpha}_V\iota_{+}\cM \arrow[r, "\overline{ev}"]\arrow[d] & W_{q+\ell} (\cM\cdot f^{-\alpha})/\cS^{-\alpha} \arrow[d] \arrow[r]& 0 \\
    {} & 0 & 0 & 0 &
\end{tikzcd}\]
Here, we set $Q_\ell \colonequals \mathrm{Im}(ev|_{W_{\ell-1}V^{\alpha}}) \subset \cM\cdot f^{-\alpha}$. By Lemma \ref{lem:evsimple}, we see that 
\[ \mathrm{Im}(\pi_\ell) = \mathrm{Im}(ev|_{V^{>\alpha}\iota_+ \cM}) = \cS^{-\alpha} \subseteq Q_\ell. \]
It follows that $Q_\ell = W_{q+ \ell}(\cM \cdot f^{-\alpha})$, proving the second exact sequence.
\end{proof}

In fact, $ev_{s=-\alpha}$ yields another similar exact sequence via $Z$-filtrations.

\begin{thm}\label{thm:Zfiltexact}
For any $\ell \in \Z_{\geq 0}$, we have short exact sequences
\[ 0 \to Z_{\ell+1}  V^{\alpha} \iota_+ \cM \xrightarrow{s+\alpha} Z_\ell V^\alpha \iota_+ \cM \xrightarrow{ev_{s=-\alpha}} W_{q+ \ell}(\cM \cdot f^{-\alpha}) \to 0, \]
\[ 0 \to \gr_{\ell+2}^Z  V^{\alpha} \iota_+ \cM \xrightarrow{s+\alpha} \gr_{\ell+1}^Z V^\alpha \iota_+ \cM \xrightarrow{ev_{s=-\alpha}} \gr^W_{q+\ell+1}(\cM \cdot f^{-\alpha}) \to 0. \]

\end{thm}

\begin{proof}
    By Proposition \ref{prop:strict} and \eqref{eqn: weight filtration}, we have the equality 
    \[ W_{q+\ell}(\cM \cdot f^{-\alpha}) = \mathrm{Im}\left(ev_{s=-\alpha}|_{W_{\ell-1}V^\alpha \iota_+ \cM}\right)=\mathrm{Im}\left( ev_{s=-\alpha}|_{Z_\ell V^\alpha \iota_+ \cM}\right). \]
    Furthermore, Theorem \ref{thm: ses of Valpha and Mfalpha} and Lemma \ref{lem:key} jointly imply that
    \[ \ker(ev_{s=-\alpha}|_{Z_\ell V^\alpha \iota_+ \cM}) = Z_\ell V^\alpha \iota_+ \cM \cap (s+\alpha)V^\alpha \iota_+ \cM = (s+\alpha) Z_{\ell+1}V^\alpha \iota_+ \cM. \qedhere\]
\end{proof}

\begin{remark}
    Note that $\nu_{m,\alpha}\leq \ell$ if and only if $mf^s \in (Z_{\ell}V^{\alpha}\iota_{+}\cM)_{(s+\alpha)}$, which in turn holds if and only if $[m f^s]\in (W_{\ell-1}\gr^{\alpha}_V\iota_{+}\cM)_{(s+\alpha)}$.
\end{remark}

\section{The weight filtration on the localization}

Let $X$ be a smooth complex algebraic variety and $f$ an algebraic function on $X$. Let $\cS$ be a holonomic $\sD_X$-module where $f$ acts injectively (so that $\cS\hookrightarrow \cS_f$). Let $\alpha \in \C$.  We use the isomorphism $\iota_{+}\cM\cong \cM[s]f^s$ in \eqref{eqn: Malgrange isomorphism} throughout.

\subsection{A mod $b$-function description of the weight filtration}\label{sec:weight}
In this section, we prove Theorem \ref{thm: remark5.6 intro}. We first introduce a filtration on $\cM \cdot f^{-\alpha}$ for any $\cM$ holonomic with $\cM=\cM_f$.

\begin{definition}\label{def: b filtration}
    For $\ell \geq 0$, let $\mathcal{B}_\ell(\cM \cdot f^{-\alpha}) = ev_{s=-\alpha}(Z_\ell V^\alpha \iota_+ \cM)$, where $ev_{s=-\alpha}:V^{\alpha}\iota_{+}\cM\to \cM \cdot f^{-\alpha}$ is from Theorem \ref{thm: ses of Valpha and Mfalpha} and the $Z$-filtration is from Definition \ref{def: Z filtration}. 
\end{definition}

Clearly, $\mathcal{B}_\bullet (\cM \cdot f^{-\alpha})$ is an increasing, exhaustive filtration of $\sD_X$-modules on $\cM \cdot f^{-\alpha}$. We have the following preliminary characterization of $\cB_{\bullet}$ based on $b$-functions. 

\begin{lemma}\label{lem:Bfirst}
    We have $\mathcal{B}_\ell (\cM \cdot f^{-\alpha}) = \{m f^{-\alpha} \mid \exists w \in \iota_+\cM, \, m f^s+(s+\alpha)w \in Z_\ell V^\alpha \iota_+ \cM. \}$.
\end{lemma}

\begin{proof}
    Clearly, the right-hand side is contained in the left. Conversely, let $m f^{-\alpha} \in \mathcal{B}_\ell(\cM\cdot f^{-\alpha})$, and $w_0 \in Z_\ell V^\alpha \iota_+ \cM$ such that $ev_{s=-\alpha}(w_0)=m f^{-\alpha}$. Since $ev_{s=-\alpha}(mf^s)=mf^{-\alpha}$, by (\ref{eq:basic}) we get $w \in \iota_+ \cM$ such that $w_0 - mf^{s}= (s+\alpha)w.$
\end{proof}

\begin{definition}\label{definition: ellk}
Let $m\in \cM$, for each $k\in \Z_{>0}$, let $\ell_{m,k}$ denote the smallest integer $\ell \geq 0$ for which there exists some $P\in \sD_X[s]$ such that
\[P \cdot m f^{s+k} = (s+\alpha)^\ell \cdot m f^{s} \pmod{(s+\alpha)^{\ell+1}\cdot \cM[s]f^s}.\]
\end{definition}

In contrast with the introduction, here we first introduce $\omega$ as follows:
\begin{prop}\label{prop: omega malpha exists}
    The limit $\omega_{m,\alpha}\colonequals \lim_{k \to \infty} \ell_{m, k}$ exists, and $\omega_{m, \alpha} \leq \nu_{m,\alpha}$.
\end{prop}
\begin{proof}
Fix $k\in \Z_{> 0}$ and consider the defining equation for $b_m^{(k)}(s)$ from Definition \ref{definition: higher order bfunction}:
\[ P\cdot mf^{s+k}=b_m^{(k)}(s)mf^s,\]
for some $P\in\sD_X[s]$. Write $b_{m}^{(k)}(s)=(s+\alpha)^{\rho_k} c(s)$, where $\rho_k\colonequals \mult_{s=-\alpha}b_m^{(k)}(s)$, ensuring $c(-\alpha)\neq 0$. Expanding $c(s) = c_0 + (s+\alpha)p(s)$ with $c_0 \neq 0$ and dividing by $c_0$ yields:
\begin{equation}
c_0^{-1} P \cdot mf^{s+k} = (s+\alpha)^{\rho_k}mf^s + (s+\alpha)^{\rho_k+1}c_0^{-1}p(s)mf^s.
\end{equation}
This demonstrates that $\ell_{m, k}$ exists and $\ell_{m,k} \leq \rho_k$. Since the sequence $\{\ell_{m,k}\}_{k > 0}$ is clearly nondecreasing and $\lim_{k\to \infty} \rho_k = \nu_{m,\alpha}$ by Theorem \ref{thm: numalpha via higher order pfunction}, $\omega_{m,\alpha}$ exists and is $\leq \nu_{m,\alpha}$.
\end{proof}

\begin{thm}\label{thm:Bnice}
Assume $\cS$ is simple and $\cM=\cS_f$, then we have
    \[W_{q+\ell} (\cM \cdot f^{-\alpha} )= \mathcal{B}_\ell (\cM \cdot f^{-\alpha}) = \{m f^{-\alpha} \mid \omega_{m, \alpha } \leq \ell \}. \]
\end{thm}

\begin{proof}
Assume first that $m\cdot f^{-\alpha} \in \mathcal{B}_\ell(\cM \cdot f^{-\alpha})$ for some $\ell\geq 0$. By Lemma \ref{lem:Bfirst}, there is a $w \in \iota_+ \cM$ such that
\[ mf^{s}+(s+\alpha)w_0 \in Z_{\ell}V^{\alpha}\iota_{+}\cM.\]
Multiplying by $(s+\alpha)^\ell$ and applying Lemma \ref{lem:znilp}, we obtain
\begin{equation*}
w \colonequals (s+\alpha)^{\ell}w' = (s+\alpha)^\ell mf^s + (s+\alpha)^{\ell+1} w_0 \in Z_0V^{\alpha}\iota_{+}\cM = V^{>\alpha}\iota_+ \cM.
\end{equation*}
Fix $k\in \Z_{> 0}$. Applying Lemma \ref{lemma: relating two elements in iota+M} to $mf^{s+k}$ and $w$, there exists $P_1 \in \sD_X[s]$ satisfying
\begin{equation}\label{eqn: Qmfs+k is css+alphaell}
P_1 \cdot mf^{s+k} = b_{w}^{(k)}(s) \cdot w = b_{w}^{(k)}(s)(s+\alpha)^\ell mf^s + (s+\alpha)^{\ell+1} w_1,
\end{equation}
for some $w_1\in \iota_{+}\cM$. Because $w\in V^{>\alpha}\iota_{+}\cM$, Theorem \ref{thm: Sabbah} implies $b_w(s)$ has no roots $>-\alpha$. So \eqref{eqn: division of higher b function of w} forces $b_{w}^{(k)}(-\alpha)\neq 0$. Writing $b_{w}^{(k)}(s) = c_0 +(s+\alpha)c(s)$ and dividing \eqref{eqn: Qmfs+k is css+alphaell} by $c_0 \neq 0$ yields
\[P_2 \cdot mf^{s+k} = (s+\alpha)^\ell mf^s + (s+\alpha)^{\ell+1} w_2,\]
for some $P_2\in \sD_X[s]$ and $w_2 \in \iota_+ \cM$. Definition \ref{definition: ellk} then implies $\ell_k \leq \ell$. Taking $k\to \infty$, Proposition \ref{prop: omega malpha exists} gives $\omega_{m,\alpha} \leq \ell$.

Conversely, assume $\omega_{m,\alpha} \leq \ell$, and note that by Proposition \ref{prop: omega malpha exists}, for $k\gg 0$, we have
\begin{equation}\label{eqn: P3}
P_3 \cdot mf^{s+k} = (s+\alpha)^{\omega_{m,\alpha}} mf^s + (s+\alpha)^{\omega_{m,\alpha}+1} w_3,
\end{equation}
for some $P_3 \in \sD_X[s]$ and $w_3 \in \iota_+ \cM$. Choosing $k$ large enough ensures $P_3 \cdot mf^{s+k} \in V^{>\alpha}\iota_+ \cM$. By Lemma \ref{lem:easy2}, the $b$-function of $mf^s + (s+\alpha) w_3$ have only roots $\leq -\alpha$, which then must lie in $V^{\alpha}\iota_{+}\cM$ by Theorem \ref{thm: Sabbah}. Using Lemma \ref{lem:znilp}, \eqref{eqn: P3} further implies
\[mf^s + (s+\alpha) w_3 \in Z_{\omega_{m,\alpha}}V^\alpha \iota_{+}\cM.\]
By Lemma \ref{lem:Bfirst}, we obtain $m f^{-\alpha} \in \mathcal{B}_{\omega_{m,\alpha}}(\cM f^{-\alpha}) \subseteq \mathcal{B}_{\ell}(\cM f^{-\alpha})$.

Theorem \ref{thm:Zfiltexact} gives $W_{q+\ell} (\cM \cdot f^{-\alpha} ) =  \mathcal{B}_\ell(\cM \cdot f^{-\alpha})$.
\end{proof}

\begin{remark}\label{remark: suffices to choose m'}
    The proof demonstrates a stronger fact: if we choose $0\neq \tilde{m} \in \cM$ such that $b_{\tilde{m}}(s)$ has no roots in $-\alpha+\Z_{\geq 0}$, then $\omega_{m, \alpha}$ is the minimal integer $\ell \geq 0$ satisfying
    \[P\cdot \tilde{m}f^{s} = (s+\alpha)^\ell mf^s \pmod{(s+\alpha)^{\ell+1}\cdot \cM[s]f^s}\]
    for some $P\in \sD_X[s]$.
\end{remark}

\begin{proof}[Proof of Theorem \ref{thm: remark5.6 intro}]
    It follows immediately from Theorem \ref{thm:Bnice} and Remark \ref{remark: suffices to choose m'}.
\end{proof}

\begin{remark}
    Corollary \ref{cor:weightb} is inspired by \cite[Theorem 3.1]{DV22}, which states that if $\cS$ underlies a complex Hodge module, then the weight filtration on $\cM \cdot f^{-\alpha}$ coincides with the Jantzen filtration up to a shift. This shift property was originally established by Beilinson--Bernstein \cite{BB} for mixed sheaves over finite base fields in their work on Jantzen's conjecture. See also \cite[\S 4.6]{Gyoja1996}.
\end{remark}

Let us mention what happens in the analytic setting, when $X$ is a complex analytic manifold, $f$ a holomorphic function, and $\cS$ a simple holonomic analytic $\sD$-module. The equality $\mathcal{B}_0(\cM \cdot f^{-\alpha}) = \cS^{-\alpha}$ still holds, as well as
\[\mathcal{B}_\ell (\cM \cdot f^{-\alpha}) = \{m f^{-\alpha} \mid \omega_{m, \alpha } \leq \ell \}.\]
In fact, whenever $\cS$ underlies the structure of a (purely imaginary) twistor $\sD$-module (or a complex pure Hodge module \cite{MHMproject}) structure, the results of this section extend in the analytic setting, and $W_{q+\ell}(\cM \cdot f^{-\alpha} ) = \mathcal{B}_\ell(\cM \cdot f^{-\alpha})$. For example, this applies to $\cS = \cO_X$. However, due to the known failure of the decomposition theorem in this setup in general, we pose the following as a challenge.
\begin{problem}
Suppose $\cS$ is a simple holonomic analytic $\sD$-module, find an example where $\gr_{\ell}^{\mathcal{B}}(\cM\cdot f^{-\alpha})$ is not a semisimple analytic $\sD$-module, with $\ell > 0$.
\end{problem}

\subsection{An algorithmic version}\label{sec: algorithm}
Assuming $\cM=(\cO_X)_f$, we provide an algorithm to compute the weight level, which has been implemented in Macaulay2 \cite{M2}, see \cite{jsag}.

\begin{algorithm}\label{algorithm}[Test the weight level]
\quad\\
Input: A polynomial $g\in \cO_X$, numbers $\alpha \in \Q$ and $\omega \in \Z_{\geq 0}$.\\
Output: Determine whether $g \cdot f^{-\alpha} \in W_{\dim X + \omega} (\cM \cdot f^{-\alpha})$.

\begin{enumerate}
    \item Compute $b_1(s)$ (the standard Bernstein--Sato polynomial of $f$). Then set:
     \begin{itemize}
    \item[(a)] $k_1$ to be the largest $\ell\in \Z_{\geq 0}$ such that $b_1(-\alpha-\ell)=0$; if no such integer exists, set $k_1=0$.
    \item[(b)] $k_2$ to be the smallest $\ell\in \Z_{\geq 0}$ such that $b_1(-\alpha+\ell+j)\neq 0$ for any $j\geq 0$.
    \end{itemize}

    \item Compute generators of $I_f:= \mathrm{Ann}_{\sD_X[s]} (f^{s})$.
    \item Then $g\cdot f^{-\alpha} \in W_{\dim X + \omega} (\cM\cdot f^{-\alpha})$ if and only if
\[(s+\alpha+k_1)^\omega \cdot g f^{k_1} \, \in \, I_f + \sD_X[s] \cdot ((s+\alpha+k_1)^{\omega+1}, \, f^{k_1+k_2}).\]
\end{enumerate}

\end{algorithm}

\begin{remark}
    Since $W_{\bullet}(\cM\cdot f^{-\alpha})\cong W_{\bullet}(\cM\cdot f^{-(\alpha-1)})$, it follows that $\omega_{m,\alpha}=\omega_{mf^{\ell},\alpha+\ell}$ for any $\ell$. We may therefore always assume $m\in \cO_X$.
\end{remark}

 \begin{remark}
 Computing $b_1(s)$ is computationally expensive. To reduce computation time, one can use \emph{a priori} estimates for $k_1$ and $k_2$ without computing $b_1(s)$. We may assume $0<\alpha\leq 1$, then
 \begin{enumerate}
     \item One can replace $k_1$ with $\lceil n-1-\alpha\rceil$. By \cite[Theorem 0.4]{Saito94}, if $\alpha+k_1+1\geq n$, then $b_1(-\alpha-k_1-i)\neq 0$ for all $i\in \Z_{\geq 1}$, which suffices for the algorithm. For special $f$, tighter bounds exist: if $f$ defines a central hyperplane arrangement, all roots of $b_1(s)$ are within $(-2,0)$ \cite[Theorem 1]{Saitohyperplanearrangement}, so we can set $k_1=\lceil 1-\alpha\rceil$.
     \item One can replace $k_2$ with $\lceil \alpha\rceil$. This is valid because $b_1(s+\lceil \alpha\rceil)$ has no roots in $-\alpha+\Z_{\geq 0}$ \cite{Kas76}.
 \end{enumerate}
 \end{remark}

Step (3) relies on the following lemma:
\begin{lemma}\label{lem:Annfs}
For $g\in \cO_X$, the weight level $\omega_{g,\alpha}$ is the smallest integer $\ell \geq 0$ such that
    \[ P(s) f^{s+k_2} = (s+\alpha)^\ell g f^s \pmod{ (s+\alpha)^{\ell+1} \cdot \sD_X[s]f^{s-k_1} } \]
    for some $P(s)\in \sD_X[s]$. In particular, the entire relation takes place in $\sD_X[s]f^{s-k_1}$.
\end{lemma}

\begin{proof}
Set $m=f^{k_2}$. Because $b_{m}(s)=b_1(s+k_2)$ has no roots in $-\alpha+\Z_{\geq 0}$, Remark \ref{remark: suffices to choose m'} implies that $\omega_{g,\alpha}$ is the smallest $\ell\in \Z_{\geq 0}$ satisfying an equation of the form
\begin{equation}\label{eq:startmod}
P(s) \cdot f^{s+k_2} = (s+\alpha)^\ell g f^{s} + (s+\alpha)^{\ell+1} w, \quad \text{for some } P(s) \in \sD_X[s] \text{ and } w \in \iota_+\cM.
\end{equation}
Choose $k\gg 0$ such that $w \in \cO_X[s]\cdot f^{s-k}$. By our choice of $k_1$, we can repeatedly apply the defining equation of the Bernstein--Sato polynomial of $f$ to find $Q(s)\in \sD_X[s]$ and $b(s) \in \C[s]$ with $b(-\alpha)=1$ such that
\[Q(s) \cdot f^{s-k_1} = b(s) \cdot f^{s-k}.\]
This ensures $b(s) w \in \sD_X[s] f^{s-k_1}$. Let $b(s) = 1+(s+\alpha) c(s)$. Multiplying \eqref{eq:startmod} by $b(s)$ yields
\[R(s) \cdot f^{s+k_2} = (s+\alpha)^\ell g f^s + (s+\alpha)^{\ell+1} (b(s)w+c(s) g f^s) \]
for some $R(s) \in \sD_X[s]$. Since $g \in \cO_X$, the claim follows.
\end{proof}
\begin{remark}\label{remark: produce BW examples}
    Our algorithm \ref{algorithm} provides an easy way to produce examples $f$ such that $\sD_X\cdot f^{-\alpha}=\sD_X\cdot f^{-\alpha+1}$ when $b_f(-\alpha)=0$, which is a question of Budur and Walther discussed in the introduction \cite{saitopower}. More precisely, if $\alpha\in (0,1)$, then $f^{-\alpha}\in \sD_X\cdot f^{-\alpha+1}$ if and only if $f^{-\alpha}\in W_{\dim X}\left((\cO_X)_f\cdot f^{-\alpha}\right)$. By Theorem \ref{thm:Bnice}, the latter is equivalent to $\omega_{1,\alpha}=0$, which can be computed by this algorithm. This is already implemented in \cite[Remark 6.9]{DY26} to produce an example with small degree, compared to \cite{saitopower}.
\end{remark}
\subsection{An approximation of the weight filtration}
We study a useful approximation of $W_{\ell}(\cM\cdot f^{-\alpha})$ induced by $\nu_{m,\alpha}$. 
\begin{definition}\label{def: N filtration}
    For $\ell \in \Z_{\geq 0}$, set $\cN_{\ell}(\cM\cdot f^{-\alpha})\colonequals \{ m f^{-\alpha} \mid \nu_{m,\alpha} \leq \ell \}\subseteq \cM\cdot f^{-\alpha}$. To simplify notation, we sometimes write $\cN_{\ell}$ for $\cN_{\ell}(\cM\cdot f^{-\alpha})$.
\end{definition}

\begin{prop}\label{prop:restrnu}
    Let $g\in \cO_X$ and $m \in \cM$. Setting $U_g \colonequals \{g\neq 0\}$, we have
    \[\nu_{m_{\mid U_g}, \alpha} = \lim_{k\to \infty} \nu_{g^k \cdot m, \alpha}.\]
    Moreover, $\cN_{\ell}(\cM \cdot f^{-\alpha}) \subseteq W_{q + \ell} (\cM \cdot f^{-\alpha})$ is a quasi-coherent subsheaf of $\cO_X$-modules.
\end{prop}
\begin{remark}
    The proof that $\cN_{\ell}(\cM \cdot f^{-\alpha})$ is a sheaf of $\cO_X$-modules works without assuming $\cS$ is simple.
\end{remark}
\begin{proof}
First, we show that on any affine open set, $\cN_\ell$ is closed under addition and $\cO_X$-multiplication. Let $m_1f^{-\alpha},m_2f^{-\alpha}\in \cN_{\ell}$, and denote $p(s) \colonequals \mathrm{lcm}(p_{m_1, \alpha}(s), p_{m_2, \alpha}(s))$. By Lemmas \ref{lem:pfunsum}(i) and \ref{lemma: basic property of pfunction}(i), $p_{m_1+m_2}(s)\mid p(s)$ and $(s+\alpha) \nmid p(s)$. Applying Lemma \ref{lemma: basic property of pfunction}(iii) to $m_1+m_2$ and $m_i$ yields
\[\mu_{p(s) \cdot (m_1+m_2),\alpha}=\nu_{m_1+m_2, \alpha}, \quad \mu_{p(s) \cdot m_i, \alpha} = \nu_{m_i, \alpha} \leq \ell.\]
Since $Z_\ell V^\alpha \iota_+ \cM$ is closed under addition (see Proposition \ref{prop: Znalpha is full Z}), we have
\[\mu_{p(s)\cdot(m_1+m_2),\alpha} \leq \max\{\mu_{p(s) \cdot m_1, \alpha}, \mu_{p(s) \cdot m_2, \alpha}\} \leq \ell.\]
Consequently, $(m_1+m_2)f^{-\alpha}\in \cN_{\ell}$. Clearly, $\cN_\ell$ is also stable under scalar multiplication, making it a $\C$-vector subspace of $\cM \cdot f^{-\alpha}$.

Next, let $mf^{-\alpha}\in \cN_\ell$ and $g\in \cO_X$. Lemmas \ref{lem:pfunsum}(ii) and \ref{lemma: basic property of pfunction}(i) show $p_{gm,\alpha}(s)\mid p_{m,\alpha}(s)$ and $(s+\alpha)\nmid p_{m,\alpha}(s)$. Applying Lemma \ref{lemma: basic property of pfunction}(iii) to $g \cdot m$, we find $\nu_{gm,\alpha}=\mu_{p_{m,\alpha}(s) \cdot gm,\alpha}$. Since $p_{m,\alpha}(s)mf^s\in Z_\ell V^{\alpha}$, by Proposition \ref{prop: Znalpha is full Z} again we have
\[\mu_{p_{m,\alpha}(s) \cdot gm,\alpha}\leq \mu_{p_{m,\alpha}(s)\cdot m,\alpha}=\nu_{m,\alpha} \leq \ell.\]
This confirms $\cN_\ell$ is an $\cO_X$-module.

Now we prove the limit equality. By the proof of Lemma \ref{lemma: basic property of pfunction}(v), the sequence $\{\nu_{g^k \cdot m, \alpha} \}_{k \geq 0}$ is non-increasing. By Theorem \ref{thm: numalpha via higher order pfunction}, $\nu_{m_{\mid U_g}, \alpha}$ is the smallest integer $\nu \geq 0$ such that
 \begin{equation}\label{eq:intersmallest}
 (s+\alpha)^{\nu} mf^s \in \bigcap_{i \in \Z} \sD_{U_g}[s]_{(s+\alpha)} \cdot mf^{s+i},
 \end{equation}
and $\nu_{g^k \cdot m, \alpha}$ is the smallest integer $\nu_k \geq 0$ such that
\[(s+\alpha)^{\nu_k} mf^s \in \bigcap_{i \in \Z} g^{-k}\cdot \sD_{X}[s]_{(s+\alpha)} \cdot g^k \cdot m f^{s+i}.\]
Because $g^{-k}\cdot \sD_{X}[s]_{(s+\alpha)} \cdot g^k \subseteq \sD_{U_g}[s]_{(s+\alpha)}$, it follows that $\nu_{g^k \cdot m, \alpha} \geq \nu_{m_{\mid U_g}, \alpha}$ for all $k \geq 0$.

Conversely, pick $i\in \Z$ large enough that $m f^{s+i} \in V^{>\alpha}\iota_+\cM$. By \eqref{eq:intersmallest}, we have an equation
\[g^{-k_0} \cdot P(s) \cdot mf^{s+i} = (s+\alpha)^{\nu_{m_{\mid U_g}, \alpha}} m f^s\]
for some $k_0 \geq 0$ and $P(s) \in \sD_X[s]_{(s+\alpha)}$. This implies
\[(s+\alpha)^{\nu_{m_{\mid U_g}, \alpha}} (g^{k_0} m) f^s \in (V^{>\alpha}\iota_+\cM)_{(s+\alpha)}.\]
Multiplying by a polynomial $q(s)$ with $q(-\alpha)\neq 0$, we can apply Lemma \ref{lemma: basic property of pfunction}(iii) to show $\nu_{g^{k_0} \cdot m, \alpha} \leq \nu_{m_{\mid U_g}, \alpha}$. Thus, $\cN_\ell((\cM\cdot f^{-\alpha})_{\mid U_g}) = \cN_\ell (\cM \cdot f^{-\alpha})_{\mid U_g}$. This confirms $\cN_\ell$ forms a quasi-coherent sheaf.

Finally, the inclusion $\cN_{\ell}(\cM\cdot f^{-\alpha})\subseteq W_{q + \ell} (\cM \cdot f^{-\alpha})$ follows from Proposition \ref{prop: omega malpha exists}. 
\end{proof}

The remainder of this section compares $\cN_{\ell}$ and $W_{\ell}$, in view of Question \ref{que: weight filtration on Mf-alpha in terms of nu}.

\begin{lemma}\label{lem:basecase}
We have $\sD_X \cdot \cN_0(\cM\cdot f^{-\alpha}) = W_{q}(\cM\cdot f^{-\alpha})$.
\end{lemma}

\begin{proof}
First, $\cN_0$ is non-empty: for any $m \in \cM$, choosing $k \geq 0$ such that $m f^{s+k} \in V^{>\alpha} \iota_+ \cM$ yields $(mf^{k})\cdot f^{-\alpha}\in \cN_0$. Meanwhile, Lemma \ref{lemma: grWM in terms of grV} implies $W_{q}(\cM\cdot f^{-\alpha})=\cS^{-\alpha}$ is a simple $\sD_X$-module. Thus, the equality must hold.
\end{proof}

\begin{remark}\label{remark: N0=W0 for OX}
    If $\cM=(\cO_X)_f$, then $\cN_0(\cM)=W_{\dim X}(\cM)$. Because $W_{\dim X}(\cM)=\cO_X$ and $\cN_0(\cM)$ is an $\cO_X$-module, it suffices to show $1\in \cN_0(\cM)$. Kashiwara \cite{Kas76} showed that $b_1(s)$ has only negative roots, so Theorem \ref{thm: numalpha via higher order pfunction} ensures $\nu_{1,0}=0$, meaning $1\in \cN_0(\cM)$.
\end{remark}

 Recall the nilpotency order $n_{\alpha}$ from Definition \ref{definition:nilpotence degree}.

\begin{prop}\label{prop:lastweight}
Assume $\gr^{\alpha}_V\iota_{+}\cM\neq 0$, we have
    \[ W_{q + n_\alpha-1}(\cM\cdot f^{-\alpha}) = \cN_{n_\alpha-1} \, \subsetneq \, \cN_{n_\alpha} = W_{q + n_\alpha}(\cM\cdot f^{-\alpha}) = \cM \cdot f^{-\alpha}.\]
\end{prop}

\begin{proof}
    Lemma \ref{lem:basecase} and Corollary \ref{cor:nilpdegmax} guarantee $\cN_{n_\alpha-1} \subsetneq \cN_{n_\alpha} = \cM \cdot f^{-\alpha}$. By Proposition \ref{prop:restrnu}, it remains only to show $W_{q + n_\alpha-1}(\cM\cdot f^{-\alpha}) \subseteq \cN_{n_\alpha-1}$.

    Take any $m f^{-\alpha} \in W_{q + n_\alpha-1}(\cM\cdot f^{-\alpha})$. By Theorem \ref{thm:Zfiltexact}, there exist $w \in Z_{n_\alpha -1} V^\alpha \iota_+ \cM$ and $w' \in V^{\alpha} \iota_+ \cM = Z_{n_\alpha}V^{\alpha} \iota_+ \cM$ (Proposition \ref{prop: Znalpha is full Z}) such that
    \[w= \frac{p_{m,\alpha}(s)}{p_{m,\alpha}(-\alpha)}mf^s + (s+\alpha) w'.\]
    Because $(s+\alpha) w' \in  Z_{n_\alpha -1} V^\alpha \iota_+ \cM$ by Lemma \ref{lem:znilp}, we must have $p_{m,\alpha}(s) m f^s \in Z_{n_\alpha -1} V^\alpha \iota_+ \cM$. This means $\nu_{m,\alpha} \leq n_{\alpha}-1$, so $m f^{-\alpha} \in \cN_{n_\alpha-1}$.
\end{proof}

\begin{corollary}\label{cor:gennu}
    Let $m \in \cM$ such that $\sD_X \cdot \left(mf^{-\alpha}\right)= \cM \cdot f^{-\alpha}$. Then $n_{\alpha}=\nu_{m,\alpha}=\omega_{m,\alpha}$.
\end{corollary}

\begin{proof}
By Proposition \ref{prop:lastweight}, $m f^{-\alpha} \in \cN_{n_\alpha}$ but $m f^{-\alpha} \notin \cN_{n_\alpha-1}$, so $\nu_{m,\alpha}=n_{\alpha}$. The same argument yields $\omega_{m,\alpha}=n_{\alpha}$.
\end{proof}

The next result illustrates why generation by $\sD_X$ is necessary in Question \ref{que: weight filtration on Mf-alpha in terms of nu}. By \cite[Remark 6.9]{DY26} or \cite[Example 4.2]{Saito16}, the function $f$ satisfying the assumption of Lemma \ref{lem:counterex} exists.

\begin{lemma}\label{lem:counterex}
    Let $\cM=(\cO_X)_f$. If $b_1(-\alpha)=0$ for $0<\alpha<1$ and $f^{-\alpha} \in \sD_X\cdot f^{-\alpha+1}$, then $\cN_0(\cM \cdot f^{-\alpha})$ is not a $\sD_X$-module.
\end{lemma}

\begin{proof}
Because $0<\alpha<1$ and $b_1(s)$ has only negative roots, we have 
\[ \nu_{f,\alpha}=\nu_{1,\alpha-1}=0, \quad \nu_{1,\alpha}=\mult_{s=-\alpha}b_1(s)\geq 1.\]
Thus, $f^{-\alpha}\notin \cN_0$ while $f\cdot f^{-\alpha}\in \cN_0$. By assumption, $f^{-\alpha}\in \sD_X\cdot \cN_0$, which means $\cN_0$ cannot be a $\sD_X$-module.
\end{proof}

\subsection{An approximation of the Hodge filtration}
In this subsection, we set $\cS=\cO_X$ so $\cM = (\cO_X)_f$. We allow $X$ to be a complex manifold and $f$ holomorphic on $X$ (see discussion at the end of Section \ref{sec:weight}). In this case, $\cS$ and $\cM \cdot f^{-\alpha}$ underlie complex mixed Hodge modules \cite{MHMproject} for any $\alpha \in \mathbb{R}$. Note that while the jumps in $V^\bullet \iota_+ \cM$ occur only at rational numbers, when $\alpha \notin \Q$ the Hodge filtrations of the simple $\sD$-modules $\cM \cdot f^{-\alpha}$ can be still nontrivial. Recall that for any mixed Hodge module $\cN$, $V^{\alpha}\iota_{+}\cN$ inherits a Hodge filtration by
\[ F_{\bullet}V^{\alpha}\iota_{+}\cN=F_{\bullet}\iota_{+}\cN\cap V^{\alpha}\iota_{+}\cN,\]
where $F_{p+1}\iota_{+}\cN=\sum_{r+s\leq p}F_r\cN\otimes \d_t^s$.

\begin{prop}\label{prop: pmalpha for O}
Assume $\alpha \in \mathbb{R}_{>0}$. For $m\in \cM$ and $k\in \Z$, we have
   \[(p_{m,\alpha}(s))_{\leq k}\cdot m f^s = \C[s]\cdot m f^s \cap F_{k+1} V^{\alpha}\iota_{+}\cM,\]
   where for an ideal $I\subseteq \C[s]$, $I_{\leq k}$ denotes the subspace of its elements of degree $\leq k$.
\end{prop}

\begin{proof}
Since $\alpha>0$, $F_{k+1}V^{\alpha}\iota_{+}\cM = F_{k+1}V^{\alpha}\iota_{+}\cO_X=V^{\alpha}\iota_{+}\cM \cap F_{k+1}\iota_{+}\cO_X$. Under the isomorphism \eqref{eqn: Malgrange isomorphism}, we have
\[ F_{k+1}\iota_{+}\cO_X = \sum_{\ell=0}^{k} q_{\ell}(s)f^{-\ell}\cO_X\cdot f^s,\]
where $q_{\ell}(s)=\prod_{i=0}^{\ell-1}(s-i)$ and $q_0(s)=1$. Let $\ell \geq 0$ be the smallest integer such that $f^{\ell} m\in \cO_X$. Then $(s+1-\ell)\mid b_m(s)$. Because $\{q_{j}(s)\}_{j=0}^k$ forms a $\C$-basis for $\C[s]_{\leq k}$, 
\begin{equation} 
F_{k+1}\iota_{+}\cO_X\cap \C[s]mf^s = (q_{\ell}(s))_{\leq k}\cdot mf^s.
\end{equation}
From \eqref{eqn: pfunction via intersection of Valpha with mfs}, it follows that 
\[ F_{k+1}V^{\alpha}\iota_{+}\cM\cap \C[s]mf^s = \left(\mathrm{lcm}\{q_{\ell}(s),p_{m,\alpha}(s)\}\right)_{\leq k}\cdot mf^s.\]
Because $\alpha>0$, Corollary \ref{cor:explicitpfunction} dictates $\prod_{i=0}^{\ell-1}(s-i)^{\nu_{m,-i}}\mid p_{m,\alpha}(s)$. Lemma \ref{lemma: basic property of pfunction}(v) ensures $\nu_{m,-i}\geq \nu_{m,1-\ell}$ for all $0\leq i\leq \ell-1$. Proposition \ref{prop:nurecursivebound} then gives $\nu_{m,1-\ell} \geq \mult_{s=-(1-\ell)}b_m(s) \geq 1$. We conclude that $q_{\ell}(s)\mid p_{m,\alpha}(s)$, meaning 
\[ \left(\mathrm{lcm}\{q_{\ell}(s),p_{m,\alpha}(s)\}\right)_{\leq k}=(p_{m,\alpha}(s))_{\leq k}.\qedhere\]
\end{proof}
\begin{corollary}\label{cor:hodgedegree}
Let $\alpha\in \mathbb{R}$. Then $m f^{-\alpha} \in F_{\deg p_{m,\alpha}(s)}(\cM\cdot f^{-\alpha})$.
\end{corollary}

\begin{proof}
We can assume $ \alpha\geq 0$ by integer shifts. Then it follows from the surjection $F_{\bullet+1}V^{\alpha}\iota_{+}\cM\xrightarrow{ev_{s=-\alpha}}F_{\bullet}(\cM\cdot f^{-\alpha})$; see \cite[Corollary 1.2]{DY24} for example.
\end{proof}
\begin{definition}
    For $m \in \cM$ and $\alpha \in \mathbb{R}$, set $d_{m,\alpha} \colonequals \deg p_{m,\alpha}(s)$. For $\ell \in \Z_{\geq 0}$, we define the $P$-filtration of $\cM \cdot f^{-\alpha}$ as 
    \[\cP_{\ell} (\cM \cdot f^{-\alpha}) \colonequals \operatorname{span}_\C\{m f^{-\alpha} \in \cM \cdot f^{-\alpha} \mid d_{m,\alpha} \leq \ell\}.\]
\end{definition}

\begin{remark}
    For $\ell \geq 1$, one must take the span in the definition above because the set $\{m f^{-\alpha} \in \cM f^{-\alpha} \mid d_{m,\alpha} \leq \ell\}$ is generally not closed under addition. For example, let $X=\C^4$ with $f$ the $2\times 2$ determinant. Consider $m_1=x_1^2/f^4$, $m_2=1/f^3$, and $m=m_1+m_2$. Then $b_{m_1}(s)=s (s-3)$ and $b_{m_2}(s) = (s-1)(s-2)$ (see \cite[Section 5]{LY25April}). Consequently, $\nu_{m_1,-3}=\nu_{m_1,-2}= \nu_{m_1,-1}=1$ and $\nu_{m_2,-3}=0$, $\nu_{m_2,-2}=1$, $\nu_{m_2,-1}=2$, with all other weights evaluating to $0$ for negative $\alpha$. Because $\cN_{\ell}$ is a vector space, we must have $\nu_{m,-3} = 1$ and $\nu_{m,-1}=2$, and we also find $\nu_{m,-2}\geq 1$ (in fact, using \cite{M2} yields $b_m(s)=(s-1)(s-3)$, making $\nu_{m,-2}= 1$). Therefore, $d_{m,0}\geq 4$, whereas $d_{m_i,0}=3$ for $i=1,2$.
\end{remark}

\begin{prop}\label{prop: P contained in F}
    Let $g\in \cO_X$ and $m \in \cM$. Setting $U_g \colonequals \{g\neq 0\}$, we have
    \[d_{m_{\mid U_g}, \alpha} = \lim_{k\to \infty} d_{g^k \cdot m, \alpha}.\]
    Moreover, $\cP_{\ell}(\cM \cdot f^{-\alpha}) \subseteq F_\ell (\cM \cdot f^{-\alpha})$ is a coherent subsheaf of $\cO_X$-modules.
\end{prop}

\begin{proof}
    The containment follows from Corollary \ref{cor:hodgedegree}. From \eqref{eqn: explicit formula of p function via bfunction}, we see that
    \[ d_{m,\alpha} = \sum_{\beta\in \Q, \beta<\alpha}\nu_{m,\beta},\]
    which is a finite sum. The first part of the statement then follows from Proposition \ref{prop:restrnu} and the fact that the sequence $\{\nu_{g^k \cdot m, \alpha} \}_{k\geq 0}$ is non-increasing. By Lemma \ref{lem:pfunsum}(ii), $\cP_{\ell}(\cM \cdot f^{-\alpha})$ is also stable under $\cO_X$-multiplication, confirming it is an $\cO_X$-module.
\end{proof}

\begin{lemma}\label{lemma:F=P}
    We have $\cP_0\cM=F_0\cM$.
\end{lemma}

\begin{proof}
    By Proposition \ref{prop: P contained in F}, it suffices to show $F_0\subseteq \cP_0$. By \cite{BS05} and \cite[Proposition 10.1]{MP16}, the map $ev_{s=-1}$ induces an isomorphism
\[F_0\cM=F_0(\cM\cdot f^{-1}) \cong (\cO_Xf^s)\cap V^1\iota_{+}\cM. \]
Hence if $mf^{-1}\in F_0\cM\cdot f^{-1}$, then $mf^s\in V^1\iota_{+}\cM$ and so $p_{m,1}(s)=1$ by definition. Thus, $m f^{-1} \in \cP_0(\cM\cdot f^{-1})=\cP_0\cM$.
\end{proof}

Define recursively
\[ \cQ_{\ell}(\cM\cdot f^{-\alpha}) = \cP_{\ell}(\cM\cdot f^{-\alpha}) + F_1\sD_X\cdot \cQ_{\ell-1}(\cM\cdot f^{-\alpha}).\]
This corrects $\cP_\ell(\cM \cdot f^{-\alpha})$ to ensure it satisfies Griffiths transversality.

\begin{question}\label{hodge}
    Is it true that $\cQ_{\ell}(\cM\cdot f^{-\alpha})=F_{\ell}(\cM\cdot f^{-\alpha})$?
\end{question}

\subsection{Thom--Sebastiani formulas}\label{sec:thomseb}
In this section, we analyze how various invariants discussed previously including $p$-functions, $Z$-filtrations and weight levels $\omega_{f,\alpha}$ behave with respect to the Thom-Sebastiani formula.

Recall that $f$ is an algebraic function on a smooth algebraic variety $X$. Consider another smooth algebraic variety $Y$ with a non-constant algebraic function $g$, and set $h=f+g$. Throughout this section, we will assume 
\begin{equation}\label{eq:singcond} \mrm{Sing}(h^{-1}(0))=\mrm{Sing}(f^{-1}(0))\times \mrm{Sing}(g^{-1}(0)).\end{equation}
This can always be achieved by replacing $Y$ by a suitable Zariski-open neighborhood of $g^{-1}(0)$. It also holds when the critical locus of $g$ on $Y$ is contained in $g^{-1}(0)$. This can be achieved if $g$ is Euler-homogeneous, i.e., there exists a vector field $\chi$ such that $\chi \cdot g=g$.

Let $\iota_f:X\to X\times \C_{t_1}$ and $\iota_g:Y\to Y\times \C_{t_2}$ be the graph embeddings of $f,g$, and $t_1,t_2$ are coordinates on two copies of $\C$, and let $\iota_h: (X\times Y) \, \to \, (X\times Y)\times \C$ be the graph embedding of $h$. Let $\cS_1$ and $\cS_2$ be simple regular holonomic $\sD_X$- and $\sD_Y$-modules that underlie (complex) polarizable pure Hodge modules \cite{MHMproject}, on which $f$ and $g$ act injectively, respectively. Clearly $\cS:=\cS_1 \boxtimes \cS_2$  is simple and one can show that it underlies a polarizable pure Hodge module \cite[Theorem 3.28 and (2.17.4)]{Saito90} (for complex Hodge modules, see \cite[Proposition 16.3.2]{MHMproject}). We will assume that $h= f+g$ acts also injectively on $\cS$, which is satisfied whenever
\begin{equation}\label{eqn: support condition}
 \supp \cS_2 \,  \cap \, \{g=0\} \neq \emptyset,
 \end{equation}

Using the Malgrange isomorphism \eqref{eqn: Malgrange isomorphism}, we have
\[ \iota_{f,+}(\cS_1)_f\cong (\cS_1)_f[s_1]f^{s_1}, \quad \iota_{g,+}(\cS_2)_g\cong (\cS_2)_g[s_2]g^{s_2},\quad \iota_{h,+}\cS_h\cong \cS_h[s]h^s.\]
We will view $ \iota_{f,+}\cS_1 \cong \cS_1[\d_{t_1}]$ inside $(\cS_1)_f[s_1]f^{s_1}$ via the isomorphism above, in which case the operator $\d_{t_1}$ corresponds to $-s_1 t_1^{-1}$.  Following \cite[(2.1.3)]{MSS}, consider the \emph{algebraic partial microlocalization} of $\iota_{f,+}\cS_1$ and its $V$-filtration:
\[ \widetilde{\cS_1}=\cS_1[\d_{t_1},\d_{t_1}^{-1}], \quad V^{\alpha}\widetilde{\cS_1}=\sum_{i\in \Z_{\geq 0}}\d_{t_1}^{-i}V^{\alpha-i} \iota_{f,+} \cS_1, \quad \forall \alpha\in \mathbb{R},\]
The inclusion $  \iota_{f,+}\cS_1\hookrightarrow  \widetilde{\cS_1}$ induces a filtered map
\begin{equation}\label{eqn: inclusion filtered}
( \iota_{f,+}\cS_1, V^{\bullet})\to (\widetilde{\cS_1},V^{\bullet}),
\end{equation}
which is strict if $\alpha<1$. For $\ell \in \Z_{\geq 0}$ and $\alpha \in \mathbb{R}$, let us define
\[Z_\ell V^\alpha \widetilde{\cS_1} = \sum_{i\in \Z_{\geq 0}}\d_{t_1}^{-i} Z_\ell V^{\alpha-i} \iota_{f,+} \cS_1,\]
We consider analogously $\widetilde{\cS_2}=\cS_2[\d_{t_2},\d_{t_2}^{-1}],  \widetilde{\cS}=\cS[\d_t,\d_t^{-1}]$ and the related $V$- and $Z$-filtrations. Following \cite[(2.2.4)]{MSS}, consider the \emph{convolution} of $V$-filtrations:
\begin{equation}\label{eqn: V on boxproduct} V^{\gamma}(\widetilde{\cS_1}\boxtimes \widetilde{\cS_2})=\sum_{\alpha+\beta=\gamma}V^{\alpha}\widetilde{\cS_1}\boxtimes V^{\beta}\widetilde{\cS_2}.\end{equation}
Since we assume \eqref{eq:singcond} holds, the Thom-Sebastiani formula \cite[(2.2.3)]{MSS} says that there is a filtered short exact sequence 
\begin{equation}\label{eqn: TS V filtration} 0\to (\widetilde{\cS_1}\boxtimes \widetilde{\cS_2},V^{\bullet+1})\xrightarrow{\d_{t_1}\boxtimes \mathrm{id}-\mathrm{id}\boxtimes \d_{t_2}} (\widetilde{\cS_1}\boxtimes \widetilde{\cS_2},V^{\bullet})\xrightarrow{\eta}(\widetilde{\cS},V^{\bullet})\to 0,\end{equation}
where the last map $\eta$ is defined by
\[ \eta\left(m_1\d_{t_1}^{i}\boxtimes m_2\d_{t_2}^{j}\right)=(m_1\boxtimes m_2)\d_t^{i+j},\quad m_1\in \cS_1, \, m_2\in \cS_2, \, i,j\in \Z.\]

We first observe that \eqref{eqn: TS V filtration} interacts nicely with $Z$-filtrations.

\begin{lemma}\label{lem:evazfilt}
For $\ell_1 , \ell_2 \in \Z_{\geq 1}$ and $\alpha, \beta \in \mathbb{R}$, we have
\[\eta\left(Z_{\ell_1} V^\alpha \widetilde{\cS_1} \boxtimes Z_{\ell_2} V^\beta \widetilde{\cS_2}  \right ) \,\subseteq \, Z_{\ell_1+ \ell_2-1}V^{\alpha+\beta}\widetilde{\cS}. \]
\end{lemma}

\begin{proof}
Set $\ell = \ell_1+\ell_2-1$. Take arbitrary elements $w_1 \in Z_{\ell_1} V^{\alpha-i} \iota_{f,+} \cS_1$ and $w_2 \in Z_{\ell_2} V^{\beta-j} \iota_{g,+} \cS_2$ and consider
\[ \widetilde{w}_1 = \partial_{t_1}^{-i} w_1, \quad \widetilde{w}_2 = \partial_{t_2}^{-j} w_2,\quad \widetilde{w}:= \eta(\widetilde{w}_1 \boxtimes \widetilde{w}_2).\]
It suffices to show that $(s+\alpha+\beta)^\ell \widetilde{w} \, \in V^{>(\alpha+\beta)}\widetilde{\cS}$. The identity $\eta\circ (s_1+s_2)=s\circ\eta$ gives
\begin{align}\label{eq:binom}
(s+\alpha+\beta)^\ell \cdot \widetilde{w} &= \eta((s_1+\alpha +s_2 + \beta)^\ell \cdot (\widetilde{w}_1 \boxtimes \widetilde{w}_2)) \nonumber \\
&=  \sum_{a+b=\ell}  {\ell \choose a} \cdot \eta(\left  (s_1+\alpha)^a \widetilde{w}_1 \boxtimes (s_2+\beta)^b \widetilde{w}_2 \right).
\end{align}
For each summand above we have  
\[(s_1+\alpha)^a \widetilde{w}_1 \boxtimes (s_2+\beta)^b \widetilde{w}_2 = \partial_{t_1}^{-i} (s_1+\alpha-i)^a  w_1 \boxtimes \partial_{t_2}^{-j} (s_2+\beta-j)^b  w_2.\]
Since $a+b = \ell_1 + \ell_2 -1$, we must have either $a \geq \ell_1$ or $b\geq \ell_2$, so either $(s_1+\alpha-i)^a  w_1 \in V^{>\alpha-i}$ or  $(s_2+\beta-j)^b  w_2 \in V^{>\beta-j}$. Either way, we have $(s_1+\alpha)^a \widetilde{w}_1 \boxtimes (s_2+\beta)^b \widetilde{w}_2 \in V^{>(\alpha+\beta)} (\widetilde{\cS_1} \boxtimes \widetilde{\cS_2})$. This shows that  $(s+\alpha+\beta)^\ell \widetilde{w} \, \in V^{>(\alpha+\beta)}\widetilde{\cS}$, as claimed.
\end{proof}

Next, we define a modified version of the $p$-function. For $w \in \iota_{h, +} \cS$ and $\gamma \in \mathbb{R}$, we set 
\[q_{w, \gamma}(s):= \frac{p_{w, \gamma}(s)} {\gcd(p_{w, \gamma}(s), [s+1]_{\lfloor \gamma \rfloor})}.\]
The following is an easy consequence of the definition of $p$-functions:
\begin{equation}\label{eq:qfun}
\C[s] \cdot w \, \cap \, V^\gamma\widetilde{\cS} =  (q_{w,\gamma}(s) ) \cdot w.
\end{equation}
We define the same way $q_{w_1,\alpha}(s),q_{w_2,\beta}(s)$ for $w_1\in\iota_{f,+}\cS_1,w_2\in \iota_{g,+}\cS_2$.

For polynomials $p(s) = \prod_{i=1} (s+\alpha_i)^{\nu'_i}$ with $\alpha_i \in \mathbb{R}$ distinct, $q(s) =\prod_{j=1} (s+\beta_j)^{\nu''_j}$ with $\beta_j \in \mathbb{R}$ distinct and $\nu'_i, \nu''_j \in \Z_{\geq 1}$, and $\gamma \in \mathbb{R}$, we define their \emph{$\gamma$-convolution} as
\[p(s) *_\gamma q(s) := \prod_{r < \gamma } (s+r)^{m_r},\]
where $m_r := \max_{\alpha_i + \beta_j =r} \{\nu'_i + \nu''_j  - 1\}$ (and $=0$ if this set is empty).

\begin{lemma}\label{lem:thomseb}
Consider $w_1 \in \iota_{f,+}\cS_1$ and $w_2 \in \iota_{g,+}\cS_2$, and set $w:= w_1 \boxtimes w_2$. Then for any $\gamma \in \mathbb{R}$, we have 
\[q_{w, \gamma}(s)  \,\, | \,\,  q_{w_1 , k}(s) *_\gamma q_{w_2, k}(s), \quad \mbox{with } k\gg 0. \]
\end{lemma}

\begin{proof}
Set $q(s) = q_{w_1 , k}(s) *_\gamma q_{w_2, k}(s),$  with  $k\gg 0$, and note this does not depend on $k \gg 0$ by Corollary \ref{cor:multpfun} (see also Corollary \ref{cor:explicitpfunction}). 

To show that $q_{w, \gamma}(s) \, | \, q(s)$, by (\ref{eq:qfun}) it is enough to show that $q(s) \cdot w \in V^\gamma\widetilde{\cS}$. To do so, we claim that there exist $c_{\alpha,i}(s_1) \in \C[s_1], d_{\beta,i}(s_2) \in \C[s_2]$, such that
\begin{equation}\label{eq:binom2}
q(s) \cdot w = \eta\left( \sum_{\alpha+\beta \geq \gamma; \, i} c_{\alpha, i}(s_1) q_{w_1, \alpha}(s_1) w_1 \boxtimes d_{\beta, i}(s_2) q_{w_2, \beta}(s_2) w_2    \right).
\end{equation}
This follows by induction on $\gamma$ using binomial expansion as in (\ref{eq:binom}) and recursion as in Corollary \ref{cor:multpfun}. Note that by (\ref{eq:qfun}) each summand satisfies
\[c_{\alpha,i}(s_1) q_{w_1, \alpha}(s_1) w_1 \boxtimes d_{\beta, i}(s_2) q_{w_2, \beta}(s_2) w_2 \in V^{\alpha} \widetilde{\cS_1} \boxtimes V^\beta \widetilde{\cS_2}.\]
Thus, by (\ref{eqn: TS V filtration}) we have $q(s) \cdot w  \in V^\gamma\widetilde{\cS}$.
\end{proof}

\smallskip

For any $m\in \cS$ and $\gamma \in \mathbb{R}$, set 
\[ \tilde{\nu}_{m, \gamma}=\begin{cases} \max \{ \nu_{m, \gamma} - 1, \, 0 \}, \quad &\textrm{if $\gamma \in \Z_{\geq 1}$},\\
\nu_{m,\gamma}, \quad &\textrm{otherwise}.\end{cases}\]
We define the notion for elements in $\cS_i$ similarly. In the same way, set 
\begin{equation}\label{eqn: modified omega} \tilde{\omega}_{m, \gamma}=\begin{cases} \max \{ \omega_{m, \gamma} - 1, \, 0 \}, \quad &\textrm{if $\gamma \in \Z_{\geq 1}$},\\
\omega_{m,\gamma}, \quad &\textrm{otherwise}.\end{cases}\end{equation}
As in Lemma \ref{lem:Bfirst}, we can characterize $\widetilde{\omega}_{m, \gamma}$ as the smallest number $\ell \in \Z_{\geq 0}$ such that 
\begin{equation}\label{eq:tildeweight}
\exists w \in \widetilde{\cS} \,\, \mbox{ with } \,\, mh^s + (s+\gamma) w  \, \in Z_\ell V^\gamma \widetilde{\cS}.
\end{equation}

By Lemma \ref{lem:thomseb} and Corollary \ref{cor:explicitpfunction} we have the following.
\begin{corollary}\label{cor:thomsebineq} For $m_1 \in \cS_1$ and $m_2 \in \cS_2$ and $\gamma \in \mathbb{R}$, we have
\[\tilde{\nu}_{m, \gamma} \leq \max_{\substack{\alpha + \beta =  \gamma \\ \tilde{\nu}_{m_1, \alpha} \, \neq \, 0 \, \neq \, \tilde{\nu}_{m_2, \beta}}} \!\!\! \{\tilde{\nu}_{m_1, \alpha} + \tilde{\nu}_{m_2, \beta} -1\}.\]
\end{corollary}
Next, we discuss some conditions that give equality above. The exact sequence (\ref{eqn: TS V filtration}) induces the following exact sequence for any $\gamma \in \mathbb{R}$  \cite{MSS}:
\[ 0 \to \bigoplus_{\alpha} \gr_V^\alpha \widetilde{\cS_1} \boxtimes \gr_V^{\gamma+1-\alpha} \widetilde{\cS_2}  \xrightarrow{\d_{t_1} \boxtimes \id - \id \boxtimes \d_{t_2}} \bigoplus_{\alpha} \gr_V^\alpha \widetilde{\cS_1} \boxtimes \gr_V^{\gamma-\alpha} \widetilde{\cS_2}  \xrightarrow{\,  d \,}  \!\! \! \bigoplus_{\alpha \in (-1,0]} \!\!\! \gr_V^\alpha \widetilde{\cS_1} \boxtimes \gr_V^{\gamma-\alpha}  \widetilde{\cS_2} \to 0\]
where $d$ is given by the diagonal map scaled by appropriate powers of $\d_{t_1} \boxtimes \d_{t_2}^{-1}$ on the components, and following isomorphism induced by $\eta$:
\[\bigoplus_{\alpha \in (-1,0]} \gr_V^\alpha \widetilde{\cS_1} \boxtimes \gr_V^{\gamma-\alpha} \widetilde{\cS_2} \,\, \xrightarrow{\sim} \, \gr_V^{\gamma} \widetilde{\cS}.\]
For given $m_i \in \cS_i$ and $\gamma \in \mathbb{R}$, we will denote the sets
\[ \Pi_\gamma\colonequals \{ \alpha \in \mathbb{R} \, | \, \tilde{\nu}_{m_1,\alpha} \neq 0 \mbox{ and } \tilde{\nu}_{m_2, \gamma-\alpha} \neq 0\},\]
\[ \Pi_\gamma^{max}\colonequals \{ \alpha \in \Pi_\gamma \, | \, \tilde{\nu}_{m_1,\alpha} + \tilde{\nu}_{m_2, \gamma-\alpha} =\max_{\alpha' \in \Pi_\gamma} \{\tilde{\nu}_{m_1, \alpha'} + \tilde{\nu}_{m_2, \gamma-\alpha'} \} \}.\]

\begin{thm} \label{thm:thomsebequality}
For $m_i \in \cS_i$, set $m=m_1\boxtimes m_2 \in \cS$, and let $\gamma \in \mathbb{R}$.  Assume that either of the following condition holds:
\begin{itemize}
\item[(a)] $\exists \alpha \in \Pi_\gamma^{max}$ such that $\alpha + k \notin \Pi_\gamma^{max}$ for any $k \in \Z \setminus \{0\}$; or
\item[(b)] $\exists  \alpha \in \Pi_\gamma^{max}$ such that $\alpha-k \notin \Pi_\gamma^{max}$ for any $k \in \Z_{>0}$, $\tilde{\nu}_{m_1,\alpha} \neq \tilde{\nu}_{m_1,\alpha+k'}$ for $k'\in \Z_{>0}$ with $\alpha+k' \in \Pi_\gamma^{max}$, and $\tilde{\nu}_{m_1, \alpha} = \widetilde{\omega}_{m_1, \alpha}$; or
\item[(c)] $\exists$ vector field $\chi$ such that $\chi \cdot g = g$ and $\chi \cdot m_2 = c \cdot m_2$, for some $c \in \C$.
\end{itemize}
Then we have
\[\tilde{\nu}_{m, \gamma} = \max_{\alpha \in \Pi_\gamma } \{\tilde{\nu}_{m_1, \alpha} + \tilde{\nu}_{m_2, \gamma-\alpha} -1\}.\]
\end{thm}

\begin{proof}
Set $\tilde{\nu}:= \max_{\alpha \in \Pi_\gamma } \{\tilde{\nu}_{m_1, \alpha} + \tilde{\nu}_{m_2, \gamma-\alpha} -1\}$, $w_1 = m_1 f^{s_1}$,  $w_2 = m_2 g^{s_2}$ and $w=m h^s$. Suppose, by contradiction, that $\tilde{\nu}_{m, \gamma} < \tilde{\nu}$, and let $q(s) := q_{w_1 , k}(s) *_{\gamma} q_{w_2, k}(s), \mbox{ with } k\gg 0$. This implies that 
\[ (s+\gamma)^{\tilde{\nu}-1} \cdot q(s) \cdot w  \in V^{>\gamma} \widetilde{\cS}.\]
Expanding as in (\ref{eq:binom2}) we get for some $c_\alpha \in \C$
\[\eta \left( \sum_{\substack{\alpha \in \Pi^{max}_{\gamma} \\ \alpha+\beta = \gamma}}  c_\alpha \cdot (s_1+\alpha)^{\tilde{\nu}_{m_1,\alpha}-1}q_{w_1, \alpha}(s_1) w_1 \boxtimes (s_2+\beta)^{\tilde{\nu}_{m_2,\beta}-1} q_{w_2, \beta}(s_2) w_2    \right) \,\, \in V^{>\gamma} \widetilde{\cS}.\]
In fact, we claim that all $c_\alpha$ are non-zero. For this, by closer inspection, we see that in the formula (\ref{eq:binom2}) the coefficients $c_{\alpha,i}(s_1)$ (resp. $d_{\beta,i}(s_2)$) can be realized as monic polynomials having roots $>-\alpha$ (resp. $>-\beta$) with the sums $\deg c_{\alpha,i}(s_1)+ \deg d_{\beta,i}(s_2)$ being independent of $i$. As $s_1$ (resp. $s_2$) acts on the classes of the elements  $(s_1+\alpha)^{\tilde{\nu}_{m_1,\alpha}-1}q_{w_1, \alpha}(s_1) w_1$ (resp. $(s_2+\beta)^{\tilde{\nu}_{m_2,\beta}-1} q_{w_2, \beta}(s_2) w_2$) as $-\alpha$ (resp. $-\beta$) in $\gr_V^\alpha\widetilde \cS_1$ (resp. $\gr_V^\beta\widetilde \cS_2$), evaluating such monic polynomials at $-\alpha$ (resp. $-\beta$) give non-zero real coefficients of signature independent of $i$, thus their sum $c_\alpha$ will be non-zero.

Pick any $\alpha_0 \in (-1, 0]$. Via the isomorphisms above, we have that the class of the following element is zero in $\gr_V^{\alpha_0} \widetilde{\cS_1} \boxtimes{ \gr_V^{\gamma-\alpha_0} \widetilde{\cS_2}}$:
\begin{equation}\label{eq:sumtensors}
\!\!\!\sum_{\substack{k \in \Z, \, \alpha_0 + k \in \Pi^{max}_{\gamma} \\ \beta_0 = \gamma-\alpha_0}} \!\! \!\! \!\!\!\! c_{\alpha_0+k} \cdot (s_1+\alpha_0)^{\tilde{\nu}_{m_1,\alpha_0+k}-1}q_{\d_{t_1}^k w_1, \alpha_0}(s_1) \d_{t_1}^k w_1 \boxtimes (s_2+\beta_0)^{\tilde{\nu}_{m_2, \beta_0-k}-1} q_{\d_{t_2}^{-k}w_2, \beta_0}(s_2) \d_{t_2}^{-k}w_2 .
\end{equation}
Note that none of the summands are zero in $\gr_V^{\alpha_0} \widetilde{\cS_1} \boxtimes{ \gr_V^{\gamma-\alpha_0} \widetilde{\cS_2}}$. Now we proceed according to the cases (a) -- (c).

In case (a), the hypothesis implies that there exists an $\alpha_0 \in (-1, 0]$ such that the element (\ref{eq:sumtensors}) has a single summand. As this is non-zero, this gives a contradiction.

Consider case (b).  Given $\alpha$ as in the hypothesis in (b), set $\alpha_0 := \alpha - \lceil \alpha \rceil$. Since (\ref{eq:sumtensors}) is zero in  $\gr_V^{\alpha_0} \widetilde{\cS_1} \boxtimes{ \gr_V^{\gamma-\alpha_0} \widetilde{\cS_2}}$, this implies that there exists a linear combination with $a_k \in \C$, the first non-zero being $a_{\lceil \alpha \rceil} =1$, of the form
\[\sum_{k \in \Z, \alpha_0 + k \in \Pi_\gamma^{max}} a_k \cdot (s+\alpha_0)^{\tilde{\nu}_{m_1,\alpha_0+k}-1}q_{\d_{t_1}^k w_1, \alpha_0}(s_1) \d_{t_1}^k w_1  \,\, \in  V^{>\alpha_0} \widetilde{\cS_1}= Z_0 V^{\alpha_0} \widetilde{\cS_1}.\]
Together with the hypothesis, this implies after multiplying by $\d_{t_1} ^{-\lceil \alpha \rceil}$ that for some $w' \in \widetilde{\cS}$ we have an equation of the form
\[q_{w_1, \alpha}(s_1) \cdot w_1 + (s_1+\alpha)w' \in Z_{\tilde{\nu}_{m_1, \alpha}-1}V^{\alpha} \widetilde{\cS_1}.\]
By (\ref{eq:tildeweight}), this implies $\widetilde{\omega}_{m_1, \alpha} \leq \tilde{\nu}_{m_1, \alpha}-1$, a contradiction.

Now consider case (c). Take any $\alpha_0 \in (-1,0]$ such that $(\alpha_0 + \Z) \cap \Pi_{\gamma}^{max} \neq \emptyset$, and set $\beta_0=\gamma - \alpha_0$. Similar to part (b), since (\ref{eq:sumtensors}) is zero in $\gr_V^{\alpha_0} \widetilde{\cS_1} \boxtimes{ \gr_V^{\beta_0} \widetilde{\cS_2}}$, this implies that there exists a non-trivial linear combination with $b_k \in \C$
\[\sum_{k \in \Z, \alpha_0 + k \in \Pi_\gamma^{max}} b_k \cdot (s+\beta_0)^{\tilde{\nu}_{m_2,\beta_0-k}-1}q_{\d_{t_2}^{-k} w_2, \beta_0}(s_2) \d_{t_2}^{-k} w_2  \in  V^{>\beta_0} \widetilde{\cS_2}.\]
The hypothesis implies that the $k$th summand is an eigenvector of $L:=\chi - s-c$ with eigenvalue $k$. Let $k_0 \in \Z$ such that $b_{k_0} \neq 0$. Then applying the operator $\displaystyle\prod_{\substack{k \neq k_0\\ \alpha_0 + k \in \Pi_\gamma^{max}}} (L-k)$ to the sum above, we obtain the contradiction
\[(s+\beta_0)^{\tilde{\nu}_{m_2,\beta_0-k_0}-1}q_{\d_{t_2}^{-k_0} w_2, \beta_0}(s_2) \d_{t_2}^{-k_0} w_2  \in  V^{>\beta_0} \widetilde{\cS_2}.\qedhere\]
\end{proof}

\begin{remark}
    Note that condition (c) above is analogous to Saito's condition in \cite[Theorem 0.8]{Saito94} (with $m_1 = 1 \in \cS_1 = \cO_X$, $m_2=1 \in \cS_2 = \cO_Y$). However we don't see a direct way to obtain one result from the other via interpreting $\nu$ as multiplicities of roots of powers of Bernstein--Sato polynomials, since the sum of higher powers of $f$ and $g$ is not a power of their sum $h$.
\end{remark}

In a way, the behavior of $\nu$ is better-behaved than the multiplicities of Bernstein--Sato polynomials in such Thom--Sebastiani-type results, as the following example illustrates.

\begin{example}\label{ex:tpqr}
    We consider \cite[Example 4.8]{Saito94}, where it was shown that the reduced Bernstein--Sato polynomial $\tilde{b}_h(s)$ does not equal the expected convolution of $\tilde{b}_f$ and $\tilde{b}_g$, and thus \cite[Theorem 0.8]{Saito94} does not necessarily hold without the Euler-homogeneous assumption on $g$ (similar to Theorem \ref{thm:thomsebequality}(c) above). On the other hand, we demonstrate that on the level of $\tilde{\nu}$, the expected equality in Theorem \ref{thm:thomsebequality} does hold in this case.

    Consider the $T_{pqr}$-type isolated singularity on $\C^3$ given by: 
    \[T_{p,q,r}(x,y,z) = x^p+ y^q + z^r + xyz, \,\, \mbox{ with } p,q,r \in \Z_{\geq 2} \mbox{ such that }\frac{1}{p}+\frac{1}{q}+\frac{1}{r}<1.\] For $f := T_{p,p,p}$ we have (see \cite[Example 4.8]{Saito94}) 
    \[\tilde{b}_{f}(s) = (s+1)^2 \cdot \prod_{i=1}^{p-1} (s+1 +\frac{i}{p}).\]
    Proposition \ref{prop:nurecursivebound} readily implies that the only non-zero $\tilde{\nu}$ are given by 
    \begin{itemize}
        \item $\tilde{\nu}_{f^s, \alpha} = 2$, when $\alpha \in \Z_{\geq 1}$;
        \item $\tilde{\nu}_{f^s, \alpha} = 1$, when $\alpha = i/p$ with $i \in \Z_{>p}$ and $p \nmid i$.
    \end{itemize}
    Now take $g:= T_{q,q,q}$ with $\gcd(p,q)=1$, and set $h=f+g$. By Corollary \ref{cor:thomsebineq}, the only potentially non-zero $\tilde{\nu}$ are given as follows:
    \begin{itemize}
        \item $\tilde{\nu}_{h^s, \gamma} \leq 3$, for $\gamma \in \Z_{\geq 2}$;
        \item $\tilde{\nu}_{h^s, \gamma} \leq 2$, for $\gamma = 2+ i/p$ (resp. $\gamma = 2+ i/q$) with $i \in \Z_{>0}$ and $p \nmid i$ (resp. $q \nmid i)$;
        \item $\tilde{\nu}_{h^s, \gamma} \leq 1$, for $\gamma = k+ i/p+j/q$ with $k\in \Z_{>1}$ and $1\leq i \leq p-1$, $1\leq j \leq q-1$.
    \end{itemize}
    Since the sequences $\{\tilde{\nu}_{h^s, \gamma+k}\}_{k \in \Z_{\geq 0}}$ are always non-decreasing, to show that the equalities hold above, it is enough to do so for the minimal possible $\gamma$. Now Theorem \ref{thm:thomsebequality}(a) implies
    \begin{itemize}
        \item $\tilde{\nu}_{h^s, 2} = 3$;
        \item $\tilde{\nu}_{h^s, 2+i/p}= 2$ (resp. $\tilde{\nu}_{h^s, 2+j/q}= 2$) for $1\leq i \leq p-1$ (resp. $1\leq j \leq q-1$);
        \item $\tilde{\nu}_{h^s, 2+i/p+j/q} = 1$, with $1\leq i \leq p-1$, $1\leq j \leq q-1$.
    \end{itemize}
    This shows that in this example we do indeed have equality for any $\gamma$ (and $m_i=1$) in Corollary \ref{cor:thomsebineq}, in contrast with the failure of \cite[Theorem 0.8]{Saito94} for $h=f+g$.
\end{example}

We conclude with a (partial) Thom--Sebastiani-type result on weight levels.

\begin{prop}\label{prop:thomsebweight}
Let $\gamma \in \mathbb{R}$,  $m_i \in \cS_i$ and set $m= m_1 \boxtimes m_2$. Assume that for any $\alpha \in \Pi_\gamma$ we have  $\tilde{\nu}_{m_2,\gamma-\alpha} = 1$. Then the following holds:
\[\widetilde{\omega}_{m, \gamma} \leq \max_{\alpha \in \Pi_\gamma} \, \widetilde{\omega}_{m_1, \alpha}.\]
\end{prop}
\begin{proof}
Set $w=mh^s$, $w_1 = m_1 f^{s_1}$, and $w_2 = m_2 g^{s_2}$. Take  $q(s) = q_{w_1 , k}(s) *_\gamma q_{w_2, k}(s)$ with $k\gg 0$. It follows from (\ref{eq:binom2}) that we can find $c_{\alpha,i}(s_1)$, $d_{\beta,i}(s_2)$ and $w_{>\gamma}  \in  V^{>\gamma}\widetilde{\cS}$ so that
\[q(s) \cdot w = w_{>\gamma} + \eta\bigg( \! \sum_{\substack{\alpha \in \Pi_\gamma;\, i \\ \alpha+\beta = \gamma}} c_{\alpha,i}(s_1) q_{w_1, \alpha}(s_1) w_1 \boxtimes d_{\beta,i}(s_2) q_{w_2,\beta}(s_2) w_2    \bigg).\]
By (\ref{eq:tildeweight}), there exist $w'_\alpha \in \widetilde{\cS_1}$ such that  $q_{w_1, \alpha}(s_1) w_1 + (s_1 + \alpha)w'_\alpha \in Z_{\widetilde{\omega}_{m_1, \alpha}} V^\alpha \widetilde{\cS_1}$. In particular, then we must have $w'_\alpha \in V^\alpha \widetilde{\cS_1}$.
Consider the expression 
\begin{align*}
q(s)\cdot w + &(s+\gamma) \cdot \eta\ \bigg(\!\sum_{\substack{\alpha \in \Pi_\gamma;\, i   \\ \alpha+\beta = \gamma}} c_{\alpha,i}(s_1)w'_\alpha \boxtimes d_{\beta, i}(s_2) q_{w_2,\beta}(s_2)w_2  \bigg) \\
=  w_{>\gamma} + \eta \bigg( &\sum_{\substack{\alpha \in \Pi_\gamma;\, i  \\ \alpha+\beta = \gamma}} c_{\alpha,i}(s_1) (q_{w_1, \alpha}(s_1)w_1 +(s_1 + \alpha)w'_\alpha) \boxtimes d_{\beta,i}(s_2) q_{w_2,\beta}(s_2)w_2   \\
+&\sum_{\substack{\alpha \in \Pi_\gamma;\, i  \\ \alpha+\beta = \gamma}} c_{\alpha,i}(s_1)w'_\alpha \boxtimes d_{\beta,i}(s_2) (s_2+\beta) q_{w_2,\beta}(s_2)w_2   \bigg).\end{align*}
Set $\widetilde{\omega} =\max_{\alpha \in \Pi_\gamma} \widetilde{\omega}_{m_1, \alpha}$, then we have
\begin{align*}
c_{\alpha,i}(s_1) (q_{w_1, \alpha}(s_1)w_1 +(s_1 + \alpha)w'_\alpha) \boxtimes d_{\beta,i}(s_2) q_{w_2,\beta}(s_2)w_2 &\in Z_{\widetilde{\omega}_{m_1, \alpha}} V^\alpha \widetilde{\cS_1} \boxtimes Z_{1} V^{\beta} \widetilde{\cS_2}, \\
c_{\alpha,i}(s_1)w'_\alpha \boxtimes d_{\beta, i}(s_2) (s_2+\beta) q_{w_2,\beta}(s_2)w_2 &\in V^\alpha \widetilde{\cS_1} \boxtimes V^{>\beta} \widetilde{\cS_2}.\end{align*}
Thus, by Lemma \ref{lem:evazfilt}  and (\ref{eqn: TS V filtration}), we obtain that the whole sum belongs to $Z_{\widetilde{\omega}} V^{\gamma} \widetilde{\cS}$, and since $q(-\gamma) \neq 0$, by (\ref{eq:tildeweight}) this implies that $\widetilde{\omega}_{m, \gamma} \leq \widetilde{\omega}$, as claimed.
\end{proof}

\section{Applications}
In this section, we deduce several applications from the results developed in the previous sections. Let $\iota:X\to X\times \C_t$ denote the graph embedding of $f$. Let $\cM$ be a holonomic $\sD_X$-module, and let $V^{\bullet}\iota_{+}\cM$ denote the $\C$-indexed $V$-filtration on $\iota_{+}\cM$ along $\{t=0\}$ (Definition \ref{defn:CV-filtration}).

\subsection{Nilpotency index}
In this subsection, we prove Theorem \ref{thm:nilpotency index}, Corollary \ref{corollary: nilpotency index for function} and Proposition \ref{prop: strict lower bound}. Let $f$ be an algebraic function on a smooth complex algebraic variety $X$,  $\cS$ be a simple holonomic $\sD_X$-module and let $\cM=\cS_f$ be its localization along $\{f=0\}$. Fix $\alpha\in \C$ such that $\gr^{\alpha}_V\iota_{+}\cM\neq 0$, and let $0\neq m\in \cM$.

\begin{thm}\label{thm: stabilization of nu and omega}
For any $0\neq m\in \cM$, if $k\geq \max\{k_0\in \Z_{\geq 0}\mid b_m(-\alpha-k_0)=0\}$, then $\omega_{m,\alpha+k}=\nu_{m,\alpha+k}=n_{\alpha}$.
\end{thm}
\begin{proof}
    Set $m'\colonequals f^{-k}m \in \cM$. Lemma \ref{lem: b function of tw} gives $b_{m'}(s) = b_m(s- k)$. By hypothesis, $b_{m'}(s)$ has no roots in $-\alpha-\Z_{> 0}$. Corollary \ref{cor: generation criterion in terms of roots} then implies $\sD_X \cdot(m'f^{-\alpha})= \cM \cdot f^{-\alpha}$. Applying Corollary \ref{cor:gennu} and Lemma \ref{lemma: basic property of pfunction}(iv) yields $n_\alpha = \nu_{m',\alpha}=\nu_{f^{-k}m,\alpha}=\nu_{m,\alpha+k}$. On the other hand, because $\cM\cdot f^{-\alpha}=\cM\cdot f^{-\alpha-k}$, Corollary \ref{cor:gennu} implies that $n_{\alpha}=\omega_{m',\alpha}=\omega_{m,\alpha+k}$.
\end{proof}

\begin{corollary}\label{cor: bound the nilpotency index by multiplicity}
We have $\mu_{m,\alpha}\leq \nu_{m,\alpha}\leq n_\alpha \leq \sum_{\gamma\in \alpha+\Z} \mu_{m, \gamma}$.
\end{corollary}

\begin{proof}
By Corollary \ref{cor:nilpdegmax}, $n_{\alpha}=\max_{m'\in \cM}\nu_{m',\alpha}$. Proposition \ref{prop:nurecursivebound} then ensures $n_{\alpha}\geq \nu_{m,\alpha}\geq \mu_{m,\alpha}$. Conversely, let $-\beta$ be the smallest root of $b_m(s)$ in $-\alpha+\Z$. Theorem \ref{thm:nilpotency index} gives $n_{\alpha}=\nu_{m,\beta}$. Iteratively applying Proposition \ref{prop:nurecursivebound} yields 
\[ n_{\alpha}=\nu_{m,\beta}\leq \sum_{\gamma\in \beta+\Z} \mu_{m,\gamma}=\sum_{\gamma\in \alpha+\Z} \mu_{m, \gamma}.\qedhere\]
\end{proof}

\begin{proof}[Proof of Corollary \ref{corollary: nilpotency index for function}]
    This follows immediately from Theorem \ref{thm:nilpotency index}, Corollary \ref{cor: bound the nilpotency index by multiplicity}, and Theorem \ref{thm: numalpha via higher order pfunction}.
\end{proof}

\begin{prop}\label{prop:strictnumu}
    If $\alpha \in \Z_{\geq 2}$, then $\nu_{1, \alpha} > \mult_{s=-\alpha} b_1(s)$.
\end{prop}

\begin{proof}
 By Definition \ref{def: reducedpfunction}, $\nu_{1,\alpha}=\mult_{s=-\alpha} b_{\tilde{p}_{1,\alpha}(s)f^s}(s)$, where the reduced $p$-function is computed using \eqref{eqn: explicit reduced rearranged}: $\tilde{p}_{1,\alpha}(s)=\prod_{i=1}^{\alpha-1} (s+i)^{\nu_{1,i}}$.
    Using \eqref{eqn: nu11 is mult}, Lemma \ref{lemma: basic property of pfunction}(v) implies:
    \begin{equation}\label{eqn: inequalities of chains of nu}
    1\leq \mult_{s=-1}b_{1}(s)=\nu_{1,1}\leq \nu_{1,2}\leq \cdots \leq \nu_{1,i}, \quad \forall i\geq 1.
    \end{equation}
    In particular, we see
    \[ (s+1)(s+2)\cdots (s+\alpha-1) \mid \tilde{p}_{1,\alpha}(s).\]
    Successive applications of Lemma \ref{lem:easy2} give
    \[ \mult_{s=-\alpha}b_{(s+1)(s+2)\cdots (s+\alpha-1)f^s}(s)\leq  \mult_{s=-\alpha}b_{\tilde{p}_{1,\alpha}(s)f^s}(s)=\nu_{1,\alpha}.\]
    By \cite[Proposition 2.8]{DLY} (see the equation below \cite[(6)]{DLY}), we know
    \[ b_{(s+1)(s+2)\cdots (s+\alpha-1)f^{s}}(s)=(s+\alpha)b_{f^s}(s)/(s+1).\]
    Since $\alpha\neq 1$ by assumption, we conclude that
    \[ 1+\mult_{s=-\alpha}b_{f^s}(s)=\mult_{s=-\alpha}b_{(s+1)(s+2)\cdots (s+\alpha-1)f^s}(s)\leq \nu_{1,\alpha}. \qedhere\]   
\end{proof}

\begin{proof}[Proof of Proposition \ref{prop: strict lower bound}]
Recall that $b_f(s)$ in that statement denotes $b_{1}(s)$. Corollary \ref{cor: bound the nilpotency index by multiplicity} gives $n_{\alpha}\geq \nu_{1,\alpha}$, thus the claim follows from Proposition \ref{prop:strictnumu} above.
\end{proof}

\subsection{Pole orders of the Archimedean zeta function}\label{sec: stabilization}
In this subsection, we prove Theorem \ref{thm: stabilization of pole order}.  Let $f$ be a holomorphic function on a complex manifold $X$, and assume $\cM=(\cO_X)_f$. For elements $w_1,w_2\in \iota_{+}\cM$, recall the meromorphic distribution $Z_{w_1, w_2}$ on $X$ from \cite[\S 4.2]{DLY}. If $w_1\in V^{\alpha}\iota_{+}\cM$, then by \cite[(23)]{DLY}:
\begin{equation}\label{eqn: bound of pole order}
\text{the poles of } Z_{w_1,w_2} \text{ are } \leq -\alpha, \quad \text{and the pole order at } s=-\alpha \text{ is } \leq \mu_{w_1, \alpha}.
\end{equation}
For $w=mf^s$, we abbreviate $Z_m\colonequals Z_{w,w}$.  We first establish several preparatory results.

\begin{lemma}\label{lemma: pole order bound by nu}
    Let $m\in \cM$ and $w\in\iota_{+}\cM$. Then $\mathrm{ord}_{s=-\alpha}Z_{w, mf^s}\leq \nu_{m,\alpha}$.  
\end{lemma}

\begin{proof}
    Consider $w'\colonequals p_{m,\alpha}(s)mf^s\in V^{\alpha}\iota_{+}\cM$. Because $p_{m,\alpha}(-\alpha)\neq 0$ by Lemma \ref{lemma: basic property of pfunction}(i), \eqref{eqn: bound of pole order} and $\C[s]$-linearity of $Z_{-,w}$ imply
    \[\mathrm{ord}_{s=-\alpha}Z_{w, mf^s}= \ord Z_{mf^s,w} = \ord p_{m,\alpha}(s)Z_{mf^s,w} = \ord Z_{w',w} \leq \mu_{w',\alpha} = \nu_{m,\alpha}.\qedhere\]
\end{proof}

\begin{prop}\label{prop: pole order via mu}
Let $m \in \cM$ and $w \in V^{\alpha}\iota_{+}\cM$. Then $\mathrm{ord}_{s=-\alpha}Z_{w,mf^{s-k}}=\mu_{w,\alpha}$ for $k\gg 0$.
\end{prop}

\begin{proof}
Note that $\ord Z_{w,mf^s}\leq \ord Z_{w,mf^{s-1}}$, so as $k$ increases, the pole order of $Z_{w,mf^{s-k}}$ at $-\alpha$ forms a non-decreasing sequence bounded above by $\mu_{w,\alpha}$ due to \eqref{eqn: bound of pole order}. By \cite[Theorem 1.8]{DLY}, we can pick $w' \in \iota_+ \cM=\cM[s]f^s$ such that the pole order of $Z_{w,w'}$ at $s=-\alpha$ is exactly $\mu_{w,\alpha}$. Write
    \[w' = \sum_{i=0}^n (s+\alpha)^i m_i f^s, \quad \text{where } m_i \in \cM.\]
If $i\geq 1$, \eqref{eqn: bound of pole order} implies the pole order of $Z_{w,(s+\alpha)^im_if^s}$ at $s=-\alpha$ is $\leq \mu_{w,\alpha}-i$. Thus, the pole order of $Z_{w, m_0 f^s}$ at $-\alpha$ must equal $\mu_{w, \alpha}$. Furthermore, for any $k\geq 0$, we have
\[ \mu_{w,\alpha} = \ord Z_{w, m_0 f^s} \leq \ord Z_{w, m_0 f^{s-k}} \leq \mu_{w,\alpha}. \]
Therefore, $\ord Z_{w, m_0 f^{s-k}} = \mu_{w,\alpha}$. Since $\cO_X$ is simple, Lemma \ref{lemma: relating two elements in iota+M} provides an equation
\[ P(s)\cdot mf^s = b_{m_0}^{(\ell)}(s)\cdot m_0 f^s,\]
for some $P(s)\in \sD_X[s]$ and $\ell \in \Z_{>0}$. Pick $k\gg 0$ large enough so that $-\alpha$ is not a root of $b_{m_0}^{(\ell)}(s-k)$. Then
\[ \mu_{w,\alpha} = \ord Z_{w, m_0 f^{s-k}} = \ord Z_{w,b_{m_0}^{(\ell)}(s-k) m_0f^{s-k}} = \ord Z_{w,P(s-k)\cdot mf^{s-k}}.\]
Using integration by parts, for any test function $\varphi$ we have
\[ Z_{w,P(s-k)\cdot mf^{s-k}}(\varphi;s) = Z_{w,mf^{s-k}}(P^{\ast}(\varphi);s),\]
where $P^{\ast}$ is the adjoint operator. Since $\ord Z_{w,mf^{s-k}}\leq \mu_{w,\alpha}$ by \eqref{eqn: bound of pole order}, we conclude that $\ord Z_{w,mf^{s-k}} = \mu_{w,\alpha}$.
\end{proof}

Next, we can realize $\nu_{m,\alpha}$ as the pole order of a modified zeta function.

\begin{thm}\label{thm:Zfk}
    Let $\alpha \in \Q$ , then for any $k \geq \max\{i \in \Z_{\geq 0} \mid b_m(-\alpha-i)=0\}$ one has
    \[ \mathrm{ord}_{s=-\alpha}Z_{mf^s,mf^{s-k}}=\nu_{m,\alpha},\]
  
\end{thm}

\begin{proof}
Note that $p_{m,\alpha}(-\alpha)\neq 0$ by Lemma \ref{lemma: basic property of pfunction}(i). Applying Proposition \ref{prop: pole order via mu} to $p_{m,\alpha}(s)mf^s$ we can choose $k\gg 0$ such that
\[ \ord Z_{mf^s,mf^{s-k}} = \ord Z_{p_{m,\alpha}(s)mf^s,m f^{s-k}} = \mu_{p_{m,\alpha}(s)mf^s,\alpha} = \nu_{m,\alpha}.\]
We now minimize $k$. For an integer $0<\ell\leq k$, consider the defining equation for $b_m^{(\ell)}(s)$:
\[ Q\cdot mf^{s+\ell}=b_m^{(\ell)}(s)\cdot mf^s,\]
for some $Q\in \sD_X[s]$. Using integration by parts, for a test function $\varphi$ we have 
\begin{equation}\label{eq:toanalyze}
Z_{mf^s,mf^{s-k}}(\varphi;s) = \frac{1}{b_m^{(\ell)}(s-k)} \cdot Z_{mf^s,mf^{s-(k-\ell)}}(Q^{\ast}(\varphi);s),
\end{equation}
for some operator $Q^{\ast}\in \sD_X[s]$. By \eqref{eqn: division of higher b function of m}, provided 
\[ k-\ell \geq \max\{i \in \Z_{\geq 0} \mid b_m(-\alpha-i)=0 \},\]
$-\alpha$ is not a root of $b_m^{(\ell)}(s-k)$. Hence, $\ord Z_{mf^s,mf^{s-k}} = \ord Z_{mf^s,mf^{s-(k-\ell)}}$, proving the desired statement.
\end{proof}

Now, we can prove Theorem \ref{thm: stabilization of pole order}, which follows from the general statement below.

\begin{corollary}\label{cor: bound of Zm achieved}
  Let $m\in \cM$ and $\alpha\in \Q$ such that $b_m(-\alpha)=0$. For any $\ell\geq \max \{i\in \Z_{\geq 0}\mid b_m(-\alpha-i)=0\}$, the pole order of $Z_m$ at $s=-\alpha-\ell$ is $\nu_{m,\alpha+\ell}$. This order also equals $n_{\alpha}$, the nilpotency index of $s+\alpha$ on $\gr^{\alpha}_V\iota_{+}(\cO_X)_f$.
\end{corollary}

\begin{proof}
    Applying Theorem \ref{thm:Zfk} at $s=-\alpha-\ell$ with $k=0$ gives $\mathrm{ord}_{s=-\alpha-\ell}Z_{m}=\nu_{m,\alpha+\ell}$, which equals $n_\alpha$ by Theorem \ref{thm:nilpotency index}.
\end{proof}

In the remainder of this subsection, we prove a better bound of pole orders in terms of the weight level. Recall that $\cM=(\cO_X)_f$ and $\cM f^{-\alpha}$ underlie the structure of (complex) mixed Hodge structures \cite{MHMproject} (see the discussion below \cite[Definition 5.22]{DY26}). Note that Lemma \ref{lemma: pole order bound by nu} implies that $\ord Z_{m}\leq\mu_{m,\alpha}\leq \nu_{m,\alpha}$. Furthermore, Corollary \ref{cor: bound of Zm achieved} implies that if $\alpha\gg 0$, then
\[ \ord Z_{m}=\nu_{m,\alpha}=n_{\alpha}=\omega_{m,\alpha}.\]
We show that the equality $\nu_{m,\alpha}=\omega_{m,\alpha}$ in this case can be explained by the following. 
\begin{prop}\label{prop:w=v}
  Let $m\in \cM$ and $\alpha\in \Q$. Then
  \begin{equation}\label{eqn: improved bound of ord}\mathrm{ord}_{s=-\alpha} Z_m \leq \max\left\{\omega_{m,\alpha},\nu_{m,\alpha}-1\right\},\end{equation}
 In particular, if $\mathrm{ord}_{s=-\alpha} Z_m=\nu_{m, \alpha}$, then $\omega_{m, \alpha}=\nu_{m, \alpha}$.
\end{prop}

\begin{proof}
If $\nu_{m,\alpha}=0$, \eqref{eqn: improved bound of ord} follows from Lemma \ref{lemma: pole order bound by nu}. So we assume $\nu_{m,\alpha}\geq 1$. By Proposition \ref{prop: omega malpha exists}, we can consider an equation computing $\omega_{m, \alpha}$:
    \[ P\cdot (mf^{s+k}) = (s+\alpha)^{\omega_{m, \alpha}}mf^s+(s+\alpha)^{\omega_{m, \alpha}+1}w_1,\]
    where $P\in \sD_X[s]$, $w_1\in \iota_{+}\cM$, and $k\gg 0$. We can choose $k$ large enough that $P\cdot (mf^{s+k})\in V^{>\alpha}\iota_{+}\cM$. Set $w\colonequals mf^s +(s+\alpha)w_1$. Lemma \ref{lem:znilp} then implies that $w$ lies in the filtration $Z_{\omega_{m, \alpha}}V^{\alpha}\iota_{+}\cM$.
    It follows that
    \begin{align*}
        Z_m = Z_{w-(s+\alpha)w_1,mf^s} = Z_{w,mf^s}-(s+\alpha)Z_{w_1,mf^s}.
    \end{align*}
    Because $w\in Z_{\omega_{m, \alpha}}V^{\alpha}\iota_{+}\cM$, we know $\ord Z_{w,mf^s} \leq \mu_{w,\alpha}=\omega_{m, \alpha}$ by \eqref{eqn: bound of pole order}. Meanwhile, Lemma \ref{lemma: pole order bound by nu} ensures that
    \[ \ord (s+\alpha)Z_{w_1,mf^s}=\max\{0,\ord Z_{w_1,mf^s}-1\}\leq \nu_{m,\alpha}-1,\]
    proving \eqref{eqn: improved bound of ord}. The second statement follows easily from
$\omega_{m, \alpha}\leq \nu_{m,\alpha}$ (Proposition \ref{prop: omega malpha exists}).
\end{proof}

As a corollary, we have the following.

\begin{corollary}\label{cor:positive Budur-Walther}
    If $\nu_{m,\alpha-1}<\nu_{m,\alpha}$ and $\underset{s=-\alpha}{\mathrm{ord}} Z_m=\nu_{m,\alpha}$, then $\sD_X\cdot mf^{-\alpha}\neq \sD_X\cdot mf^{-\alpha+1}$.
\end{corollary}

\begin{proof}
 By Proposition \ref{prop:w=v}, we have $\omega_{m, \alpha}=\nu_{m, \alpha}$. By Proposition \ref{prop: omega malpha exists}, $\omega_{m, \alpha-1} \leq \nu_{m, \alpha-1}$. Thus, $\omega_{m, \alpha-1}< \omega_{m, \alpha}$. This implies that $mf^{-\alpha} \notin W_{\dim X+\omega_{m, \alpha-1}}(\cM \cdot f^{-\alpha}) \supseteq \sD_X\cdot mf^{-\alpha+1}$.
\end{proof}

\begin{remark}
The corollary above recovers \cite[Proposition 5.11]{DLY}, which says that if $\alpha\in (0,1)$ with $b_1(-\alpha)=0$ and $f^{-\alpha}\in \sD_X\cdot f^{-\alpha+1}$, then $\ord Z_{f^s}<\mult_{s=-\alpha}b_{1}(s)$, since in this case $\nu_{1, \alpha} = \mult_{s=-\alpha}b_{1}(s)$, and $\nu_{1, \alpha-1}=0$.
\end{remark}

Based on the discussion above, we ask the following.
\begin{question}\label{zeta}
Is it true that $\mathrm{ord}_{s=-\alpha} Z_m\leq \omega_{m, \alpha}\,$?
\end{question}

\subsection{Adjoint ideals and a question of Walther}
Let $f$ be an algebraic function on a smooth complex algebraic variety $X$.

\begin{proof}[Proof of Theorem \ref{thm: afs+1 lies in adjoint}]
    Let $\cM=(\cO_X)_f$. By \cite{Olano} and \cite{BS05}, we know that
    \begin{align*}
    W_{\dim X+1}F_0(\cM\cdot f^{-1})&=\mathrm{adj}(X,D)\cdot f^{-1}.
    \end{align*}
    Therefore, Lemma \ref{lemma:F=P} and Theorem \ref{thm:Bnice} imply that 
    \begin{equation}\label{eqn: characterization of elements in adjoint ideal} 
    g\in \mathrm{adj}(X,D) \Longleftrightarrow gf^s\in V^1\iota_{+}\cM \text{ and } \omega_{g,1}\leq 1.
    \end{equation}

    Let $g\in \mathfrak{a}_{f, s+1}$. By definition, there exists $P\in \sD_X[s]$ such that
    \begin{equation}\label{eqn: def of afs+1} P\cdot f^{s+1} = (s+1)gf^s.\end{equation}
    Set $m=f\in \cM$. Since $b_m(s)=b_1(s+1)$ has no roots in $-1+\Z_{\geq 0}$ \cite{Kas76}, Remark \ref{remark: suffices to choose m'} ensures $\omega_{g,1}\leq 1$.  Next, we show that $gf^s\in V^1\iota_{+}\cM$. Because $b_1(s)$ has only negative roots, Theorem \ref{thm: Sabbah} implies that $f^s \in V^{>0}\iota_{+}\cM$, which forces $P\cdot f^{s+1}\in V^{>1}\iota_{+}\cM$. It follows that $(s+1)gf^s\in V^{>1}\iota_{+}\cM$. By Lemma \ref{lem:easy2}, $b_{g}(s)\mid (s+1)b_{(s+1)g}(s)$. By Theorem \ref{thm: Sabbah}, all roots of $b_{(s+1)g}(s)$ are $< -1$, so all roots of $b_g(s)$ must be $\leq -1$. Hence, $gf^s\in V^1\iota_{+}\cM$. From \eqref{eqn: characterization of elements in adjoint ideal}, we conclude $g\in \mathrm{adj}(X,D)$.

    To establish $\mathfrak{a}_{f, s+1} \subseteq \sqrt{\mathrm{Jac}(f)}$, on the open set $U=\{g\neq 0\}$, we can rewrite \eqref{eqn: def of afs+1} as 
    \[\frac{1}{g} P(s) \cdot f^{s+1} = (s+1) f^s.\]
    This implies the $b$-function of $f_{|U}$ divides $s+1$, meaning $f_{|U}$ is smooth. Consequently, $\mathrm{Sing}(D) \subseteq \{g=0\}$, yielding $g\in\sqrt{\mathrm{Jac}(f)}$.
\end{proof}

In general, neither $\sqrt{\mathrm{Jac}(f)}$ nor $\mathrm{adj}(X,D)$ contains the other. The containment in Theorem \ref{thm: afs+1 lies in adjoint} can be strict:

\begin{example}
    Let $f=x^2+y^3$, then 
    \[ \Jac(f)=\mathfrak{a}_{f,s+1}=(x,y^2)\subsetneq (x,y)=\mathrm{adj}(X,D)=\sqrt{\Jac(f)}.\]
The equation $(x,y)=\mathrm{adj}(X,D)$ follows from a computation (e.g. in Macaulay2). Lemma \ref{lem:redwalther} below implies that  $y \in \mathfrak{a}_{f,(s+1)(s+7/6)}$, since $b_{f^s}(s)=(s+1)(s+5/6)(s+7/6)$ and $b_{yf^s}(s) = (s+1)(s+7/6)(s+11/6)$, e.g. again by a computation in Macaulay2. Note that the equality $\mathfrak{a}_{f,s+1}=(x,y^2)$ is equivalent to proving that $y\notin \mathfrak{a}_{f,s+1}$. To do this, suppose by contradiction that $y \in \mathfrak{a}_{f,s+1}$, which would imply the existence of an equation
\[P(s) \cdot f^{s+1} = (s+1) \cdot y f^s, \quad P(s) \in \sD[s].\]
Evaluating at $s=-7/6$, we obtain that $y \cdot f^{-7/6} \in \sD \cdot f^{-1/6}$. Since $\nu_{f^s,-1/6} = 0$, we have $\omega_{f^s, -1/6}=0$, which implies $\omega_{yf^s, -7/6}=0$. But by Proposition \ref{prop: isoquasi} below we must have $\omega_{yf^s, -7/6} = \nu_{yf^s, -7/6}=1$ (this can be checked also with Algorithm \ref{algorithm}), a contradiction.
\end{example}

Clearly, $\mathfrak{a}_{f, b_{f^s}(s)} = \cO_X$. As a consequence of Lemma \ref{lemma: relating two elements in iota+M}, we can say more about the ideals $\mathfrak{a}_{f, q(s)}$ studied by Walther for more general $q(s) \in \C[s]$:
\begin{lemma}\label{lem:redwalther} For $g \in \cO_X$, we have $g\in \mathfrak{a}_{f, \, \gcd(b_{f^s}(s), \, b^{(k)}_{gf^s}(s))}, \mbox{ for } k\gg 0$.
\end{lemma}

The next result characterizes the adjoint ideal in terms of $b$-functions.

\begin{prop}\label{prop:adj}
    We have
    \[\mathrm{adj}(X,D) = \{g\in \cO_X \mid \text{every root of $ b_{gf^s}(s)$ is $\leq -1$} \text{ and } \mult_{s=-1}b_{gf^s}(s)\leq 1\}.\]
    More generally, for any $k\geq 1$,
    \[W_kI_0(D) = \{g\in \cO_X \mid  \text{every root of $ b_{gf^s}(s)$ is $\leq -1$} \text{ and } \mult_{s=-1}b_{gf^s}(s)\leq k\}. \]
\end{prop}

\begin{proof}
    We retain the notation from \S \ref{sec: stabilization}. The proof of \cite[Theorem 1.6]{DLY} shows that $\mathrm{ord}_{s=-1}Z_{g}=\mult_{s=-1}b_{gf^s}(s)$. Because $\mult_{s=-1}b_{gf^s}(s)=\nu_{g,1}$ by Theorem \ref{thm: numalpha via higher order pfunction}, Proposition \ref{prop:w=v} applies, yielding $\omega_{g,1}=\nu_{g,1}=\mult_{s=-1}b_{gf^s}(s)$. The result then follows from \eqref{eqn: characterization of elements in adjoint ideal} and its generalization to $W_kI_0(D)$.
\end{proof}

\begin{remark}
    Because $f$ has rational singularities if and only if $1 \in \mathrm{adj}(X,D)$, Proposition \ref{prop:adj} immediately implies that $f$ has rational singularities if and only if $b_1(s)/(s+1)$ has no roots $\geq -1$, recovering Saito's celebrated result \cite{Saitorational}.
\end{remark}

\subsection{The intersection complex of a hypersurface}\label{sec: intersection complex of hypersurface}
We prove Theorem \ref{thm: torelli}. Let $X$ be a smooth complex algebraic variety, with $\cM=(\cO_X)_f$ and $D=\mathrm{div}(f)$. Recall from Definition \ref{def: N filtration} that $\cN_{\ell}\cM=\{ m\in \cM \mid \nu_{m,0}\leq {\ell}\}$.

\begin{prop}\label{prop: IC is weight 1}
    Fix a local coordinate $x$ of $X$ such that $m\colonequals \frac{\partial f}{\partial x} \cdot f^{-1}\in \cM\setminus \cO_X$. Then $\nu_{m, 0} = 1$ and $p_{m,0}(s)=1$. If $f$ is irreducible, then
    \[ W_{\dim X+1} \cM = \sD_X \cdot m = \sD_X \cdot \cN_1\cM \quad \text{and} \quad W_{\dim X+1} \cM /\cO_X \cong \mathrm{IC}_{D}.\]
\end{prop}

\begin{proof}
    By Kashiwara's theorem \cite{Kas76}, we have that $f^s \in V^{>0}\iota_+\cM = Z_0 V^0\iota_+ \cM$. Consider the equation
    \[\partial_x \cdot f^s = s \cdot m \cdot f^s.\]
    Since $\partial_x \cdot f^s \in Z_0 V^0\iota_+ \cM$, we get by Lemma \ref{lem:znilp} that $m f^s \in Z_1 V^0\iota_+ \cM$, i.e. $\nu_{m, 0} \leq 1$. Since $m \notin \cO_X$, by Remark \ref{remark: N0=W0 for OX} we must have $\nu_{m, 0}=1$ and $p_{m,0}(s)=1$.
    
    By Proposition \ref{prop:restrnu} we have $m \in W_{\dim X+1}\cM \setminus \cO_X$. Since $\mathrm{IC}_{D}$ is the unique simple submodule of $\cM/\cO_X$, the semisimplicity of $W_{\dim X+1} \cM /\cO_X$ forces it to be $\mathrm{IC}_{D}$. Because $\cO_X$ is the unique simple submodule of $W_{\dim X+1}\cM$, we get
    \[W_{\dim X+1} \cM = \sD_X \cdot m = \sD_X \cdot \cN_1\cM.\qedhere\]
\end{proof}

\begin{remark}
    Because $\omega_{m, 0} \neq 0$, the equation  $\partial_x \cdot f^s = s \cdot m f^s$ immediately implies $\omega_{m, 0} = 1$ by Remark \ref{remark: suffices to choose m'}. However, Proposition \ref{prop: IC is weight 1} establishes the stronger bound $\nu_{m,0}=1$, aided by Proposition \ref{prop: omega malpha exists}. 
\end{remark}

\begin{proof}[Proof of Theorem \ref{thm: torelli}] 
    Using \eqref{eqn: nu11 is mult}, we have $\mult_{s=-1}b_1(s)=\nu_{f^{-1},0}$. On the other hand, since $f^{-1}\notin \cO_X$, by definition $[f^{-1}] \in \mathrm{IC}_D=W_{\dim X+ 1}\cM/W_{\dim X}\cM$ if and only if $\omega_{f^{-1},0}=1$. So the equivalence \eqref{eqn: Torrelli equivalence} is equivalent to
    \[ \nu_{f^{-1},0}=1 \Longleftrightarrow \omega_{f^{-1},0}=1.\]
Since  $1\leq \omega_{f^{-1},0}\leq \nu_{f^{-1},0}$ by Proposition \ref{prop: omega malpha exists}, this immediately proves the statement (a): if $\nu_{f^{-1},0}=1$, then we must have $\omega_{f^{-1},0}=1$.

Now let us turn to the statement (b): suppose $\nu_{f^{-1},0}>1$. We consider two cases:

    \textbf{Case 1}: Suppose we have $\mathrm{ord}_{s=-1} Z_1=\mult_{s=-1}b_1(s)$ (noting $Z_f,b_f$ in statement (b) are $Z_1,b_1$). It follows that 
       \[ \mathrm{ord}_{s=0} Z_{f^{-1}} = \mathrm{ord}_{s=-1} Z_1 = \mult_{s=-1}b_1(s)=\nu_{f^{-1},0}. \]
    Applying Proposition \ref{prop:w=v} to $m=f^{-1}$ and $\alpha=0$ yields $\omega_{f^{-1},0}=\nu_{f^{-1},0}>1$. 

    \textbf{Case 2}: Assume $n_1=\mult_{s=-1}b_1(s)$, which means $n_0=\nu_{f^{-1},0}$. Since $\nu_{f^{-1},0}=\mult_{s=-1}b_1(s)>1$ by assumption, Proposition \ref{prop:lastweight} gives
    \[ W_{\dim X+1}\cM \subseteq W_{\dim X+n_0-1}\cM = \cN_{n_0-1}\cM \subsetneq \cN_{n_0}\cM.\]
    Because $n_0=\nu_{f^{-1},0}$, we know $f^{-1}\notin \cN_{n_0-1}\cM$ by definition. Therefore, $f^{-1}\notin W_{\dim X+1}\cM$.

Finally, let us turn to the statement (c). Note that taking $\cS_1 = \cO_X$ and $\cS_2 = \cO_Y$ the assumption \eqref{eqn: support condition} is satisfied. The assumptions $\mult_{s=-c}b_f(s)=1=\mult_{s=-(1-c)}b_g(s)$ give $\nu_{f^s,c}=1=\nu_{g^s,1-c}$ by Theorem \ref{thm: numalpha via higher order pfunction}. By Theorem \ref{thm:thomsebequality}(a), we get $\tilde{\nu}_{h^s, 1} =1$, so that $\mult_{s=-1} b_h(s) =2$  (note that if $g$ is Euler-homogeneous this also follows from \cite[Theorem 0.8]{Saito94} or Theorem \ref{thm:thomsebequality}(c)  -- see also the discussion after \cite[Theorem 1.4]{DM24bfunction} for the global $b$-functions). On the other hand, note that the assumption $f^{-c}\in \sD_Xf^{1-c}$ and $c\in (0,1)$ implies that $f^{-c}\in W_{\dim X}(\cO_X)_f\cdot f^{-c})$ (see Example \ref{ex: coex Torrelli}), hence $\omega_{1,c}=0$. So we can apply Proposition \ref{prop:thomsebweight} to $m_1=1\in \cO_X$ and $m_2=1\in \cO_Y$ to conclude that $\widetilde{\omega}_{1\boxtimes 1, 1} = 0$. By \eqref{eqn: modified omega}, this gives that $\omega_{1\boxtimes 1, 1} \leq 1$ (hence it is equal to $1$), thus $h^{-1}\in W_{\dim X+\dim Y}(\cO_{X\times Y})_h$, i.e., $[h^{-1}] \in \mathrm{IC}_{h=0}$.
\end{proof}

\begin{remark}\label{remark: strengthening of Torelli}
Under the assumptions in Theorem \ref{thm: torelli}(c) and that $g$ is quasi-homogeneous, the conclusion $f^{-c}\in \sD_X\cdot f^{1-c}\Rightarrow [h^{-1}]\in \IC_{h=0}$ can be strengthened to
\begin{equation*}\label{eqn: strengthened} f^{-c}\in \sD_X\cdot f^{1-c} \text{ or } g^{-(1-c)}\in \sD_Y\cdot g^{c}\Longleftrightarrow [h^{-1}]\in \IC_{h=0}.\end{equation*}
We leave the proof for the interested readers.
\end{remark}

\begin{example}\label{ex: coex Torrelli}
Recall that the existence of $f$ with $b_f(-c)=0$ and $f^{-c}\in \sD_X\cdot f^{1-c}$ is a question of Budur-Walther, where the first such example is found by Saito \cite{saitopower}. As explained in Remark \ref{remark: produce BW examples}, one can produce such examples systematically: it suffices to find $f$ with $c\in (0,1),b_f(-c)=0$ and $\omega_{1,c}=0$, where the latter can be tested using the algorithm in \S \ref{sec: algorithm}. Here is a particularly simple example from \cite[Remark 6.9]{DY26}: $X=\C^3$, $f=x_1 x_2(x_1-x_2)x_3^2(x_1-x_3)^4$ and $c=1/3$, then $f^{-c}\in \sD_X\cdot f^{1-c}$. Set $g(y)=y_1^3+y_2^3$. Using Macaulay2 we obtain
\[ b_f(s)=(s+\frac{1}{3})(s+\frac{2}{3})(s+\frac{4}{3})\prod_{i=1}^4(s+\frac{i}{4})\prod_{i=2}^8(s+\frac{i}{7})\prod_{i=4}^{11}(s+\frac{i}{9}),\]
and similarly (or by the Thom-Sebastiani formula)
\[ b_g(s)=(s+1)^2(s+\frac{2}{3})(s+\frac{4}{3}).\]
Then $h(x,y):=f(x)+g(y)$ and $c=1/3$ satisfy the assumptions of Theorem \ref{thm: torelli}(c). In other words, the implication \eqref{eqn: Torrelli equivalence} does not hold for $h$.
\end{example}
\begin{remark}\label{remark: recovering Torrelli}
    By Kashiwara's well-known result \cite[Proposition 4.2]{Torrelli}, we have $\sD_X\cdot f^{-1}=(\cO_X)_f$ if and only if $-1$  is the only integer root of $b_{1}(s)$. Using Theorem \ref{thm:nilpotency index} and \eqref{eqn: nu11 is mult}, the latter condition forces $n_1=\nu_{1,1}=\mult_{s=-1}b_1(s)$, meaning the $``\Leftarrow"$ part of \eqref{eqn: Torrelli equivalence} holds by Theorem \ref{thm: torelli}. It is worth emphasizing that the condition $n_1=\mult_{s=-1}b_1(s)$ is strictly weaker: for example, the function $f=x^3-z^3-x\cdot y^2$ satisfies $n_1=2=\mult_{s=-1}b_1(s)$, but $-1$ is not its only integer root.
\end{remark}

\begin{corollary}\label{cor: Torrelli Thm 1.2}
   Assume $f$ is irreducible, then $\mathrm{IC}_D=(\cO_X)_f/\cO_X$ is equivalent to $b_{f^s}(s)/(s+1)$ having no integer roots.
\end{corollary} 

\begin{proof}
    By Proposition \ref{prop: IC is weight 1}, $\mathrm{IC}_D=(\cO_X)_f/\cO_X$ if and only if $W_{\dim X+1}(\cO_X)_f=(\cO_X)_f$. By Proposition \ref{prop:lastweight}, this is equivalent to $n_0=1$. 
    
    If $n_0=1$, Proposition \ref{prop: strict lower bound} implies 
    \[ 1 = n_0 = n_{\ell} > \mult_{s=-\ell}b_{f^s}(s) \quad \text{for all } \ell\in \Z_{\geq 2}.\]
    Thus, $b_{f^s}(s)/(s+1)$ cannot have integer roots.

    Conversely, suppose $b_{f^s}(s)/(s+1)$ has no integer roots. If $\ell\geq 2$, Proposition \ref{prop:nurecursivebound} gives
    \[ \nu_{1,\ell-1} \leq \nu_{1,\ell} \leq \nu_{1,\ell-1} + \mult_{s=-\ell}b_{f^s}(s) = \nu_{1,\ell-1}.\]
    Hence, $\nu_{1,\ell}=\nu_{1,\ell-1}=\cdots=\nu_{1,1}=\mult_{s=-1}b_{f^s}(s)=1$ by \eqref{eqn: nu11 is mult}. Theorem \ref{thm:nilpotency index} then yields $n_0 = \lim_{\ell\to \infty} \nu_{1,\ell}=1$.
\end{proof}

\subsection{Examples}\label{sec: examples}

In this section, we compute several classes of examples explicitly and answer the corresponding questions raised throughout the text. Throughout the section, for simplicity we assume $X$ is a smooth algebraic variety and $f$ a regular function on $X$. 

\subsubsection{Isolated quasihomogeneous singularities}

Denote by $b_f(s)$ the usual Bernstein-Sato polynomial of $f$.  Recall that $\tilde{b}_f(s) = b_f(s)/(s+1)$ denotes the reduced Bernstein-Sato polynomial of $f$. 

\begin{prop}\label{prop: isoquasi}
    Suppose $f$ has an isolated quasihomogeneous singularity. Then for any $\alpha \in \Q$ we have $\cN_{\ell}((\cO_X)_f \cdot f^{-\alpha})=W_{\dim X+\ell}((\cO_X)_f \cdot f^{-\alpha})$. Furthermore, if $\tilde{b}_f(s)$ has no integer roots then $n_0=1$, otherwise $n_0=2$, and $n_{\alpha}= 1$ if $\alpha\not\in \Z$ and $b_f(-\alpha)=0$.
\end{prop}

\begin{proof}
We may assume the only singularity of $f$ is $\{0\}$. Since $f$ is quasihomogeneous, the monodromy on the Milnor fiber $F$ at $0$ has finite order \cite[p.~19]{Dimca90} so the corresponding monodromy on the cohomology is semisimple. Set $\cM=(\cO_X)_f$. 
     
    Suppose first that $\alpha\not\in \Z$ with $b_f(-\alpha)=0$. Since the singularity of $f$ is isolated, $n_\alpha$ is the maximal size of a Jordan block of the monodromy on the $e^{-2\pi i\alpha}$-eigenspace of $H^{n-1}(F,\C)$. Hence $n_\alpha\leq1$ by the semisimplicity. In particular, by \eqref{eqn: upper and lower bound of nalpha} we have $n_{\alpha}=\mult_{s=-\alpha}b_f(s)=1$.
The equality $\cN_{\ell}(\cM \cdot f^{-\alpha}) = W_{\dim X+\ell}(\cM \cdot f^{-\alpha})$ then follows from Proposition \ref{prop:lastweight}.

Next, suppose $\alpha=0$. Again because $f$ has an isolated singularity at $\{0\}$, the unipotent vanishing cycles $\gr_V^0\iota_+\cO_X$ support at $\{0\}$, where they are identified with the eigenspace $H^{n-1}(F,\C)_1$. By above, since the monodromy on $H^{n-1}(F,\C)_1$ is semisimple, the logarithm of its unipotent part, which can be identified with the action by $s$ on $\gr_V^0\iota_+\cO_X$, is zero. In other words, $s\cdot\gr_V^0\iota_+\cO_X=0$. Since $s=-\partial_t t$, we have 
\[
(s+1)^2=(t\partial_t)^2=t(\partial_t t)\partial_t=-ts\partial_t=0 \quad \textrm{on $\gr_V^1\iota_+\cO_X$}.\]
Since $V^1\iota_+\cM=V^1\iota_+\cO_X$ and $s+1=-t\partial_t$, it follows that 
\[ (s+1)^2\gr_V^1\iota_+\cM=0.\]
Hence $n_0=n_1\leq2$. Since $n_0\geq1$, we conclude that $n_0=1$ or $2$. Thus, the equality $\cN_{\ell}(\cM) = W_{\dim X+\ell}(\cM)$ follows from Proposition \ref{prop:lastweight} and Remark \ref{remark: N0=W0 for OX}.

Finally, there are two possibilities for integer roots. If $-1$ is the only integer root of $b_f(s)$ with multiplicity $\leq 2$, then Corollary \ref{corollary: nilpotency index for function} gives $n_0=\mult_{s=-1}b_f(s)$. Otherwise, suppose $b_f(s)$ has another integer root, say the largest such being $-k$ with $k \in \Z_{\geq 2}$. By Proposition \ref{prop:nurecursivebound} we have $\mult_{s=-1}b_f(s) = \nu_{1,1}= \nu_{1,2} = \dots = \nu_{1,k-1} \leq \nu_{1,k}$. On the other hand, Proposition \ref{prop:strictnumu} implies that $\nu_{1,k}\geq 2$.  Since $\nu_{1,k} \leq n_k = n_0$ by Corollary \ref{cor: bound the nilpotency index by multiplicity}, we get $n_0=\nu_{1,k} =2$.
\end{proof}

Note that in the proof above we implicitly obtained $\nu_{1,\alpha}$ for all $\alpha \in \Q$, given the well-known formula for $b_f(s)$ (e.g. see \cite[Theorem 6.18]{kashibook}). Furthermore, Proposition \ref{prop: isoquasi} implies that $\nu_{1,\alpha}=\omega_{1,\alpha}$. We summarize this below.

\begin{corollary}\label{cor: weight level}
    With $f$ as above, we have
    \[\omega_{1,\alpha} = \begin{cases}
        0, \mbox{ if } b_f(s) \mbox{ has no roots in } -\alpha+\Z_{\geq 0}; \\
        1,  \mbox{ if either } \alpha\in\Z_{\geq 1} \mbox{ and } \tilde{b}_f(s) \mbox{ has no roots in } -\alpha+\Z_{\geq 0}, \\
        \quad \mbox{ or } \alpha \notin \Z \mbox{ and } b_f(s) \mbox{ has roots in } -\alpha+\Z_{\geq 0}; \\
        2,  \mbox{ if } \alpha\in \Z_{\geq 1} \mbox{ and } \tilde{b}_f(s) \mbox{ has a root in } -\alpha+\Z_{\geq 0}. \\
    \end{cases}\]
\end{corollary}

\begin{remark}
     Additionally, using the computations of $\nu$ we can write down the $p$-functions $p_{1,\alpha}(s)$ for all $\alpha\in \Q$ as well as formulas for the Bernstein-Sato polynomials of powers of $f$.
\end{remark}

   \begin{corollary}\label{cor:quasihomogeneous-zeta-poles}
With $f$ as above, for every $\alpha\in\mathbb Q$, then we have 
\[ \operatorname{ord}_{s=-\alpha} Z_1=\omega_{1,\alpha}=\nu_{1,\alpha}.\]
\end{corollary}

\begin{proof}
By the discussion above, we have $\omega_{1,\alpha}=\nu_{1,\alpha}$. Moreover, Lemma~\ref{lemma: pole order bound by nu} gives 
\begin{equation}\label{eqn: pole order bound proof}
\operatorname{ord}_{s=-\alpha} Z_1 \leq \nu_{1,\alpha}=\omega_{1,\alpha}.\end{equation}
It remains to prove the reverse inequality. By \cite[Th\'eor\`eme~1.9]{Loeser}, every negative integer is a pole of $Z_1$, and every nonnegative
integer shift of a root of $b_f(s)$ is a pole. Moreover, whenever $\alpha\in\mathbb Z_{\ge1}$ and $\widetilde b_f(s)$ has a root in $-\alpha+\mathbb Z_{\ge0}$, the pole at $s=-\alpha$ has order at least two.

If $\omega_{1,\alpha}=0$, the bound \eqref{eqn: pole order bound proof} implies $\operatorname{ord}_{s=-\alpha}Z_1=0$.
If $\omega_{1,\alpha}=1$, Corollary~\ref{cor: weight level} implies that either $\alpha\in\mathbb Z_{\ge1}$, or $\alpha\notin\mathbb Z$ and $b_f(s)$ has a root in $-\alpha+\mathbb Z_{\ge0}$. In both cases, Loeser's theorem gives
$\operatorname{ord}_{s=-\alpha}Z_1\ge1$. Finally, if $\omega_{1,\alpha}=2$, Corollary~\ref{cor: weight level} gives
$\alpha\in\mathbb Z_{\ge1}$ and a root of $\widetilde b_f(s)$ in $-\alpha+\mathbb Z_{\ge0}$; hence Loeser's theorem gives
$\operatorname{ord}_{s=-\alpha}Z_1\ge2$. This proves the desired equality. 
\end{proof}

\subsubsection{Hyperplane arrangements}
Let $f=\prod_{i=1}^d f_i$ with $f_i$ are linear defining a hyperplane arrangement $\cA$ in $X=\C^n$. Write $H_i=\{f_i=0\}$. Because the statement below is local, we can assume that $\cA$ is central, i.e. $0\in H_i$ for each $i$. Further, without loss of generality we can also assume that $\cA$ is essential, i.e. $\bigcap_{i=1}^d H_i=\{0\}$. For $I\subseteq\{1,\ldots,d\}$, set $f_I=\prod_{i\in I}f_i$ and $L_I=\bigcap_{i\in I}H_i$, with $f_\varnothing=1$ and $L_\varnothing=X$. We say that $I$ is \emph{independent} if $\{f_i\}_{i\in I}$ are linearly independent as linear forms. 

\begin{prop}\label{prop:wfil arbitrary arrange}
Let $\cM=(\cO_X)_f$, then $n_0 = n$ and
\[ W_{n+r}\cM=\sum_{\substack{|I|=r,\, I\ {\rm independent}}}(\cO_X)_{f_I}=\sum_{\substack{|I|=r,\, I\ {\rm independent}}}\sD_X\cdot(1/f_I), \quad \textrm{if $0\leq r\leq n$}.\]
\end{prop}

\begin{proof}
For an independent set $I$, set $\cM_I=(\cO_X)_{f_I}$, where $f_I$ is a simple normal crossing divisor. So $\cM_I=\sD_X\cdot(1/f_I)$ and its weight levels (as mixed Hodge modules) are $\leq n+|I|$. The natural morphism $\cM_I\hookrightarrow\cM$ underlies a morphism of mixed Hodge modules, and strictness of the weight filtration gives $\cM_I\subseteq W_{n+r}\cM$. For $0\leq r\leq n$, set $\cM_r=\sum_{|I|=r,\ I\ {\rm independent}}\cM_I$, which is contained in $W_{n+r}\cM$ by above. Since $\cA$ is essential, every independent subset of size $r-1<n$ extends to an independent subset of size $r$. It follows that $\cM_{r-1}\subseteq\cM_r$.

We prove by induction on $r$ that $\cM_r=W_{n+r}\cM$. For $r=0$, we have $\cM_0=\cO_X=W_n\cM$. Assume that $\cM_{r-1}=W_{n+r-1}\cM$. We claim that the natural map
\begin{equation}\label{eqn: surjectivity arbitrary arrangement} \cM_r\twoheadrightarrow\gr^W_{n+r}\cM \end{equation}
is surjective. Granting this, note that by induction hypothesis the kernel of this map is 
\[ \cM_r\cap W_{n+r-1}\cM=\cM_r\cap\cM_{r-1}=\cM_{r-1}.\]
Therefore $\cM_r/\cM_{r-1}\simeq\gr^W_{n+r}\cM$. Since $\cM_r\subseteq W_{n+r}\cM$, this implies $\cM_r=W_{n+r}\cM$.

It remains to establish \eqref{eqn: surjectivity arbitrary arrangement} for each $r$. To do so, let us recall a known formula for $\gr^W_{n+r}\cM$. Let $j:U=X\setminus\{f=0\}\hookrightarrow X$. By \cite[(1.7.1)]{BS10}\footnote{Although this is stated for the projectivized arrangement, the same incidence complex can be defined directly on $X=\C^n$ by replacing each \(\mathbf P(L)\) with \(L\). The stalk computation in \cite[Lemma 1.8]{BS10} is local along the arrangement stratification and therefore gives a quasi-isomorphism $j_!\mathbf Q_U\simeq \cK_X^\bullet$ in the affine setting as well. Then we obtain the desired statement by lifting to mixed Hodge modules and dualizing, following \cite{BS10}.}we have
\[ \gr^W_{n+r}j_\ast\mathbf Q_U^H[n]\simeq\bigoplus_{\substack{L\in L(\cA),\, \mrm{codim}\,L=r}}\left((i_L)_\ast\mathbf Q_L^H[n-r](-r)\right)^{\oplus|\mu(L)|}, \]
where $L(\cA)$ is the intersection lattice of $\cA$ (i.e. the collection of all edges of $\cA$, which are of the form $L_I$ for some $I\subseteq \{1,\ldots,d\}$), $i_L:L\hookrightarrow X$ is the inclusion, and $\mu(L)$ is the M\"obius function of $L(\cA)$. Passing to the underlying $\sD_X$-modules gives
\begin{equation}\label{eqn: decomposition grW hyperarrangement}
    \gr^W_{n+r}(\cO_X)_f\simeq\bigoplus_{\substack{L\in L(\cA),\, \mrm{codim}\,L=r}}\left((i_L)_+\cO_L\right)^{\oplus|\mu(L)|}.
\end{equation}  

Next, we recall how to understand the multiplicities $|\mu(L)|$ of summands in \eqref{eqn: decomposition grW hyperarrangement}. Put $\omega_i=df_i/f_i$. By the Brieskorn--Orlik--Solomon theorem (see, for example, \cite[Theorem 2.1]{Dupont}), $H^\bullet(U,\C)$ is generated as an algebra by the classes $[\omega_i]\in H^1(U,\C)$, and its degree-$r$ part is spanned by the classes $[\omega_I]=[\omega_{i_1}\wedge\cdots\wedge\omega_{i_r}]$ where  $I=\{i_1,\ldots,i_r\}$ is an independent set. By \cite[\S1.3]{BDS}\footnote{The same argument of \cite[(1.3.2)--(1.3.3)]{BDS} applies in the affine setting.}, for an edge $L$ of codimension $r$, the multiplicity space attached to the  summand $(i_L)_+\cO_L$ in \eqref{eqn: decomposition grW hyperarrangement} is spanned by the classes $[\omega_I]$ with $I$ independent, $|I|=r$, and $L_I=L$; note that its dimension is exactly $|\mu(L)|$. Moreover, if $\cA_I$ is the subarrangement indexed by $I$, then the natural map $X\setminus \cA\to X\setminus \cA_I$ induces on the codimension-$r$ multiplicity spaces the map sending the class $[\omega_I]$ for $\cA_I$ to the class $[\omega_I]$ for $\cA$.

Now we can prove \eqref{eqn: surjectivity arbitrary arrangement}. Fix an edge $L$ of codimension $r$. For every independent set $I$ with $|I|=r$ and $L_I=L$, as $f_I$ is simple normal crossing, we have $\gr^W_{n+r}\cM_I=(i_L)_+\cO_L$. Consider the map
$\gr^W_{n+r}\cM_I\to \gr^W_{n+r}\cM$ induced by the natural map $\cM_I\hookrightarrow \cM$. By the previous paragraph, the class $[\omega_I]$ for the subarrangement $\cA_I$ is sent to the class $[\omega_I]$ for $\cA$ in the multiplicity space of $(i_L)_+\cO_L$. By the description of the multiplicity space of $(i_L)_+\cO_L$, this implies that  \eqref{eqn: surjectivity arbitrary arrangement} holds. Therefore we finish the proof that $\cM_r=W_{n+r}\cM$. It follows from Proposition \ref{prop:lastweight} that $n_0=n$.
\end{proof}

\begin{prop}\label{prop: weight equals N arbitrary arrangement}
We have $W_{n+r}(\cO_X)_f=\sD_X\cdot\cN_r(\cO_X)_f$ for every $0\leq r\leq n$.
\end{prop}
\begin{remark}
Since $\omega_{m,0}\leq \nu_{m,0}$, Propositions \ref{prop:wfil arbitrary arrange} and \ref{prop: weight equals N arbitrary arrangement} imply that if $m=1/f_I$ with an independent set $I$, then $\nu_{m,0}=|I|$.
\end{remark}
\begin{proof}
Set $\cM=(\cO_X)_f$. By Proposition \ref{prop:restrnu}, we already have $\sD_X\cdot\cN_r\cM\subseteq W_{n+r}\cM$. 

By Proposition \ref{prop:wfil arbitrary arrange}, it suffices to prove that $1/f_I\in\cN_r(\cO_X)_f$ for every independent set $I$ of size $r$. Fix such an $I$ and put $m=1/f_I$. It suffices to prove $\nu_{m,0}\leq r$. Set $g=f/f_I$. Then $mf^s=gf^{s-1}$ and so $b_m(s)=b_{gf^s}(s-1)$. By \cite{Kas76}, every root of $b_{f^s}(s)$ is negative, so $f^s\in V^{>0}\cM$ by Theorem \ref{thm: Sabbah}. It follows that $gf^s\in V^{>0}\cM$, hence by Theorem \ref{thm: Sabbah} again every root of $b_{gf^s}(s)$ is negative. Therefore $b_m(k)\neq0$ for every integer $k\geq1$ and Theorem \ref{thm: numalpha via higher order pfunction} gives $\nu_{m,0}=\mult_{s=0}b_m(s)=\mult_{s=-1}b_{gf^s}(s)$. 

The idea is to use Lemma \ref{lemma: generalized DM multiplicity} below to bound $\mult_{s=-1}b_{gf^s}(s)$. Let $\pi:Y\to X$ be the iterated blowup of $X$ along all proper edges of $\cA$ from dimensions $0$ to $n-1$, which is a log resolution of $(X,\mathrm{div}(f))$ and is an isomorphism over $X\setminus \{f=0\}$; this is the wonderful model associated to the \emph{maximal} building set of $\cA$, see \cite{DP95}. The boundary divisors $E_L$  of $\pi$ are indexed by proper edges $L$ of $\cA$. For an edge $L$, set $c_L=\mrm{codim}_X L$, $d_L=\#\{1\leq j\leq d\mid L\subseteq H_j\}$, and $e_I(L)=\#\{i\in I\mid L\subseteq H_i\}$. The  numerical data of $\pi$ are $\mrm{ord}_{E_L}(f\circ\pi)=d_L$ and $\mrm{ord}_{E_L}(K_{Y/X})=c_L-1$. Since $g=f/f_I$, we also have $\mrm{ord}_{E_L}(g\circ\pi)=d_L-e_I(L)$.

For the divisor $E_L$, the inequality in \eqref{eqn:mult-1bound} becomes $d_L\geq(c_L-1)+(d_L-e_I(L))+1$, equivalently $e_I(L)\geq c_L$. Since $I$ is independent, clearly $e_I(L)\leq c_L$. Thus the divisor $E_L$ contributes to the bound in \eqref{eqn:mult-1bound} if and only if $e_I(L)=c_L$. Let $L$ be an edge such that $e_I(L)=c_L$, then the independence of $I$ gives $L=\bigcap_{i\in I, L\subseteq H_i}H_i$. In particular, every edge $L$ contributing to the bound \eqref{eqn:mult-1bound} is an intersection of some hyperplanes in $I$, and its codimension belongs to $\{1,\ldots,r\}$. Finally, if $E_{L_1},\ldots,E_{L_p}$ meet at a point, then by \cite[\S 3.2, Theorem (2)]{DP95}, $\{(\C^n/L_1)^*,\ldots,(\C^n/L_p)^*\}$ is a nested set of the maximal building set of $\cA$. But then they must form a chain \cite[\S 2.4, p.~472]{DP95}, say $L_1\supsetneq \cdots\supsetneq L_p$. Hence $1\leq c_{L_1}< \cdots<c_{L_p}\leq r$, so $p\leq r$. So the bound \eqref{eqn:mult-1bound} now gives $\mult_{s=-1}b_{gf^s}(s)\leq r$, and hence $\nu_{m,0}\leq r$, proving the proposition.
\end{proof}

\begin{lemma}\label{lemma: generalized DM multiplicity}
Let $X$ be a smooth algebraic variety, let $f,g\in\cO_X(X)$ be nonzero with $g\mid f$, and put $u=gf^s$. Let $\pi:Y\to X$ be a log resolution of $(X,\mrm{div}(f))$ isomorphic over $X\setminus\{f=0\}$. Let $\{E_i\}_{i\in I}$ be the irreducible components of $\pi^{-1}\mrm{div}(f)$ and write $a_i=\mrm{ord}_{E_i}(f\circ\pi)$, $b_i=\mrm{ord}_{E_i}(g\circ\pi)$, and $k_i=\mrm{ord}_{E_i}(K_{Y/X})$. Then
\begin{equation}\label{eqn:mult-1bound} \mult_{s=-1}b_u(s)\leq\max_{y\in Y}\#\{i\in I\mid y\in E_i,\ a_i\geq k_i+b_i+1\}. \end{equation}
\end{lemma}

\begin{proof}
This is extracted from the proof of \cite[Theorem 1.2]{DirksMustata}. 
The assertion is local on $X$, so we assume $X$ is affine and $\omega_X$ is trivial. We may also assume that the critical locus of  $f$ is contained in $\{f=0\}$, so that the condition (7) of loc. cit holds for both $f$ and $f\circ \pi$, as required in the proof of \cite[Theorem 1.2]{DirksMustata}. Note that $u$ is a section of $\tilde{B}_f=(\cO_X)_f[s]f^s$, let $u^\ast$ denote the corresponding section of the right $\sD_X\langle s,t\rangle$-module $\omega_X\otimes_{\cO_X}\tilde{B}_f$. By \cite[Example 3.3]{DirksMustata}, $b_u(s)=b_{u^\ast}(-s-1)$.

The idea is to bound $\mult_{s=0}b_{u^{\ast}}(s)$ by the map $\pi$. Fix a nowhere-vanishing local volume form $dx$ on $X$. Let us cover $Y$ by affine opens $V$ where for each $V$ there exists a point $y_V\in V$ such that if $E_i\cap V\neq\varnothing$, then $E_i$ contains $y_V$. On each $V$, write $f\circ\pi=p_1\prod y_i^{a_i}$, $\pi^\ast(dx)=p_2\prod y_i^{k_i}dy$ and $g\circ\pi=h\prod y_i^{b_i}$, where the $p_1,p_2,h$ are units. We can do this because of the assumption $g\mid f$ (see \cite[Remark 3.4]{DirksMustata}). Let $v=h\prod y_i^{k_i+b_i}(\prod y_i^{a_i})^s$, set $B_V(s)=b_v(-s-1)$ and let $B(s)$ be the least common multiple of $B_V(s)$ over all $V$. Then the proof of \cite[Theorem 1.2]{DirksMustata} gives an integer $N\geq0$ so that
\begin{equation}\label{eqn: buestimate}b_{u^\ast}(s)\mid\prod_{j=0}^{N}B(s-j).\end{equation}
Using \cite[Lemma 2.6]{DirksMustata} to remove $h$ and then using \cite[Lemma 2.8(ii)]{DirksMustata}, we have
\begin{equation*} b_v(s)\mid\prod_i\prod_{j=1}^{a_i}\left(s+\frac{k_i+b_i+j}{a_i}\right). \end{equation*}
It follows that every root of $B_V(s)$ is $>-1$ and
\[\mult_{s=0}B(s)=\max_V\mult_{s=0}B_V(s)\leq\max_{y\in Y}\#\{i\mid y\in E_i,\ a_i\geq k_i+b_i+1\}.\]
Hence $B(-j)\neq0$ for every integer $j\geq1$, and so $\mult_{s=0}b_{u^\ast}(s)\leq\mult_{s=0}B(s)$ by \eqref{eqn: buestimate}. Using $b_u(s)=b_{u^\ast}(-s-1)$ gives \eqref{eqn:mult-1bound}.
\end{proof}

\begin{remark}
Let $f$ define a reduced, central, and essential hyperplane arrangement in $\C^n$. Then, for every $\alpha\in\mathbb Q$ with $\alpha\geq 1$, one can show that
\[
\operatorname{ord}_{s=-\alpha} Z_1=\omega_{1,\alpha}=\nu_{1,\alpha}=n_\alpha, \quad \operatorname{ord}_{s=-k} Z_1=n, \quad \forall k\in\mathbb Z_{\geq 1}.\]
Indeed, by \cite[Theorem~1]{Saitohyperplanearrangement}, all roots of $b_{f^s}(s)$ lie in $(-2,0)$, and $-1$ is a root of multiplicity $n$. Fix $\alpha\in\mathbb Q$ with $\alpha\geq1$. First, if $b_{f^s}(s)$ has no roots in
$-\alpha+\mathbb Z$, then $n_\alpha=0$ by Corollary \ref{cor: bound the nilpotency index by multiplicity}. Since by Lemma~\ref{lemma: pole order bound by nu}, $\operatorname{ord}_{s=-\alpha}Z_1\leq\nu_{1,\alpha}\leq n_\alpha$ and $\omega_{1,\alpha}\leq\nu_{1,\alpha}$, all four quantities vanish. Otherwise, let $-\beta$ be the largest root of $b_{f^s}(s)$ in $-\alpha+\mathbb Z$. Since all roots lie in $(-2,0)$ and we assume $\alpha\geq 1$, we must have $\alpha=\beta+\ell$ for $\ell\in \Z_{\geq 0}$. Applying Theorem~\ref{thm: stabilization of nu and omega}
and Corollary~\ref{cor: bound of Zm achieved}, we obtain
\[
\operatorname{ord}_{s=-\alpha}Z_1=\omega_{1,\alpha}=\nu_{1,\alpha}=n_\beta=n_\alpha.
\]
Finally, using Remark~\ref{remark:nu11is mu}, we obtain $n_1=\nu_{1,1}=\operatorname{mult}_{s=-1}b_{f^s}(s)=n$, which gives the second equality.

On the other hand, if $\alpha<1$, this can be quite subtle. For example, if $f$ defines a central essential and indecomposable arrangement in $\C^n$ and $\deg f=d$, one can have $\operatorname{ord}_{s=-n/d}Z_1=0=\omega_{1,n/d}<1=\mult_{s=-n/d}b_{f^s}(s)= \nu_{1,n/d}$. See \cite[Remark 6.9]{DY26} for an example.
\end{remark}

\subsubsection{Irreducible isotypic components and spherical varieties}

In this section, we recast some of the results from \cite{LY25April} in the framework of this article. As in \emph{loc. cit.}, let $X$ be a representation of a complex connected reductive group $G$, and let $f$ be a $G$-semi-invariant function. Let $\cS$ be a simple equivariant holonomic $\sD$-module with full support on $X$.

\medskip

Suppose that $\lambda$ is an irreducible representation of $G$ such that the isotypic component $\cM_\lambda$ is irreducible. Let $m\in \cM_\lambda$ be a non-zero element; in this case, $b_m(s)$ satisfies a special equation \cite[Equation (2)]{LY25April}. In this setup, $b_m(s)$ determines $\nu_{m, \alpha}$ for all $\alpha$, as seen in \cite[Equation (5)]{LY25April}. In fact, the right-hand inequality in Proposition \ref{prop:nurecursivebound} is an equality for all $\alpha$. Furthermore, \cite[Theorem 1.2]{LY25April} can be rephrased as: 
\begin{equation}\label{eq:omega=nu}
    \omega_{m, \alpha} = \nu_{m,\alpha}.
\end{equation}

Assuming $\cS=\cO_X$, the polynomial $b_m(s)$ also determines the pole orders of $Z_m$:

\begin{prop}
In the setup above, $\underset{s=-\alpha}{\mathrm{ord}} Z_m=\nu_{m, \alpha}$.
\end{prop}

\begin{proof}
    Because the inequality $\ord Z_m\leq \nu_{m, \alpha}$ holds in general (see Lemma \ref{lemma: pole order bound by nu}), it is sufficient to find a test function $\varphi$ for which $\ord Z_m(\varphi ; s) = \nu_{m, \alpha}$. 
    
    We pick $\varphi= \exp(-|x|^2)$, as in \cite[Theorem 6.3.1]{igusazetabook}, and set $Z(s):= Z_m(\exp(-|x|^2) ; s)$. As in the proof of \emph{loc. cit.}, we obtain a functional equation (up to a constant):
    \begin{equation}\label{eq:funceq}
        b_m(s) \cdot Z(s) = Z(s+1).
    \end{equation}
    We factor $b_m(s)$ as $\prod (s+ \lambda_i)$, with $\lambda_i \in \Q$. We claim that there exists a non-zero constant $c_0 \in \mathbb{R}^*$ such that
    \begin{equation}\label{eq:Zexp}
        Z(s) = c_0 \cdot \prod_{i} \Gamma(s+\lambda_i).
    \end{equation}
    Consider the function $C(s) = Z(s) / \prod_i\Gamma(s+\lambda_i)$. The functional equation (\ref{eq:funceq}) implies that $C(s+1)=C(s)$. Since $Z(s)$, and thus $C(s)$, is holomorphic when $\Re(s) \gg 0$, the periodicity of $C(s)$ implies that it is holomorphic on all of $\C$. Now, the same argument as in the proof of \cite[Theorem 6.3.1]{igusazetabook}, mutatis mutandis, applies to show that $C(s)$ must be a non-zero constant $c_0$.

    The formula for $\nu_{m,\alpha}$ from \cite[Equation (5)]{LY25April}, together with (\ref{eq:Zexp}), implies that $\ord Z(s) = \nu_{m,\alpha}$ for all $\alpha \in \Q$.
\end{proof}

Note that the result above, together with Proposition \ref{prop:w=v}, provides a new proof of (\ref{eq:omega=nu}) based on zeta functions in the case where $\cS=\cO_X$. Additionally, it implies that Question \ref{zeta} has a positive answer in this setup (in a strong form, since equality holds).

\medskip

\cite[Theorem 1.4]{LY25April} states that Corollary \ref{cor:hodgedegree} is sharp in this setup (recall that $d_{m,\alpha} = \deg p_{m,\alpha}(s)$):
\begin{equation}\label{eq:hodge}
    \mbox{The Hodge level of } m\cdot f^{-\alpha} \mbox{ is precisely } d_{m,\alpha}.
\end{equation}

\bigskip

Now, assume $X$ is a smooth affine spherical variety under the action of $G$, in which case a simple equivariant holonomic $\sD$-module $\cS$ is automatically regular holonomic (see \cite{LY25April}). Then, all isotypic components of $\cM$ are irreducible \cite[Lemma 3.6]{LY25April}; thus, (\ref{eq:omega=nu}) and (\ref{eq:hodge}) yield positive answers to Questions \ref{que: weight filtration on Mf-alpha in terms of nu} and \ref{hodge}. In fact, stronger versions hold:

\begin{corollary}
    Assume that $X$ is a smooth affine spherical variety. Then, for all $\alpha$ and $\ell$, we have
    \[ W_{q+\ell}(\cM \cdot f^{-\alpha}) = \cN_\ell(\cM \cdot f^{-\alpha}). \]
    Furthermore, if $\cS= \cO_X$, then
    \[ F_\ell(\cM \cdot f^{-\alpha}) = \cP_\ell(\cM \cdot f^{-\alpha}).\]
\end{corollary}

\subsubsection{Minimal exponent}

In this section we compute explicitly some of the invariants studied in this paper in the case when $\cS=\cO_X$ and $\alpha$ is the minimal exponent $\tilde{\alpha}_f$ of $f$, based on the results in \cite{DLY} and \cite{DY24}.

\begin{prop}\label{prop:minexp}
Set $\ell:= \lceil \tilde{\alpha}_f -1 \rceil$. Then:
\begin{itemize}
    \item[(a)] $p_{1,\tilde{\alpha}_f}(s) = \prod_{i=1}^\ell (s+i)$,
    \item[(b)] We have $\omega_{1,\tilde{\alpha}_f}=\nu_{1, \tilde{\alpha}_f} = \underset{s=-\tilde{\alpha}_f}{\operatorname{ord} Z_{1}} = \begin{cases}
        \mult_{s=-\tilde{\alpha}_f} b_1(s), \mbox{ if  } \tilde{\alpha}_f \notin \Z \mbox{ or } \tilde{\alpha}_f =1,\\
        1+\mult_{s=-\tilde{\alpha}_f} b_1(s), \mbox{ if  } \tilde{\alpha}_f \in \Z \setminus\{1\}.
    \end{cases}$
    \item[(c)] The Hodge level of $f^{-\tilde{\alpha}_f}$ in $(\cO_X)_f\cdot f^{-\tilde{\alpha}_f}$ is $\ell= d_{1, \tilde{\alpha}_f}$, i.e. $f^{-\tilde{\alpha}_f}\in F_{\ell}\setminus F_{\ell-1}$.
\end{itemize}
\end{prop}

\begin{proof}
    Take any $\alpha \in (0, \tilde{\alpha}_f)$. Proposition \ref{prop:nurecursivebound} implies that $\nu_{1, \alpha} = 0$, if $\alpha \notin \Z$, while $\nu_{1,\alpha} = 1$, if $\alpha \in \Z$. Then Corollary \ref{cor:explicitpfunction} gives part (a).

    Now we compute $\nu_{1,\tilde{\alpha}_f}$ in part (b). If $\tilde{\alpha}_f \notin \Z_{\geq 2}$ then the claim follows readily from Proposition \ref{prop:nurecursivebound}. Thus, assume $\tilde{\alpha}_f \in \Z_{\geq 2}$. In this case,  Proposition \ref{prop:nurecursivebound} gives $\mu_{1,\tilde{\alpha}_f} \leq \nu_{1,\tilde{\alpha}_f} \leq \mu_{1,\tilde{\alpha}_f} +1$. On the other hand, Proposition \ref{prop:strictnumu} gives $\nu_{1,\tilde{\alpha}_f} > \mu_{1,\tilde{\alpha}_f}$, so that we must have $\nu_{1,\tilde{\alpha}_f}=\mu_{1,\tilde{\alpha}_f}+1$.    By \cite[Theorem 1.1]{DLY}, we have $\nu_{1, \tilde{\alpha}_f} = \mathrm{ord}_{s=-\tilde{\alpha}_f}Z_{1}$. Proposition \ref{prop:w=v} then implies that $\omega_{1,\tilde{\alpha}_f}= \nu_{1, \tilde{\alpha}_f}$. 
    
    Finally, we prove (c). By part (a) and Corollary \ref{cor:hodgedegree}, the Hodge level of $f^{-\mef}$ is at most $\ell$. It remains to prove the reverse inequality. First assume that $\mef\notin\Z$, and write $\mef=\ell+\beta$ with $\beta\in(0,1)$. Since $\beta>0$, we have $F_\bullet V^\beta\iota_+(\cO_X)_f=F_\bullet V^\beta\iota_+\cO_X$, as in the proof of Proposition \ref{prop: pmalpha for O}. So by \cite[Lemma 5.5]{DLY}, we have
\begin{equation}\label{eqn:lemma5.5dly}1\otimes\partial_t^\ell\in F_{\ell+1}V^\beta\iota_+(\cO_X)_f, \quad 0\neq[1\otimes\partial_t^\ell]\in F_{\ell+1}\gr_V^\beta\iota_+(\cO_X)_f,\qquad F_\ell\gr_V^\beta\iota_+(\cO_X)_f=0.\end{equation}
Under the isomorphism \eqref{eqn: Malgrange isomorphism}, evaluation at $s=-\beta$ sends $1\otimes\partial_t^\ell$ to
\[\prod_{i=0}^{\ell-1}(\beta+i)\,f^{-\tilde{\alpha}_f}=c\cdot f^{-\tilde{\alpha}_f}, \quad \textrm{with a constant $c\neq 0$}.\]
Recall we have a short exact sequence from \cite[Corollary 1.2]{DY24}
\[ 0\to F_{q-1}V^{\beta}\iota_+(\cO_X)_f\xrightarrow{s+\beta} F_{q}V^{\beta}\iota_{+}(\cO_X)_f\xrightarrow{ev_{s=-\beta}}  F_{q-1}((\cO_X)_f\cdot f^{-\beta})\to 0, \quad \forall q\geq 1.\]
Suppose by contradiction that $f^{-\tilde{\alpha}_f}=f^{-\ell}\cdot f^{-\beta}\in F_{\ell-1}((\cO_X)_f\cdot f^{-\beta})$, then applying $q=\ell$ above we can find an element $u\in F_{\ell}V^{\beta}\iota_{+}(\cO_X)_f$ such that $ev_{s=-\beta}(u)=f^{-\tilde{\alpha}_f}$. Because $ev_{s=-\beta}(\frac{1}{c}1\otimes\partial_t^\ell)=f^{-\tilde{\alpha}_f}$, and $1\otimes\partial_t^\ell\in F_{\ell+1}V^\beta\iota_+(\cO_X)_f$, by above we must have 
\[ \frac{1}{c}1\otimes\partial_t^\ell-u\in \ker(ev_{s=-\beta})\cap F_{\ell+1}V^{\beta}\iota_+(\cO_X)_f=(s+\beta)F_{\ell}V^{\beta}\iota_+(\cO_X)_f.\]
In other words,
\[1\otimes\partial_t^\ell\in F_\ell V^\beta\iota_+(\cO_X)_f+(s+\beta)F_{\ell} V^\beta\iota_+(\cO_X)_f\subseteq V^{>\beta}\iota_+(\cO_X)_f,\]
where the last containment follows from \eqref{eqn:lemma5.5dly}, i.e. $F_\ell\gr_V^\beta\iota_+(\cO_X)_f=0$. This contradicts the non-vanishing of $[1\otimes\partial_t^\ell]$ in $\gr^{\beta}_V\iota_{+}(\cO_X)_f$. Hence $f^{-\tilde{\alpha}_f}\not\in F_{\ell-1}((\cO_X)_f\cdot f^{-\beta})$. It follows that $f^{-\mef}$ has Hodge level $\ell$.

Now assume that $\mef\in\Z$. If $\mef=1$, the claim is immediate. Otherwise, write $\mef=\ell+1$ with $\ell\geq1$. By \cite{Saito09}, we have a comparison of Hodge filtration and pole order filtration (see also \cite{MPQHodgeideal}):
\[F_{\ell-1}((\cO_X)_f\cdot f^{-1})\subseteq f^{-(\ell-1)}\cO_X\cdot f^{-1}.\]
Therefore $f^{-\ell}\cdot f^{-1}\notin F_{\ell-1}((\cO_X)_f\cdot f^{-1})$. This again shows that the Hodge level of $f^{-\mef}$ is precisely $\ell$.
\end{proof}

\bibliographystyle{alpha}
\bibliography{references}
\vspace{\baselineskip}

\footnotesize{
\textsc{Department of Mathematics, University of Oklahoma, 660 Parrington Oval, Norman, OK 73019, United States} \\
\indent \textit{E-mail address:} \href{mailto:lorincz@ou.edu}{lorincz@ou.edu}

\vspace{\baselineskip}

\textsc{Department of Mathematics, University of Kansas, 1450 Jayhawk Blvd, Lawrence, KS 66045, United States} \\
\indent \textit{E-mail address:} \href{mailto:ruijie.yang@ku.edu}{ruijie.yang@ku.edu} 
}
\end{document}